\documentclass{amsart}
\usepackage{amsmath,amsthm,amsfonts,amssymb}
\usepackage{hyperref}
\usepackage[all]{xy}
\usepackage{graphicx}
\usepackage{amscd}
\usepackage{pb-diagram}
\usepackage{pstricks}

\numberwithin{equation}{section}

\newtheorem{theorem}{Theorem}[section]
\newtheorem{definition}[theorem]{Definition}
\newtheorem{corollary}[theorem]{Corollary}

\newtheorem{lemma}[theorem]{Lemma}
\newtheorem{remark}[theorem]{Remark} 
\newtheorem{prop}[theorem]{Proposition}

\newtheorem{conjecture}[theorem]{Conjecture}

\newtheorem{assumption}[theorem]{Assumption}
\newtheorem{def-thm}[theorem]{Definition-Theorem}

\newcommand{\CM}{\mathcal M}

\newcommand{\R}{\mathbb{R}}
\newcommand{\RR}{\mathbb{R}}
\newcommand{\C}{\mathbb{C}}
\newcommand{\CC}{\mathbb{C}}
\newcommand{\Z}{\mathbb{Z}}
\newcommand{\CP}{\mathbb{C}P}
\newcommand{\ZZ}{\mathbb{Z}}
\newcommand{\Q}{\mathbb{Q}}

\newcommand{\bx}{\boldsymbol{x}}

\newcommand{\dbar}{\overline{\partial}}

\newcommand{\CL}{\mathcal{L}}
\newcommand{\CA}{\mathcal{A}}

\newcommand{\WT}[1]{\widetilde{#1}}
\newcommand{\WH}[1]{\widehat{#1}}
\newcommand{\OL}[1]{\overline{#1}}

\newcommand{\UL}[1]{\underline{#1}}

\newcommand{\Hom}{{\rm Hom}}

\newcommand{\CR}{\textit{crit}}
\newcommand{\AI}{A_\infty}
\newcommand{\KI}{{\rm{Ker}} I_{\eta}}
\newcommand{\KIO}{{\rm{Ker}} I_{\eta}^{orb}}

\begin{document}
\author[Cho]{Cheol-Hyun Cho}
\address{Department of Mathematical Sciences, Research institute of Mathematics\\ Seoul National University\\ San 56-1, 
Shinrimdong\\ Gwanakgu \\Seoul 47907\\ Korea}
\email{chocheol@snu.ac.kr}
\author[Hong]{Hansol Hong}
\address{Department of Mathematical Sciences \\ Seoul National University\\ San 56-1, 
Shinrimdong\\ Gwanakgu \\Seoul 47907\\ Korea}
\email{hansol84@snu.ac.kr}
\title{Finite Group actions on Lagrangian Floer theory}
\begin{abstract}
We construct finite group actions on Lagrangian Floer theory when symplectic manifolds have finite group actions and Lagrangian submanifolds have  induced group actions. We first define finite group actions on Novikov-Morse theory. We introduce the notion of a {\em spin profile } as an obstruction class of extending the group action on Lagrangian submanifold to the one on its spin structure, which is a group cohomology class in $H^2(G;\Z/2)$.  For a class of Lagrangian submanifolds which have the same spin profiles, we define a finite group action on their Fukaya category.
In consequence, we obtain the $s$-equivariant Fukaya category as
well as the $s$-orbifolded Fukaya category for each group cohomology class $s$.
We also develop a version with $G$-equivariant bundles on Lagrangian submanifolds, and
explain how character group of $G$ acts on the theory.
 As an application, we define an orbifolded Fukaya-Seidel category of
a $G$-invariant Lefschetz fibration, and also discuss homological mirror symmetry conjectures with group actions.

\end{abstract}
\maketitle

\tableofcontents

\section{Introduction}
A group action is one of the most important tool in the study of mathematical objects. Finite group actions in symplectic geometry appear
in various contexts, such as mirror symmetry, singularity theory, etc. In this paper, we will construct finite group actions on Lagrangian Floer theory.
Lagrangian Floer homology (developed by Floer\cite{Fl}, Oh\cite{Oh1}, Fukaya-Oh-Ohta-Ono \cite{FOOO}) has been a powerful tool, which 
revolutionized the study of symplectic topology. Furthermore, its categorical versions, Fukaya category (Fukaya \cite{Fuk1}) and Fukaya-Seidel
category (Seidel \cite{Se}, \cite{Se2}), are the main players
of homological mirror symmetry as envisioned by Kontsevich \cite{K}.

When a symplectic manifold as well as a Lagrangian submanifold admits a finite group action, it is  indeed quite natural to expect that such an action
exists on Lagrangian Floer theory. But there are several technical difficulties to define such an action as we will explain below. 
In the case of $\Z/2$-actions on exact symplectic manifolds, Seidel has defined 
$\Z/2$-equivariant Fukaya categories \cite{Se}. Another known case is when a group
acts freely on the set of connected components of a given exact Lagrangian submanifold in an exact symplectic manifold. In such a case, it is relatively easy to define a group action, since one can use pull-back data (such as spin structure, perturbation data, etc) to make Floer theory compatible with the given group action.
Such construction was proposed by Seidel \cite{Se2} in his study of mirror symmetry of quartic surfaces in an exact setting.

Let us now describe our main results. We will first explain the exact case, and later the case of general closed symplectic manifolds.
For an exact symplectic manifold $(M,\omega)$, consider a $G$-invariant exact Lagrangian submanifold $L$, which means that a $G$-action preserves the Lagrangian submanifold.
Suppose that $L$ is spin and that its spin structure is $G$-invariant. We show that the $G$-action on $L$ gives rise to
a group cohomology class in $H^2(G,\Z/2)$, ``the spin profile of $L$", which is the obstruction to extending the group action
to the spin bundle of $L$. The main observation is that if two Lagrangian submanifolds have the same spin profile, then it is possible to define a group
action on the Lagrangian Floer cohomology of the pair in the exact case.

Thus for each cohomology class $s  \in H^2(G,\Z/2)$,
an $s$-equivariant brane is defined as a $G$-invariant exact Lagrangian submanifold
whose associated spin profile is $s$, with a $G$-invariant grading.
\begin{theorem}[see Theorem \ref{thm:exeqfuk}]
Let $(M,\omega)$ be an exact symplectic manifold. For $s \in H^2(G,\Z/2)$, an {\em $s$-equivariant
Fukaya category} can be defined as follows.
An object of $s$-equivariant Fukaya category is given by an $s$-equivariant brane.
Morphism between two $s$-equivariant branes $L_0^\sharp, L_1^\sharp$ are given by $CF^*(L_0^\sharp, L_1^\sharp)$,
which has a $G$-action.
There are $\AI$-operations
$$m_k : CF^*(L_0^\sharp, L_1^\sharp) \times \cdots \times CF^*(L_{k-1}^\sharp, L_k^\sharp) \to CF^*(L_0^\sharp, L_k^\sharp),\quad k=1,2, \cdots$$ 
which are compatible with the $G$-action. i.e. for $k=1,2, \cdots$
\begin{equation}\label{strictGAinfty1}
m_k(gx_1,\cdots, gx_k) = gm_k(x_1,\cdots,x_k).
\end{equation}
\end{theorem}
To be more precise, we should equip each $s$-equivariant brane with a $G$-equivariant flat complex vector bundle as in Section \ref{sec: orbibundle}.
The $s$-orbifolded Fukaya category is defined to have the same set of objects as the $s$-equivariant Fukaya category, but we take $G$-invariant parts of morphisms of the $s$-equivariant category. Then, the $s$-orbifolded Fukaya category naturally inherits an $\AI$-structure from the $s$-equivariant category from \eqref{strictGAinfty1}.

As an application, we define both equivariant and orbifolded Fukaya-Seidel categories of a $G$-invariant Lefschetz fibration.
Namely, for  each $s  \in H^2(G,\Z/2)$, we  define {\em $s$-equivariant Fukaya-Seidel category of $\pi :M \to \C$} (See Definition \ref{def:FS}, \ref{def:FS2}).  By taking $G$-invariant
parts on morphisms, and the induced $\AI$-structure, we obtain the {\em $s$-orbifolded Fukaya-Seidel} category of $(\pi,G)$.

Let us turn our discussion to the case of a general closed symplectic manifold (which is not exact) with an effective action of a finite group $G$. Consider a Lagrangian submanifold $L_i$ ($i=0,1$), which has an induced $G$-action. Suppose that $L_i$ has a $G$-invariant spin structure, with a spin profile $s_i \in H^2(G;\Z/2)$
for $i=0,1$. Fix a base path $l_0:[0,1] \to M$ such that $l_0(i) \in L_i$ for $i=0,1$.
Let $G_\alpha$ be the energy zero subgroup $G_{\alpha}$ of $l_0$ as in Definition \ref{def:galpha}. 

\begin{theorem}[see Definition \ref{def:aact}, Proposition \ref{prop:galphacf}, \ref{prop:12.4}]
If $s_0=s_1$, then the Lagrangian Floer filtered cochain complex of $L_0, L_1$ admits a $G_\alpha$-action.
More generally, filtered $\AI$-bimodule of the pair $(L_0, L_1)$ admits a natural $G_\alpha$-action.
\end{theorem}

\begin{prop}[see Section \ref{sec:12}]
Suppose that $L_0$ is connected 
and let $L_1=\phi_H^1 (L_0)$.
Then $G_\alpha = G$, and we have  a $G$-equivariant filtered $\AI$-bimodule quasi-isomorphism
 $$\Omega(L) \widehat{\otimes} \Lambda_{nov} \to CF_{R,l_0}^\ast(L,L_1) \otimes_{\Lambda (L, L_1;l_0)}\Lambda_{nov}.$$
\end{prop}
The method developed in this paper,  together with the work of Fukaya
in \cite{Fuk} should define equivariant Fukaya categories in the general case.

More detailed introduction and description are now in order.
\subsection{Spin profiles} 
In this paper, we introduce the notion of a {\em spin profile} of a Lagrangian submanifold with a $G$-action.  Such a notion is deeply related to the orientation of the moduli space of  $J$-holomorphic discs and
the $G$-action on it.  Recall that if a  Lagrangian submanifold is spin, a choice of spin structure
determines canonical orientations of the moduli spaces of $J$-holomorphic discs (see \cite{FOOO}).
Hence, we need to analyze the relation between the group action on $L$ and that of its spin structure.

We assume that the related spin structure is $G$-invariant, which means that for a principal spin
bundle $P \to L$, the pull-back bundle $g^*P$ via the action of $g \in G$ is isomorphic to the original bundle $P$.
In fact, the precise isomorphism between $g^*P$ and $P$ (as in $G$-equivariant sheaves),
which can be considered as a lift of the $g$-action to the one on spin bundle, will be necessary to
determine  the sign of the $g$-action on the Floer complex. In general, these isomorphisms are not necessarily
compatible with each other, and the $G$-action on $L$ does {\em not} lift to the one on the spin bundle.
Hence, one can not expect that
any orientation preserving action of a symplectic manifold gives rise to an action on its Fukaya category.

A {\em spin profile} is defined as the group cohomology class
in $H^2(G,\Z/2)$ which measures the obstruction to extending the action of $G$ on a manifold to the one on its spin structure. We show that if spin profiles of two Lagrangian submanifolds coincide, then we are able to define a group action on Lagrangian Floer theory of a pair (we need to develop $G$-Novikov theory additionally for non-exact symplectic manifolds). The main point is that the exact signs from failures of group actions on spin bundles cancel out for Lagrangians with the same spin profiles.
This means that the equivariant Fukaya category is defined for each choice of spin profiles in $H^2(G,\Z/2)$.
Construction of the $\Z/2$-equivariant Fukaya category by Seidel \cite{Se}, in fact, corresponds to the non-trivial class in $H^2(\Z/2,\Z/2)$.
We explain several properties of spin profiles, and also the change of spin profiles under the change of spin structures.


\subsection{Equivariant Fukaya categories for exact cases and $G$-Lefschetz fibrations}
Let $(M,\omega)$ be an exact symplectic manifold  with an effective action of a finite group $G$. We fix a class $s \in H^2(G,\Z/2)$, and consider $s$-equivariant branes as explained above. It may happen that some of $g$-action send $L$ to $g(L)$ which is disjoint from $L$, in which case we can take $\sqcup_{g \in G, g(L) \neq L} g(L)$ to be the new $G$-invariant Lagrangian submanifold. We do not assume that $L$ is connected in the exact case.

In general, the equivariant Fukaya category of $M$ should also have $G$-equivariant Lagrangian immersions to $M$.
In the sequel of this paper \cite{CH2}, we will investigate extensively such $G$-equivariant Lagrangian immersions, and
explore the relationship  of $G \times G$-equivariant immersions into $M \times M$ with  Chen-Ruan orbifold cohomology of $[M/G]$. 
In fact, one should consider a more general notion of orbifold embeddings \cite{ALR}, but they were shown to be equivalent to  $G$-equivariant immersions for
global quotients in \cite{CHS}, and hence it is enough to study the latter.

There is another technical issue, which are equivariant transversalities of the related moduli spaces of holomorphic discs.
Namely, we would like to choose perturbation data equivariantly, so that we have
strict $G$-equivariant $\AI$-operations.
In the exact case, we follow Seidel \cite{Se2} to achieve it with aid of an algebraic machinery.
In the general
case later, Kuranishi structure (as in \cite{FO}, \cite{FOOO}, \cite{Fuk}) will be used, which is in fact a device to achieve this kind of equivariant transversality

We construct a version of Donaldson-Fukaya categories ($C$ in Definition \ref{Donald}),
and produce a $G$-action on it from the weak $G$-action on the related Floer complex (see \cite{CH}
for the weak $G$-actions on  Morse complexes). Then, the homological perturbation type lemma of Seidel (Lemma 4.3 of \cite{Se2}) enables
us to define the $s$-equivariant Fukaya category for an exact symplectic manifold. By taking
$G$-invariant parts on morphism spaces, the $s$-orbifolded Fukaya category of $[M/G]$ is obtained. This should be regarded as a first approximation of orbifold theory, 
as we should also add bulk deformation by twisted sectors as in \cite{CP}.

 As an application, we define {\em the directed equivariant (and the orbifolded) Fukaya-Seidel category} of
 a $G$-invariant Lefschetz fibration  $\pi:M \to \C$, where $M$ is equipped with a $G$-action, and $\pi$ is $G$-invariant. 
 In mirror symmetry, orbifolds and orbifolded Landau-Ginzburg models appear
 naturally, and have been actively investigated. Our work provides the necessary framework to
 state homological mirror symmetry in these cases by defining appropriate orbifold Fukaya categories.
 We only deal with $G$-invariant Lefschetz fibrations in this paper and leave the case of an orbifolded $LG$-model
 which is not Lefschetz fibration and of a $G$-equivariant $LG$-model to the future investigation. 
 
\subsection{$G$-Novikov theory}
For a closed symplectic manifold which is not exact, the Lagrangian Floer cohomology is the Morse-Novikov homology on path spaces between two Lagrangian submanifolds. Hence, we need to develop $G$-Novikov theory.

We first recall $G$-Morse theory from our previous work \cite{CH}.
 Given a Morse-Smale function $f:M \to \R$ on a compact manifold $M$, one can define the Morse-Smale-Witten complex (Morse complex for short) by counting gradient flow lines between critical points of $f$. If $f$ is $G$-invariant,  we studied how to define a finite group action on the  Morse complex, whose homology of the $G$-invariant part
is isomorphic to the singular homology of  the quotient orbifold.  One of the main observation of \cite{CH} is that one should consider not only the action on critical points, but also on the orientation spaces (defined from their unstable manifolds).

Now,  Morse-Novikov theory (Novikov theory for short) generalizes Morse theory to the
case of a general closed Morse one form $\eta$, instead of $df$. Novikov theory is a kind of Morse theory associated to a Morse function on a certain Novikov covering (so that the pull-back of the  given closed one form becomes exact), together with the action of the group ring of its deck transformation group.

In this paper, we define finite group actions on Novikov theory.
The main difficulty is that the base point chosen for  Novikov theory is {\em not} preserved by the group action.
We resolve this issue by introducing the subgroup $G_\eta$ of energy zero elements of $G$, where $g \in G$ is said to be energy zero if there is an energy zero  path (with respect to the action functional)
from the base point to its $g$-action image. And we will construct an action of energy zero subgroup on the Novikov complex using this energy zero path. Together with an appropriate orbifold Novikov ring coefficient, taking $G_\eta$-invariant part defines $G$-Novikov theory

\subsection{Group actions in general Lagrangian Floer theory}
By using spin profiles and $G$-Novikov theory, we can define group actions on Lagrangian Floer cohomology groups for non-exact cases as follows.
Let $(M,\omega)$ be a closed symplectic manifold with an effective $G$-action. 
Fix $s \in H^2(G,\Z/2)$ and consider a Lagrangian submanifold $L_i$ (for $i=0,1$) whose spin profile is $s$.
For a base path $l_0$ between $L_0$ and $L_1$, take an energy zero surface $w_g$ connecting $l_0$ and $g(l_0)$ for
$g$ in the energy zero subgroup $G_{\alpha}$ (Definition \ref{def:galpha}) as an analogue of $G$-Novikov theory.
We assume that $G = G_\alpha$ for simplicity. For each generator of the Floer complex of a pair $(L_0, L_1)$, one can consider its orientation space,
which is the determinant bundle of the associated Cauchy-Riemann operator. This construction already appeared in \cite[3.7.5]{FOOO}, and Seidel also used orientation spaces in exact cases. Then, the actual group action is defined to
be an action on orientation spaces, where the action image is obtained by gluing the determinant bundles of $w_g$ (more precisely, $(-1)^{sp(w_g)}w_g$)
to the naive $g$-action image of a generator of the Floer complex (see Definition \ref{def:aact}). Here the Lagrangian bundle data along
the boundary $\partial w_g$ will be given a specific spin structure, and $sp(w_g)$ is defined to be 0 (resp. 1) if such spin structure
is trivial (resp. non-trivial). 
Then, the following identity from Proposition \ref{prop:compadm}
$$sp(w_g) \cdot sp(w_h) = (-1)^{\textnormal{spf}_{L_0}(g,h)+\textnormal{spf}_{L_1}(g,h)}sp(w_{gh}),$$
guarantees the cocycle condition of the group action, provided that the spin profiles of $L_0$ and $L_1$ agree with each other.

In this way, we get a well-defined action of $G$ on the Floer cochain complex $CF_{R,l_0} (L_0,L_1)$ between $L_0$ and $L_1$. With equivariant Kuranishi perturbations, we show that the Floer differential  as well as the $\AI$-bimodule structure on $CF_{R,l_0} (L_0,L_1)$ is compatible with the $G$-action. Therefore, the bimodule structure descends to the $G$-invariant part of $CF_{R,l_0} (L_0,L_1)$, which should be interpreted as a morphism space between $[L_0 /G]$ and $[L_1/G]$ in the Fukaya category of the orbifold $[M/G]$.

Invariance of the Floer cohomologies under $G$-invariant Hamiltonians can be proven as well. That is, if $L_1$ is the Hamiltonian deformation of $L_0$ by a $G$-invariant Hamiltonian function on $M$, then the assumption $G_\alpha=G$ always holds and moreover, $CF(L_0,L_0)$ and $CF(L_0, L_1)$ are $G$-equivariantly isomorphic as $CF(L_0,L_0)$-modules.

\subsection{Mirror symmetry}
For the purpose of mirror symmetry, we need to consider complex flat line bundles on Lagrangian submanifolds. With $G$-action, $G$-equivariant structures of such bundles should be also taken into account. 
The constructions in this paper can be adapted to the cases  with $G$-equivariant line bundles by following the standard constructions.
Such a $G$-equivariant bundle may be considered as an orbi-bundle on the orbi-Lagrangian $[L/G]$.
Hence an object of an orbifolded category $\mathcal{F}uk_G (X)$ is given by a  $G$-invariant Lagrangian (or a $G$-equivariant immersion) equipped with  a $G$-equivariant line bundle on it. Here morphisms of $\mathcal{F}uk_G^s (X)$ are obtained by taking the $G$-invariant part of morphisms in the $G$-equivariant category.

Recall that $B$-branes of mirror symmetry are given by  derived categories of coherent sheaves. When there is a group action, it is natural to work with $G$-equivariant sheaves. 
\begin{remark}
The standard naming convention for equivariant sheaves are unfortunately different from ours. 
Namely, what we call an orbifolded category (taking $G$-invariant parts on morphisms of an equivariant category), is called an equivariant
category for the case of sheaves.
\end{remark}
We will see later that if $X$ is given by a Lagrangian torus fiber bundle equipped with an action of $G$  compatible with bundle structure, then its dual torus bundle $Y$ naturally admits a $G$-action. In general, if $Y$ is given as a SYZ (Strominger-Yau-Zaslow) mirror of $X$, we conjecture that
the (derived) orbifold category $D\mathcal{F}uk_G (X)$ is equivalent the derived category of $G$-equivariant sheaves on $Y$. 
When the mirror of $X$ is a Landau-Ginzburg model $W : Y \to \CC$, it turns out that $W$ is invariant under the $G$-action on $Y$. We conjecture that homological mirror symmetry is the equivalence between the $G$-equivariant matrix factorization category of $W$ and orbifolded Fukaya category $\mathcal{F}uk_G (X)$. 
(see Section \ref{MirrorSymm} for the precise statements including spin profile data)

Character group of $G$ acts on both sides of mirror by twisting the equivariant structures on bundles. Namely, given an orbifold bundle (or $G$-equivariant  bundle) for $\mathcal{F}uk_G(X)$, one can change $G$-equivariant structure by multiplying a character.
We expect that these actions are compatible with homological mirror symmetry.
We will look into toric cases, and present  more conjectures and examples.

 Let us also mention that in a forthcoming work with Siu-Cheong Lau, 
we will analyze mirror symmetry phenomena for $G$-equivariant immersions into finite group quotients, and construct a {\em localized homological mirror  functor} for immersed Lagrangian submanifolds.\\

\subsection*{Notations and Conventions}
We will assume throughout the paper that $G$ is a finite group.
And $G$-action are always assumed to be effective, proper, and orientation preserving, unless stated otherwise.
A submanifold $L \subset M$ preserved by the $G$-action of $M$ is called $G$-invariant. Here, the action of $g \in G$ sends $L$ to $L$, but it may not fix $L$ pointwisely. 
In fact, if $L$ is a Lagrangian submanifold of $M$ and $g$ fixes every point  of $L$, then it is easy to see that $g$ fixes a neighborhood of $L$ in $M$. Since the action on $M$ is assumed to be effective, such $g$ should be the identity element of $G$.

Let $R$ be a field containing $\Q$, and we will use $R = \R$ or $\C$ for Floer theory.
We use the following $R$-normalization convention due to Seidel \cite[Section (12e)]{Se}.
This is because, roughly speaking, we are only interested in the signs, rather than magnitudes in orientations spaces.
For any $1$-dimensional real vector space  $\Theta$,  we denote by
$|\Theta|_R$ the $1$-dimensional $R$-vector space generated by the two orientations of $\Theta$,
with the relation that the sum of two generators is zero. An $R$-linear isomorphism
$\phi:\Theta_1 \to \Theta_2$ induces a $R$-linear isomorphism $|\Phi|_R:|\Theta_1|_R \to
|\Theta_2|_R$, called the $R$-normalization of $\Phi$.
Normalization will be also used in families, turning real line bundles into local coefficient systems with fiber $R$. It can be thought of as passing to the associated bundle via the homomorphism given by
$$\R^\times \to \{\pm 1\} \subset R^{\times}, \qquad  t \to t/|t|.$$

To simplify notations, we denote also by $g\in G$ the diffeomorphism of $M$ induced
by the action of $g$ (sometimes written as $A_g$). For $\phi:M \to N$, we denote by $\phi_\sharp$ the
induced map on the fundamental group $\phi_\sharp : \pi_1(M,x_0) \to \pi_1(N, \phi(x_0))$.
We denote by $\gamma_1 \star \gamma_2$, the concatenation of two paths $\gamma_1$ and $\gamma_2$.\\

\subsection*{Acknowlegements} 
We have benefited very much from Paul Seidel's work \cite{Se}, \cite{Se2}. We also thank him for helpful comments to our work. We are indebted to Kenji Fukaya for Lemma \ref{lem:fukaya} and valuable advices. We are grateful to Kaoru Ono for the kind explanations of orientations in \cite{FOOO}, and to Siu-Cheong Lau for several helpful discussions especially on character group actions. 
We would like to express our gratitude to Yong-Geun Oh,  Kwokwai Chan, Hsian-hua Tseng, Mainak Poddar, Nick Sheridan,  Hiroshige Kajiura, Gyoung-Seog Lee, Hyung-Seok Shin, Sangwook Lee  for their help. 
This work was supported by the National Research Foundation of Korea (NRF) grant funded by the Korea Government (MEST) (No. 2012R1A1A2003117)

\section{$G$-Novikov theory in the easiest case}
In this section, we review Novikov theory by considering  $G$-Novikov theory under the assumption \ref{assum:x_0fixed}. In this case, $G$-Novikov theory is given by an equivariant version of the standard one. Also, we will review $G$-Morse theory from our previous work \cite{CH},
which is needed for the proper definition of  the group action on a Novikov complex.
\begin{assumption}\label{assum:x_0fixed}
There exists a point $x_0 \in M$ such that $g \cdot x_0 = x_0$ for all $g \in G$.
\end{assumption}
This assumption enables us to have a natural $G$-action on the fundamental group $\pi_1(M,x_0)$. We will fix a base point to be $x_0$, and simply write $\pi_1(M)$ for $\pi_1 (M, x_0)$ in this section.
Note that the $G$-action on $\pi_1 (M)$ is compatible with the group operations of $\pi_1 (M)$. i.e. the $G$-action gives a group homomorphism $G \to {\rm{Aut}} (\pi_1 (M))$.

\subsection{Novikov coverings}
 For a closed $1$-form $\eta$ on $M$, we call $\eta$ a {\em Morse $1$-form} if the graph of $\eta$ in $T^\ast M$ is transversal to the zero section. Assume further that $\eta$ is $G$-invariant. $\eta$ gives rise to a homomorphism $I_\eta : \pi_1 (M) \to \RR$,
\begin{equation}
I_\eta (\alpha) = \int_{S^1} \alpha^\ast \eta,
\end{equation}
for $\alpha \in \pi_1 (M)$. 
 Note that $I_\eta$ is $G$-invariant as well, i.e. $I_\eta(g \cdot \alpha) = I_\eta(\alpha)$.

Let  $\pi : \WT{M} \to M$ be the covering space associated to the kernel of $I_\eta$, denoted by ${\rm{Ker}} I_\eta$. (i.e. for the universal covering  $\WT{M}_{uni}$ of $M$, we have
$\WT{M} = \WT{M}_{univ} / {\rm{Ker}} I_\eta$).
We fix a base point $\WT{x}_0$ in $\WT{M}$ such that $\pi(\WT{x}_0) = x_0$.

\begin{lemma}\label{lem:coveraction}
There is an induced $G$-action on $\WT{M}$ and $\pi: \WT{M} \to M$ is $G$-equivariant.
\end{lemma}
\begin{proof}
Since $I_\eta$ is $G$-invariant, ${\rm{Ker}} I_\eta$ is preserved by $g_\sharp:\pi_1(M) \to \pi_1(M)$. Hence for any $g \in G$, 
$$g_\sharp(\pi_{\sharp} \pi_1 (\WT{M}) ) = g_\sharp ({\rm{Ker}} I_\eta ) ={\rm{Ker}} I_\eta$$
Thus, the composition map $ g \circ \pi : \WT{M} \to M$ has a unique lift
$\WT{g}:\WT{M} \to \WT{M}$ such that $(g \circ \pi )(\WT{x}_0) = \WT{x}_0$.
The lemma, then, follows from the standard covering theory.
\end{proof}

The deck transformation group $\Gamma$ of the covering $\WT{M} \to M$ is given by
$\pi_1 (M) / {\rm{Ker}} I_\eta$, and the Novikov ring $\Lambda_{[\eta]}$ is a 
completion of the group algebra of $\Gamma$:
\begin{equation}\label{ring}
\Lambda_{[\eta]} = \left\{ \left. \sum a_i h_i \, \right| \, a_i \in R, \, h_i \in \Gamma,\,\,  |\{i : a_i \neq 0, I_\eta (h_i) <c\} | < \infty \,\, \forall c \in \RR \right\},
\end{equation}
where $R$ is a commutative ring with unity.

\begin{lemma}
The induced $G$-action on $\Gamma= \pi_1 (M) / {\rm{Ker}} I_\eta$ is trivial.
\end{lemma}
\begin{proof}
Since $I_\eta$ is $G$-invariant, $G$ acts on $\Gamma= \pi_1 (M) / {\rm{Ker}} I_\eta$. However, for $\alpha \in \pi_1 (M)$,
\begin{eqnarray*} 
I_{\eta} \left((g \cdot \alpha) \star \alpha_- \right) &=& I_\eta (g \cdot \alpha) - I_\eta (\alpha) \\
&=& I_\eta (\alpha) - I_\eta (\alpha) =0,
\end{eqnarray*}
where $\alpha_- (t) = \alpha (1-t)$.
\end{proof}
We have two group actions on $\WT{M}$, which are the action of deck transformation group $\Gamma$ and the induced action of $G$. The following lemma
will play an important role later.
\begin{lemma}\label{GGamma}
 The $\Gamma$-action and the $G$-action on $\WT{M}$ commute with each other. Namely
 for each $g \in G$, and $h \in \Gamma$, we have
 $$\WT{g} \circ h = h \circ \WT{g}.$$
\end{lemma}
\begin{proof} 
We first work over the universal cover $\WT{M}_{univ}$. Note that we can lift the $G$-action on $M$ to the $G$-action on the universal cover $\WT{M}_{univ}$ similar as in Lemma \ref{lem:coveraction}.  
For a 
deck transformation $h$ of $\WT{M}_{univ}$,  we denote the corresponding loop in $\pi_1(M)$ by
$\alpha_h$.

The action of $g$ takes the loop $\alpha_h$ to the loop $g(\alpha_h)$, and
we claim that 
\begin{equation}\label{ghg}
g(\alpha_h) = \alpha_{\WT{g} \circ h \circ \WT{g}^{-1}}
\end{equation}

For the expression $\alpha_{\WT{g} \circ h \circ \WT{g}^{-1}}$ to make sense, we first check that the composition $\WT{g} \circ h \circ \WT{g}^{-1}$ is a deck transformation of $\pi_{univ} : \WT{M}_{univ} \to M$:
\begin{eqnarray*}
 \pi_{univ} \circ (\WT{g} \circ h \circ \WT{g}^{-1}) &=& g \circ \pi_{univ} \circ h \circ \WT{g}^{-1} \\
 &=& g \circ \pi_{univ} \circ \WT{g}^{-1} \\
 &=& g \circ g^{-1} \circ \pi_{univ} =\pi_{univ}.
\end{eqnarray*}

To prove the claim, it is enough to compare actions of both sides in \eqref{ghg} on the fiber $\pi_{univ}^{-1}(x_0)$ since the covering is regular. Take any point $y$ in  the fiber $\pi_{univ}^{-1}(x_0)$.
 We consider the path $\WT{\alpha_h}$ which is the lift of the loop $\alpha_h$ starting from $y$. Then, its end point gives
$$h (y )  = \WT{\alpha_h} (1).$$
Hence, $g (\alpha_h)$ is a loop based at $x_0$, and we get a corresponding deck transformation, say
$g(h)$. Let  $\WT{g (\alpha_h)}$ be the lift of $g(\alpha_h)$ starting from $y$. Then, $\alpha_{g(h)}= g(\alpha_h)$ sends $y$  to $(\WT{g(\alpha_h)})(1)$. To get \eqref{ghg}, we have to show that
$$(\WT{g(\alpha_h)})(1)= \WT{g}( h(\WT{g}^{-1} (y))) .$$

Let $\WT{\alpha_h}'$ be the lifting of $\alpha_h$ which begins at $\WT{g}^{-1} (y)$ instead of $y$.  $\WT{g} (\WT{\alpha_h}')$ is the lifting of $g (\alpha_h)$ which starts from $y$ because
$$\WT{g} ( \WT{\alpha_h}') (0) = \WT{g} \cdot \left( \WT{g}^{-1} (y)  \right)=y,$$
and
$$\pi \circ \left(\WT{g} ( \WT{\alpha_h}' )\right) = g ( \pi \circ \WT{\alpha_h}') = g (\alpha_h).$$
Now
$$(\WT{g (\alpha_h}))(1) = \WT{g}( \WT{\alpha_h}' (1)) = \WT{g}( h (\WT{g}^{-1} (y))),$$
which proves the claim \eqref{ghg}.

By the previous lemma, 
$g(\alpha_h)$ and $\alpha_h$ are the same elements modulo ${\rm Ker} I_\eta$.
Projecting down \eqref{ghg} from $\WT{M}_{univ}$ to $\WT{M}$, we obtain for $h \in \Gamma$ that
$$\WT{g} \circ h \circ \WT{g}^{-1} = h.$$
on $\WT{M}$.
\end{proof}
We will give more geometric proof of Lemma \ref{GGamma} in Subsection \ref{decktrans}. The $G$-action on $\Gamma$ induces the trivial $G$-action on $\Lambda_{[\eta]}$ as well.

Note that $\pi^\ast \eta$ is exact by construction. So, we can find a smooth function  $\WT{f}:\WT{M} \to \R$ such that
$$d \WT{f} = \pi^\ast \eta.$$
From the transversality assumption on $\eta$, $\WT{f}$ is a Morse function on $\WT{M}$.

\begin{lemma}\label{invMorse}
$\WT{f}$ is $G$-invariant.
\end{lemma}
\begin{proof}
Since $\eta$ is $G$-invariant and $\pi$ is $G$-equivariant, we have for each $g \in G$,
$$d \left( \WT{g}^\ast \WT{f} - \WT{f} \right) = \WT{g}^\ast d \WT{f} - d \WT{f} 
= \WT{g}^\ast (\pi^\ast \eta) - d \WT{f} 
= \pi^\ast (g^\ast \eta) - d \WT{f} 
= \pi^\ast \eta - d \WT{f} =0.$$

As we assumed that $M$ is connected,  there is a constant $c_g$ such that for all $x \in M$,
\begin{equation}\label{const}
c_g = \WT{f}(g \cdot x) - \WT{f}(x).
\end{equation}
Putting $g^{i} \cdot x$ in the place of $x$, we get
\begin{equation}\label{ind}
\left\{
\begin{array}{rll}
c_g &=& \WT{f}(g^2 \cdot x) - \WT{f} (g \cdot x) \\
c_g &=& \WT{f}(g^3 \cdot x) - \WT{f} (g^2 \cdot x) \\
&\vdots&\\
c_g &=& \WT{f}(x) - \WT{f} (g^{|g|-1} \cdot x)\\
\end{array} \right.
\end{equation}
since $g^{|g|} = 1$ where $|g|$ is the order of $g \in G$. By summing up, we have $|g| \, c_g =0$ and, hence $c_g=0$. We conclude that $g^\ast \WT{f} = \WT{f}$. In fact, with the assumption \ref{assum:x_0fixed}, we have $c_g = \WT{f}(g \cdot x_0)  - \WT{f}(x_0)=0$. But the above argument works without the assumption  \ref{assum:x_0fixed}.
\end{proof}

\begin{lemma}\label{lem:isolocalgroup}
The isotropy group  
$G_{\WT{x}}$ is the same as $G_x$ for all $x \in M$.
\end{lemma}
\begin{proof}
It is easy to see that $G_{\WT{x}} \subset G_x$.
For $g \in G_x$, $g \cdot \WT{x}$ and $\WT{x}$ may not be homotopic (regarding $\WT{x}$ as a path from $x_0$ to $x$), but
they differ by the $\KI$-action. More precisely $\WT{x} \star  (g \WT{x})_- \in \KI$ sends $g \cdot \WT{x}$ to $\WT{x}$.
\end{proof}

\subsection{Morse theory for Global quotients}
Before constructing the $G$-Novikov complex, we briefly recall the construction of Morse complexes 
for global quotients from \cite{CH} (using the language of orientation spaces in \cite{Se}).

For the global quotient orbifold $[M/G]$, consider a $G$-invariant Morse-Smale
function $f:M \to \R$, and denote by $\OL{f}:M/G \to \R$, the induced map on the quotient.
$G$-action sends a critical point of $f$ to another critical point, but this naive group action on $crit(f)$ is not
the right action to consider as observed in  \cite{CH}. 

 For the critical point $p$, consider the unstable manifold $W^u(p)$ of the negative gradient flow of $f$, and its tangent space $T_pW^u(p)$. In \cite{CH}, we observed that $G$ should act on the determinant bundle of ``unstable directions'',
$\wedge^{top}(T_pW^u(p))$  of each critical point. 
(We remark that such determinant line bundles also appear for Morse complexes
of Morse-Bott functions.)
\begin{definition}
Orientation space $\Theta_{p}^-$ at $p$ is defined as
$$\Theta_{p}^- := \Lambda ^{top} T_{p} W_{f}^u (p).$$
\end{definition}

Without a group action, such additional data are always trivial. However when there is a group action, $G$ may act on orientation spaces in a non-trivial way. This is a crucial difference, in the sense that the critical point $p$ is fixed by the $G_p$-action, but $p$ together with  its orientation space may not be fixed by the $G_p$-action, as $G_p$ may reverse the orientation of unstable directions (such a critical point was called non-orientable in \cite{CH}). Hence, when we take the $G$-invariant part, non-orientable critical points disappear, and rightfully so since the topology of the sublevel sets of the orbifold does not change at such a critical point (see \cite{LT}). We refer readers to \cite{CH} for more details.

Orientation spaces naturally arise also when we consider the assignment of
signs for each isolated gradient flow in the definition of the boundary operator of Morse complexes.
Consider two critical points $p$ and $q$ with Morse indices $i(p) = i(q)+1$.
Then, the space $\WT{\CM}(p,q)$ of gradient flows between $p$ and $q$ can be identified with
\begin{equation}\label{orimorsefor}
\WT{\CM}(p,q) = W^u(p) \cap W^s(q),
\end{equation} 
where $W^s(q)$ is  the stable manifold of $q$ for $-\nabla f$.  

In terms of orientations, \eqref{orimorsefor} can be written as follows:
For $u \in \WT{\CM}(p,q)$, 
\begin{equation}\label{eq:moropq}
 \wedge^{top} T_u \WT{\CM}(p,q) \cong   \Theta_{p}^- \otimes   (\Theta_{q}^-)^\vee,
\end{equation}
where $(\Theta_{q}^-)^\vee$ is the dual space with respect to the canonical pairing $(\Theta_{q}^-)^\vee \otimes (\Theta_{q}^-) \to \R$. In the isomorphism \eqref{eq:moropq}, two copies of $\wedge^{top}TM$ should appear but cancelled. These copies come from
the relation of $W^s(q)$ and $W^u(q)$, and from taking the intersection of them in $M$.
(see (12.2) of \cite{Se} for the Floer theoretic analogue). Therefore, this also explains why
orientation spaces $\Theta_{p}^-$'s are needed to define a $G$-action.

Here, $\Theta_{p}^-$  is  a one dimensional real vector space, and we will be only
concerned about the sign of the elements of these vector spaces, not the actual magnitudes when identified with $\R$. Hence, we use the normalization convention, following \cite{Se} as explained in the introduction.
We set up the Morse complex of $f$ as follows.
\begin{definition}
Define
\begin{equation}
C_k(M,f) = \bigoplus_{{\rm ind}(p) =k} |\Theta_{p}^-|_R
\end{equation}
\end{definition}
We recall that the differential of the Morse complex can be understood as follows.
The time translation of $\WT{\CM}$ gives an element  in $\wedge^{top} T_u \WT{\CM}(p,q)$,
and from the equality \eqref{eq:moropq}, this determines an isomorphism
$c_u:\Theta_{q}^- \to   (\Theta_{p}^-)$.
We use $|c_u^{-1}|_R$ to define a boundary map 
\begin{equation}\label{eq:morsebdy}
\partial^{p,q}: \sum_u |c_u^{-1}|_R: |\Theta_{p}^-|_R \to |\Theta_{q}^-|_R.
\end{equation}
This is a coordinate-free way of writing the usual boundary operator of the Morse complex.
Namely, a trajectory has an induced orientation from the transversal intersection $W^u(p) \cap W^s(q)$ and we assign $+1$ if it agrees with the direction of the flow itself (the time translation) and $-1$ otherwise.

Now, we can consider a natural $G$-action on the Morse complex $C_*(M,f)$.
It is not hard to see that the $G$-action commutes with the Morse differential in this new setting.
Hence, we can take the $G$-invariant part  $C_*(M,f)^G$, which is
the Morse complex for $[M/G]$. Its homology is isomorphic to the
singular homology of the quotient space of $[M/G]$ with $R$-coefficients for $R \supset \Q$. Here $\Q$ is needed since
over $\Z$, local Morse data are much more complicated.

One difficulty of building up $G$-Morse theory is the equivariant transversality problem.
Namely, it is hard to make a generic perturbation of either the Morse
function or the metric so that it becomes both Morse-Smale and  $G$-invariant.
There are two ways to overcome this issue. In \cite{CH}, we introduced the notion of weak group actions on Morse-complexes, 
for a generic Morse-Smale function $f$, which is not $G$-invariant.  It can be also applied
in this setting, and later we will use this idea to define the equivariant Fukaya category for an exact symplectic manifold.

Another way is to work with a $G$-invariant Morse function which is not necessarily Morse-Smale, and
use the multi-valued perturbation scheme such as Kuranishi structure(\cite{FO}). We will use this approach in
the general case of Floer theory to achieve equivariant transversality.
We also need $\Q \subset R$ for this method. From now on, we assume that $\WT{f}$ on $\WT{M}$
is of Morse-Smale type, and leave more general case to one of the above two methods.

\subsection{Constructions of $G$-Novikov complexes}
We spell out how to set up the $G$-equivariant Novikov chain complex of $\eta$ under Assumption 
\ref{assum:x_0fixed}.
 We denote by
$crit(\WT{f})$ the set of critical points of $\WT{f}$. We will use the normalization $|\Theta|_R$
defined in the introduction.

\begin{definition}\label{def:chain}
We define the Novikov chain complex of $M$ with respect to $\eta$, $CN_\ast(M;\eta)$ as follows:
the graded $R$-vector space $CN_\ast (M;\eta)$ is defined by
$$CN_\ast (M;\eta) := \left\{  \left. \sum_i  x_{\WT{p}_i}  \right|  x_{\WT{p}_i} \in  |\Theta_{\WT{p_i}}^- |_R \;\; \textrm{satisfying}\;\;  \Diamond \right\},$$
where $\Diamond$ is the condition that 
for each $c \in \RR$, the following set is finite: 
$$ \{ i : x_i \neq 0, \, \WT{f} (\WT{p}_i ) < c\}.$$
The boundary operator $\delta$ on $CN_\ast (M;\eta)$ is defined by counting gradient flow lines of $\WT{f}$ in $\WT{M}$ as in $G$-Morse theory \eqref{eq:morsebdy}. 
\end{definition}
Note that the deck transformation group $\Gamma$ preserves the set $\CR(\WT{f})$, and this  gives $CN_\ast (M ; \eta)$ a $\Lambda_{[\eta]}$-module structure,
\begin{equation}\label{eq:modulecn}
\phi : \Lambda_{[\eta]} \times CN_\ast ( M;\eta) \to CN_\ast ( M;\eta).
\end{equation}
More precisely, consider a generator $h$ of $\Lambda_{[\eta]}$. i.e. an element $h$ of $\Gamma = \pi_1 (M) / {\rm{Ker}} I_{\eta}$. Then the scalar multiplication of $h$ to $\WT{p} \in \CR(\WT{f})$ is simply the action image $h (\WT{p})$. Then, $\Lambda_{[\eta]}$-module structure on $CN_\ast (M;\eta)$ is given as follows. As $\Theta_{\WT{p}}^-$ is defined by $\Lambda ^{top} T_{\WT{p}} W_{\WT{f}}^u (\WT{p})$ and $h$ sends $W_{\WT{f}}^u (\WT{p})$
to  $W_{\WT{f}}^u (h(\WT{p}))$, we have an isomorphism (denoted by the same letter)
$$h : \Theta_{\WT{p}}^- \to \Theta_{h(\WT{p})}^-.$$

\begin{lemma}
We have an induced $G$-action on $CN_\ast (M;\eta)$, which is $\Lambda_{[\eta]}$-linear. 
The differential $\delta$ is $G$-equivariant.
\end{lemma}
\begin{proof}
As $\WT{f}$ being $G$-invariant, $G$ acts on the set $\CR(\WT{f})$. Thus, each $g$ induces an isomorphism 
$$g : \Theta_{\WT{p}}^- \to \Theta_{g(\WT{p})}^-.$$
as before. Hence, we obtain a $G$-action on generators of $CN_\ast (M;\eta)$ over $\Lambda_{[\eta]}$. Note that $G$ and $\Gamma$ actions commute (since their actions on $\WT{M}$ commute) and the lemma follows. The proof of $G$-equivariancy follows from \eqref{eq:moropq}
and the definition of $\delta$.
 \end{proof}

\begin{definition}
We call the complex $CN_\ast (M;\eta)$ equipped with the above $G$-action, the $G$-equivariant Novikov complex of $(M, \eta)$.
Its $G$-invariant part $CN_\ast^G (M;\eta)$ (which
is a $\Lambda_{[\eta]}$-module)
 is called the Novikov complex of $([M/G],\OL{\eta})$ where $\OL{\eta}$ is the induced 1-form on $[M/G]$.   
 \end{definition}
Here, the $\Lambda_{[\eta]}$-module structure of $CN_\ast^G (M;\eta)$ is given by taking the $G$-invariant part of \eqref{eq:modulecn} (recall that the $G$-action is $\Lambda_{[\eta]}$-linear):
$\phi^G : \Lambda_\eta \times CN_\ast^G ( M;\eta) \to  CN_\ast^G ( M;\eta)$, 
and the boundary $\delta$ is $\Lambda_{[\eta]}$-linear(\cite{On}).
In conclusion, we get
$$HN_\ast^G (M, \eta) := H_\ast \left(CN_\ast^G (M;\eta), \delta  \right),$$
which is a module over $\Lambda_{[\eta]}$.


What actually happens in taking the $G$-invariant part is as follows.
For a critical point $\WT{p}$ of $\WT{f}$, consider the isotropy group $G_p (=G_{\WT{p}})$. If the $G_p$-action on $\Theta_{\WT{p}}^- $ is orientation preserving, then this provides
a non-trivial element in the $G$-invariant part. If some element $g$ of $G_p$ reverses the orientation of $\Theta_{\WT{p}}^- $, then exactly the half of $G_p$ reverses
the orientations of $W_{\WT{f}}^u (\WT{p})$. Thus
the sum of $G$-action images of $( \WT{p}, \Theta_{\WT{p}}^- )$ becomes
zero by the cancellation, and hence do not contribute to the $G$-invariant part of the Novikov complex. 
 We refer readers to \cite{CH} for more details on related phenomenon.

\section{General $G$-Novikov theory}\label{sec:general base}
In this section, we construct a $G$-Novikov complex in the general case without assumption \ref{assum:x_0fixed}.
The main difficulty is that the construction of covering spaces (as spaces of homotopy classes of paths from a chosen base point) is not compatible with the $G$-action. Indeed, the $G$-action may move the chosen base point to another point. 

For this, we will introduce the notion of  {\it energy zero subgroup} $G_\eta$, and define its action on Novikov complexes.
If $G =G_\eta$, we can define equivariant and orbifolded theory from this $G_\eta$ action.
But in the general case that $G \neq G_\eta$, we need to use a different Novikov ring, what is called $\Lambda^{orb}$.
In fact, the most natural way to view it is to  consider the orbifold fundamental group of the global quotient orbifold $[M/G]$ and the orbifold analogue of the Novikov covering construction. We will explain how $G_\eta$, and $\Lambda^{orb}$ arise in this orbifold setting.
Hence, if $G \neq G_\eta$, we cannot define $G$-equivariant theory, but we can define $G$-orbifolded theory, 
by taking  $G_\eta$-invariant parts of the Novikov complex, with Novikov ring given by $\Lambda^{orb}$.

\subsection{Energy zero subgroups and their actions on Novikov coverings}
We first take a generic point $x_0 \in M$ such that the isotropy group at $x_0$ of the $G$-action is trivial. That is, $g \cdot x_0 \neq x_0$ for all $g \neq 1 \in G$. As the $G$-action on $M$ is effective, such $x_0$ always exists. Let $\eta$ be a $G$-invariant Morse 1-form $\eta$ on $M$ and 
consider the universal covering space $\WT{M}_{univ}$ and Novikov covering $\pi:\WT{M} \to M$ obtained by $\WT{M} = \WT{M}_{univ}/\textrm{Ker} I_\eta$ where $I_\eta(\alpha) = \int \alpha^\ast \eta$ for $\alpha \in \pi_1(M,x_0)$.

Now, we take the group action into account. Note that the $G$-action on $M$ does not induce a $G$-action on $\WT{M}$:
consider $\WT{x} \in \WT{M}$ with $\pi(\WT{x}) =x \in M$. Then, $\WT{x}$ is an equivalent class of paths from $x_0$ to $x$.
The naive $G$-action sends $\WT{x}$ to  $g(\WT{x})$, which is a homotopy class of paths from $g \cdot x_0$ to $g \cdot x$. This path is not an element of $\WT{M}$ since it does not start from $x_0$.

\begin{definition}
A subgroup $G_\eta$ of $G$ consists of an element $g \in G$ such that
there exist a path $\gamma_g$ from $x_0$ to $g \cdot x_0$, with 
$$\int \gamma_g^\ast \eta=0.$$
We call $G_\eta$ the {\em energy zero} subgroup of $G$ for $\eta$ and a path $\gamma_g$ an energy zero path.
\end{definition}
To see that $G_\eta$ indeed forms a subgroup, consider $g,h \in G_\eta$ with energy zero paths $\gamma_g$ and $\gamma_h$ respectively. Then, we can take an energy zero path for $gh \in G_\eta$ to be the concatenation 
$$\gamma_{hg}=\gamma_g \star g(\gamma_h).$$

\begin{remark}
A priori, such energy zero paths
may not always exist (i.e. $G_\eta \neq G$ in general).
For example, consider $G=\Z/3$-action defined by the rotation on $S^1$ and $d \theta$ a $G$-invariant $1$-form on $S^1$. Then, $G_\eta =\{id\} \neq G$.
\end{remark}

By considering the energy zero subgroup $G_\eta$, we can define a $G_\eta$-action on the Novikov covering $\WT{M}$ as below.
\begin{definition}\label{GactionNovcov}
Given $\WT{x} \in \WT{M}$ (regarded as a class of paths in $M$) and $g \in G_\eta$ with an energy zero path $\gamma_g$,
we define the action of $g$ on $\WT{x}$ as the concatenation of $\gamma_g$ with the naive $g$-action image of $\WT{x}$:
\begin{equation}\label{star}
g \cdot \WT{x} := \gamma_g \star g(\WT{x}).
\end{equation}
\end{definition}
\begin{lemma}
This gives a well-defined $G_\eta$-action on $\WT{M}$.
\end{lemma}
\begin{proof}
For a different choice of an energy zero path $\gamma_g'$ for $g$, the concatenation $\gamma_g \star (\gamma_g')_-$ 
defines a loop in $\KI$ where $\gamma_- (t) = \gamma (1-t)$ for a path $\gamma$. Then, two elements
$\gamma_g \star ( g (\WT{x}))$, and $\gamma'_g \star ( g (\WT{x}))$ differ by
an action of $\KI$, hence the action of $G_\eta$ on $\WT{M}$ is well-defined. Also it is easy to check that
$g \cdot (h \cdot \WT{x}) =gh \cdot \WT{x}$ for $g,h \in G_\eta$.
\end{proof}
If one tries to define a $g$-action similarly for an element $g \in G \setminus G_\eta$
(by taking a path $\gamma_g$ from $x_0$ to $g(x_0)$), then it is straightforward to see that the relation $g^{|g|} = id$ cannot hold for the action of $g$. Indeed, attaching $\gamma_g$ keeps increasing (or decreasing) the energy.

Now let us closely look into the $G_\eta$-action $\WT{M}$. From the construction, the projection $\pi : \WT{M} \to M$ is a $G_\eta$-equivariant map:
$$\pi(g \cdot \WT{x}) = g\cdot \pi(\WT{x}).$$ 
The pull-back $\pi^\ast \eta$ is exact by the definition of $\WT{M}$ and hence we can choose $\WT{f}:\WT{M} \to \R$
so that $ d \WT{f} =\pi^\ast \eta$ as before. We can proceed as in Lemma \ref{invMorse} to obtain
\begin{lemma}
$\WT{f} : \WT{M} \to \RR$ is $G_\eta$-invariant. 
\end{lemma}

As we assumed that $M$ is connected, we can find interesting properties of $G_\eta$:

\begin{lemma}
$G_\eta$ is a normal subgroup of $G$.
\end{lemma}

\begin{proof}
For $h \in G_\eta$ and $g \in G$, we want to show that $g h g^{-1}$ lies in $G_\eta$. So, we need to find an energy zero path from a base point $x$ to $(g h g^{-1}) \cdot x$. Set $y= g^{-1} \cdot x$. Then, it suffices to find such a path from $y$ to $h \cdot y$. Choose any path $\delta$ from $y$ to $x$ and let $\gamma_h$ be an energy zero path from $x$ to $h \cdot x$. Then, it is easy to check that $\delta \star \gamma_h \star h (\delta)_-$ from $y$ to $h \cdot y$ has energy zero.
\end{proof}

\begin{lemma}\label{locgpinG}
For any $y \in M$, the isotropy group $G_y$ is a subgroup of $G_\eta$. 
\end{lemma}

\begin{proof}
We fix $y \in M$, and $g \in G_y$.
Since $M$ is connected, we can choose a path $\alpha$ from $x_0$ to $y$. Then, $g (\alpha)$ defines a path from $g \cdot x_0$ to $g \cdot y(=y)$.
Then, we get a path  from $x_0$ to $g \cdot x_0$ by
$$ \gamma := \alpha \star (g \cdot \alpha)_-.$$
Note that 
$$I_\eta(\gamma)= I_\eta(\alpha) + I_\eta((g \cdot \alpha)_-) =
I_\eta( \alpha) - I_\eta(\alpha) =0$$
Thus, $\gamma$ has zero energy and realizes $g$ as an element of $G_\eta$.  This
implies that $G_y \subset G_\eta$ for all $y$.
\end{proof}
\begin{corollary}\label{cor:getag}
If $M$ has a point, say $x_1$, which is fixed by the whole $G$-action, then
$G_\eta = G$.
\end{corollary}
\begin{proof}
Take $y$ to be $x_1$, and apply Lemma \ref{locgpinG}.
\end{proof}

\subsection{Deck Transformations and $G_\eta$-actions on Novikov complexes}\label{decktrans}
\begin{lemma}\label{lem:deccomm}
The $G_\eta$-action on $\WT{M}$ commutes with the action of the deck transformation
group $\Gamma = \pi_1(M) / {\rm{Ker}} I_{\eta}$.
\end{lemma}
\begin{proof}
The proof of Lemma \ref{GGamma} (for the case with a $G$-fixed base point) extends to this
general case. But here, we give a geometric
proof of it using the $G_\eta$-action \eqref{star}.

Let $h \in \Gamma$ be represented by a loop $\alpha_h$ at $x_0$. If $\gamma_g$ realizes $g \in G_\eta$ (i.e. $\gamma_g (1) = g \cdot x_0$ and $I_{\eta} (\gamma_g) =0$), then $\gamma_{g^{-1}} = g^{-1} \cdot (\gamma_g)_-$ realizes $g^{-1}$. Any point $\WT{x}$ in $\WT{M}$ can be thought of as a (homotopy class of) path from $x_0$ modulo $\KI$-action. Then, we have path representations of elements in $\WT{M}$ as follows:
\begin{figure}[h]
\begin{center}
\includegraphics[height=2in]{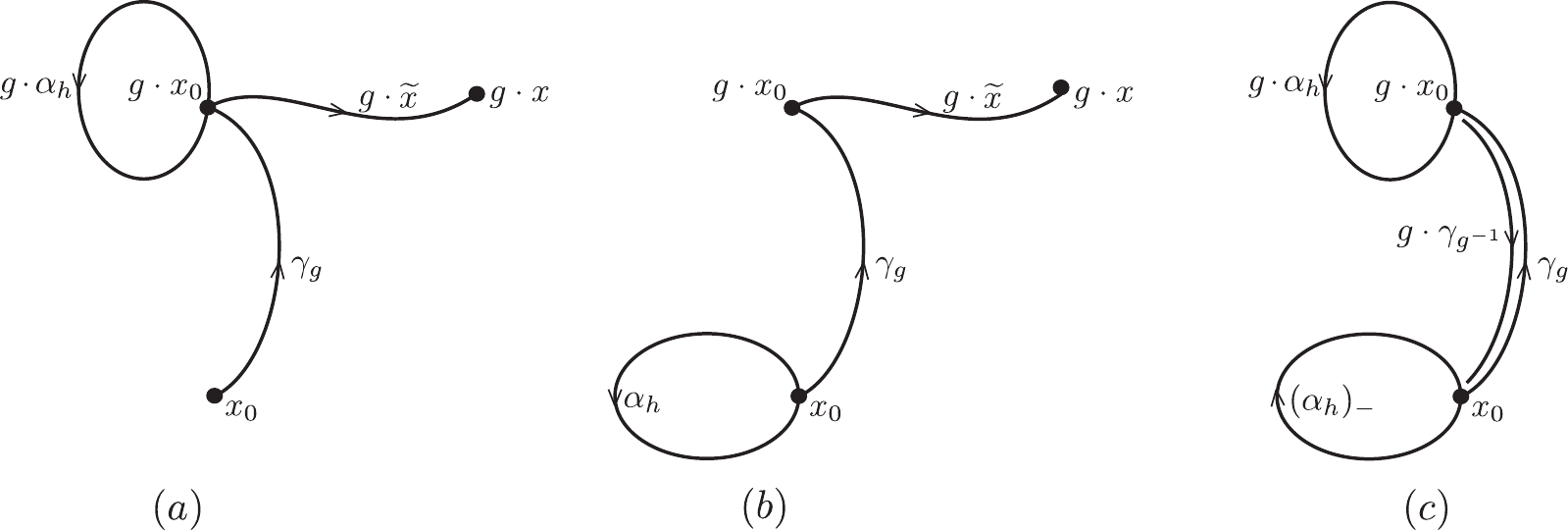}
\caption{(a) $g\cdot (h\cdot \WT{x})$ \,  (b) $h \cdot (g \cdot \WT{x})$  \, (c) $g\cdot (h\cdot \WT{x})- h \cdot (g \cdot \WT{x})$}
\end{center}
\end{figure}

\begin{enumerate}
\item[(a)] $g \cdot (h \cdot \WT{x}) $ is represented by the concatenation $\gamma_g \star (g \cdot \alpha_h) \star (g \cdot \WT{x})$.
\item[(b)] $h \cdot (g \cdot \WT{x}) $ is represented by the concatenation $\alpha_h \star \gamma_g \star (g \cdot \WT{x})$
\item[(c)] The difference between (a) and (b) is represented by $\gamma_g \star (g \cdot \alpha_h) \star (\gamma_g)_- \star (\alpha_h)_-$.\\
\end{enumerate}

Since $\eta$ is $g$-invariant 
$$\int_{\alpha_h} \eta + \int_{g \cdot (\alpha_h)_-} \eta =0,$$
We conclude that (c) has energy zero, or equivalently 
$$g\cdot (h\cdot \WT{x})=h \cdot (g \cdot \WT{x}).$$
\end{proof}

Since the $G$-action on $M$ is orientation preserving, it is easy to see that $G_\eta$ preserves the induced orientation on $\WT{M}$.
Here is a simple lemma on the relationship between isotropy groups in $M$ and $\WT{M}$
analogous to Lemma \ref{lem:isolocalgroup}
\begin{lemma}\label{lem:isolocalgroup2}
Consider $\WT{x} \in \WT{M}$ with $\pi(\WT{x}) =x$.
Then the isotropy group $(G_\eta)_{\WT{x}}$ at $\WT{x}$ in $G_\eta$ is given by
\begin{equation}\label{isotropyGeta}
(G_\eta)_{\WT{x}}=G_x \cap G_\eta = G_x.
\end{equation}
\end{lemma}
\begin{proof}
The second equality in \eqref{isotropyGeta} follows from the lemma \ref{locgpinG}. We only prove the first equality, here. For $g \in G_\eta$, we choose an energy zero path $\gamma_g$ from $x_0$ to $g \cdot x_0$ in $M$. Then,  if $g \in  (G_\eta)_{\WT{x}}$, then 
$ g \cdot \WT{x}$ (which is represented by $\gamma_g \star g(\WT{x})$) has the same end point as $\WT{x}$.
This means that $g$ is an element of $G_x$.

Conversely, if $g$ fixes $x$, and admits an energy zero path $\gamma_g$, 
then both $\gamma_g \star g(\WT{x})$ and $\WT{x}$ is a path from $x_0$ to $x$ in $M$,
and $I_\eta$ vanishes on $\gamma_g \star g(\WT{x}) \star \WT{x}_-$. Thus
$g \cdot \WT{x}$ and $\WT{x}$ differ by an action of $\KI$, which proves that
$g \in (G_\eta)_{\WT{x}}$.
\end{proof}

To define a $G$-equivariant Novikov complex, let us first assume that $G=G_\eta$,
and postpone the study of general case to the next subsection.
We define the Novikov chain complex as in Definition \ref{def:chain}, and Novikov ring $\Lambda_{[\eta]}$ as in \eqref{ring}. By proceeding as before, we obtain
\begin{prop}\label{thm:novchaincpx}
Assume that the induced $G (=G_\eta)$-invariant function  $\WT{f}:\WT{M} \to \R$ is of Morse-Smale type.
Then, we have an induced $G$-action on the Novikov complex $(CN_*(M;\eta),\delta)$ which is $\Lambda_{[\eta]}$-linear
(i.e.  $G$-action  commutes with $\Lambda_{[\eta]}$).
\end{prop}
This defines the $G$-equivariant Novikov complex of $(M,\eta)$. 
By taking the $G$-invariant part of $CN_*(M;\eta)$,
we get the chain complex 
$$(CN_*^{G}(M;\eta), \delta),$$
which should be considered as Novikov homology of the orbifold $([M/G],\OL{\eta})$ with respect to the induced 1-form $\OL{\eta}$.
This construction agrees with the previous construction of Novikov complex with a $G$-fixed point by Corollary \ref{cor:getag}, and we leave the detailed check as an exercise.

When $G \neq G_\alpha$, proper definition of orbifolded Novikov complex is more involved as we need
to introduce orbifold Novikov ring $\Lambda^{orb}_{[\eta]}$.

\subsection{Orbifold  Novikov ring $\Lambda^{orb}_{[\eta]}$ and orbifold Novikov complex}\label{subsec:OrbFund}
First, we explain how the energy zero subgroup $G_\eta$ naturally arises from the orbifold setting. We give a brief review on the orbifold fundamental group $\pi_1^{orb} \left( [M/G]\right)$.

We call a continuous map $\alpha:[0,1] \to M$ a generalized loop based at $x_0$ if
$\alpha(0)=x_0$ and $\alpha(1) = g_\alpha \cdot \alpha(0)$. i.e. it is a loop up to $G$-action.
 Note that a genuine loop at $x_0$ is obviously a generalized loop. Since $x_0$ is generic, $g_\alpha$ is uniquely determined from $\alpha$.

Two generalized loops  $\alpha$ and $\beta$ (based at $x_0$) are homotopic  if
\begin{itemize}
\item $\alpha (1) = \beta(1)$.
\item There is a homotopy $H$ between $\alpha$ and $\beta$ relative to end points.
\end{itemize} 
We write $\alpha \sim \beta$ for such homotopy relation. 

Two generalized loops can be multiplied as follows.
\begin{equation}\label{newactconcate}
\alpha \cdot \beta = \alpha \, \star  \, g_\alpha ( \beta).
\end{equation}
where $\star$ denotes the concatenation of two paths as before.

The set of homotopy classes of generalized loops (based at $x_0$), denoted by $\pi_1^{orb} \left( [M/G],\OL{x}_0\right)$,
has a group structure induced by \eqref{newactconcate}.
 One can also show that $\alpha^{-1} = g_\alpha^{-1} (\alpha_-)$ gives the inverse of $[\alpha]$, where $\alpha_- (t) = \alpha (1-t)$. 
 
We have a group homomorphism $e : \pi_1^{orb} \left( [M/G]\right) \to G$  defined by 
$$e([\alpha]) = g_\alpha.$$
It is easy to see that $${\rm{Ker}} \, e = \pi_1 (M, x_0).$$

Thus, we have the following  short exact sequence of groups which is non-split in general.
\begin{equation}\label{ses}
1 \longrightarrow \pi_1 (M) \longrightarrow \pi_1^{orb} \left( [M/G]\right) \stackrel{e}{\longrightarrow} G \longrightarrow 1.
\end{equation}
The map $e$ is surjective since we assumed that $M$ is connected.

\begin{lemma}\label{lem:orbfix}
If $M$ has a point, say $x_1$, which is fixed by the whole $G$-action, then
orbifold fundamental group can be written as a semi-direct product:
$$\pi_1^{orb} ([M/G]) \cong G \ltimes \pi_1 (M).$$
\end{lemma}
\begin{proof}
Recall that the statement of the lemma is equivalent to the
existence of a splitting of the extension \eqref{ses}. Thus, we need to find a homomorphism $\phi:G \to \pi_1^{orb} ([M/G])$ such that
$e \circ \phi = id$ on $G$.
Now, if $G$ fixes $x_1$, we can construct $\phi$ as follows.
Take a path $\gamma:[0,1] \to M$ such that $\gamma(0)=x_0$ and $\gamma(1) = x_1$.
For each $g \in G$, $g(\gamma)$ is a path from $g(x_0)$ to $x_1$.
Thus, consider the concatenation of two paths
$$ \phi(g) := \gamma \star g(\gamma_-),$$
which is a generalized loop starting from $x_0$ ending at $g(x_0)$.
It is easy to check that $\phi$ is a homomorphism, and $e \circ \phi =id$.
\end{proof}

 
It is well-known that the most of the covering theory extends to the case of orbifolds if one considers orbifold fundamental group, orbifold covering and orbifold universal covering. (see for example
Takeuchi \cite{T}). The orbifold universal covering for $[M/G]$ is given by the universal covering
$\WT{M}_{univ}$  of $M$, with the projection maps
$\WT{M}_{univ} \to M \to M/G$. And $\pi_1^{orb} ([M/G])$ becomes
the deck transformation group of  $\WT{M}_{univ} \to M/G$, whose action can be written as follows.

\begin{definition}\label{def:pi1act}
We define the (left) action of 
$\pi_1^{orb} ([M/G])$ on $\WT{M}_{univ}$ as the concatenation of paths:
for $\gamma \in \pi_1^{orb}([M/G])$ with $\gamma(1) = g_{\gamma}(x_0)$, and for a
path $\WT{x}$ from $x_0$ to $x$, as an element of $\WT{M}_{univ}$, 
we define
$$\gamma \cdot \WT{x} := \gamma \star g_{\gamma}(\WT{x}).$$
\end{definition}

Recall $\eta$ is a Morse 1-form on $M$, which is $G$-invariant. We denote by $\overline{\eta}$  the induced 1-form  on the quotient $M/G$.
\begin{definition}
We define $I_{\eta}^{orb} : \pi_1^{orb} ([M/G]) \to \RR$ as follows:
 $$   I_{\eta}^{orb} (\alpha) := \int_0^1 \alpha^*{\eta} \quad \textrm{for}\; \alpha \in \pi_1^{orb} ([M/G]).$$
\end{definition}
This defines a group homomorphism since $\eta$ is $G$-invariant.

We consider the subgroup $\KIO$, the kernel of $I_{\eta}^{orb}$, which acts on the orbifold universal covering $\WT{M}_{univ}$,
and consider the quotient orbifold $[\WT{M}_{univ}/\KIO]$.
An analogue of Novikov covering $ \WT{M} \to M$ in section 2 would be the orbifold covering 
$$[\WT{M}_{univ}/\KIO]  \to [M/G].$$ 
But it is cumbersome to work directly with $[\WT{M}_{univ}/\KIO]$. For example, $\KIO$ is discrete but may not be finite in general. Alternatively, we find another presentation of the orbifold
$[\WT{M}_{univ}/\KIO]$, which is much easier to understand.

\begin{remark}
Similarly to \ref{lem:orbfix}, if there exists a $G$-fixed point, then we have
$$ {\rm{Ker}} I_\eta^{orb} \left( [M/G]\right) \cong G \ltimes {\rm{Ker}} I_\eta.$$
\end{remark}

Before we give the actual construction, it may be helpful to first explain the formalism.
Consider the group $B$ acting effectively on a manifold $X$, and let $A$ be a normal subgroup of $B$.
Then, $A$, $B$, and the quotient group $C = B/A$ form an exact sequence
$ 1 \to A \to B \to C \to 1.$
Note that the $B$-action on $X$ induces the $B/A$-action on the quotient space $X/A$ because $A$ is a normal subgroup. Also the quotient satisfies
\begin{equation}\label{abc}
 (X/A) / (B/A) \cong X/B.
 \end{equation}
Thus instead of $X/B$, we will consider $X/A$ and its quotient by the group $C = B/A$,
and the benefit is that  $X/A$, $B/A$ are simpler than $X$ and $B$
(also in our case below $X/A$ will be still a manifold).

Now, we put $B = \KIO$, and $A = \KI$ and consider 
the following diagram of subgroups of fundamental groups.
\begin{equation}\label{fundgpdia}
\xymatrix{ 1  \ar[r] & {\rm{Ker}} I_{\eta} \ar[d]_{\leq} \ar[r] & {\rm{Ker}} I_{\eta}^{orb} \ar[d]_{\leq} \ar[r]^{\qquad e_{\eta}} & G_{\eta} \ar[d]_{\leq} \ar[r] & 1 \\
1  \ar[r] & \pi_1 (M, x_0)  \ar[r] \ar[d] & \pi_1^{orb} ([M/G])  \ar[r]^{\qquad e} \ar[d] & G  \ar[r]  \ar[d] & 1 \\ 
1 \ar[r] &\Gamma \ar[r]&  \Gamma^{orb}  \ar[r] & G /G_\eta  \ar[r] & 1 }
\end{equation}


\begin{definition}
We define $G_\eta$ to be the quotient group $\KIO/\KI$.
$G_\eta$ can be identified with  a subgroup of $G$ via the map
$$\KIO / \KI \to \pi^{orb}_1 ([M/G]) / \pi_1 (M) =G.$$
\end{definition}
Hence, we have $C = \KIO/\KI$ in the formalism above. From the diagram \eqref{fundgpdia}, we can correspondingly construct a diagram of covering spaces as follows:
\begin{equation}\label{psi}
\xymatrix{ 
   \WT{M}_{univ} \ar[r]^{id} \ar[d]  & \WT{M}_{univ} \ar[d] \\
   \WT{M}=\WT{M}_{univ} / {\rm{Ker}} I_{\eta} \ar[r]^{/G_\eta } \ar[d]    & [\WT{M}_{univ} /{\rm{Ker}} I_{\eta} ^{orb} \ar[d]]  \\
   M  \ar[r]_{/G} & [M/G]   }
\end{equation}
Here $\WT{M} = \WT{M}_{univ} / {\rm{Ker}} I_{\eta}$, which is the Novikov covering space
considered in the previous section. From the discussion above, the identity \eqref{abc} becomes
\begin{equation}
[\WT{M}/ G_\eta] \cong [\WT{M}_{univ}/\KIO].
\end{equation}

Note that $\WT{M}$ is a manifold (it is a covering of $M$), and we have a presentation 
of the orbifold which involves simpler group $G_\eta$ than $\KIO$. This provides an explanation of $G_\eta$, and should extend to the general orbifold case.

In the case when $G \neq G_\eta$, there is a subtle problem with Novikov rings as examined from now on.
The deck transformation group for the Novikov cover $[\WT{M}_{univ}/\KIO] \to [M/G]$ is
given by the quotient group 
$$ \Gamma^{orb} :=  \pi_1^{orb} ([M/G]) / {\rm{Ker}} I_{\eta}^{orb},$$
while the deck transformation group for $\WT{M} \to M$ is given  by
$$\Gamma :=  \pi_1([M/G]) / {\rm{Ker}} I_{\eta}.$$
From the diagram \eqref{fundgpdia}, we see that the difference of two groups is exactly given by  $G/G_\eta$.
We write the associated Novikov rings as $\Lambda_{[\eta]}^{orb}$ and $\Lambda_{[\eta]}$ respectively as in \eqref{ring}. The inclusion $\Gamma \to \Gamma^{orb}$ gives rise to a map $\Lambda_{[\eta]} \to \Lambda_{[\eta]}^{orb}$
between the (compeletions of) group rings.

The Novikov complex is defined in the following way.
From the previous construction, we have the Novikov complex $(CN_*(M;\eta),\delta)$ for 
$\WT{M}$ with Novikov coefficients $\Lambda_{[\eta]}$ and the $G_\eta$-action on it.
Now, we can replace the Novikov coefficients by
\begin{equation}\label{def:novcpx2}
 CN_*(M;\eta) \otimes_{\Lambda_{[\eta]}} \Lambda_{[\eta]}^{orb},
\end{equation}
which still has a $G_\eta$-action.
\begin{definition}\label{def:obnoveta}
We define the orbifolded Novikov complex of $([M/G], \OL{\eta})$ to be
\begin{equation}\label{def:novcpx1}
 \big( CN_*(M;\eta) \otimes_{\Lambda_{[\eta]}} \Lambda_{[\eta]}^{orb} \big)^{G_\eta}
\end{equation}
\end{definition}

For example, let us consider the case that $G$-action is free.
The following proposition shows that the complex in Definition \ref{def:obnoveta} is the right Novikov complex
for the quotient space of the free action.
\begin{prop}
If $G$ acts on $M$ freely, then we have $\Lambda^{orb}_{[\eta]} \cong \Lambda_{[\OL{\eta}]}$ and
the following isomorphism of $\Lambda_{[\OL{\eta}]}$-modules
$$ \big( CN_*(M;\eta) \otimes_{\Lambda_{[\eta]}} \Lambda_{[\eta]}^{orb} \big)^{G_\eta} \cong CN_\ast (M/G; \bar{\eta})$$
\end{prop}

\begin{proof}
Since $G$ acts freely on $M$
\begin{equation}\label{freeNovqout}
\pi_1^{orb} ([M/G]) \cong \pi_1 (M/G), \quad {\rm Ker} I_{\eta}^{orb} \cong {\rm Ker} I_{\bar{\eta}}.
\end{equation}
This gives the desired identification of $\Lambda^{orb}_{[\eta]}$ and $\Lambda_{[\OL{\eta}]}$.

Note that the $G_\eta$-action on $\WT{M}$ is also free. One can see from the diagram \eqref{psi} that the Novikov covering space of $M/G$ with respect to the Morse $1$-form $\OL{\eta}$ is exactly the quotient of $\WT{M}$ by the free $G_\eta$-action. Therefore, $ \big( CN_*(M;\eta) \otimes_{\Lambda_\eta} \Lambda_\eta^{orb} \big)^{G_\eta}$ is naturally isomorphic to $CN_\ast (M/G; \bar{\eta})$ which is essentially the Morse chain complex of $\bar{f} : \WT{M} / G_\eta \to \RR$ with the Novikov coefficients.
\end{proof}

\subsection{Generalizations}
$\left. \right.$\\
{\bf(1)} Let us first drop the assumption that $M$ is connected from theorem \ref{thm:novchaincpx}. 
Denote connected components of $M$ by $M_i$ $(1\leq i \leq k)$. i.e. $M= \sqcup_{i=1}^k M_i$.
Let us call $M_i$ and $M_j$, {\em $G$-related} if there exist  $x \in M_i$ and $g \in G$ such
that $g(x) \in M_j$. This  defines an equivalence relation. 
We divide the index set $\{1,\cdots, k\}$ into 
$$\{i_{11},\cdots, i_{1a_1}\} \sqcup \cdots \sqcup \{i_{j1},\cdots, i_{ja_j}\} $$
using this equivalent relation so that  connected components $M_{i_{l1}}, \cdots, M_{i_{la_l}}$ are $G$-related for any $l =1,\cdots, j$.

To define a group action, we need to make the following assumption.
\begin{assumption}
If we denote by $G_{ab}$ the subgroup of $G$ whose elements preserve the connected component 
$M_{i_{ab}}$, then we require that
$$(G_{ab})_\eta = G_{ab}.$$
\end{assumption}
Hence, for any $g\in G_{ab}$ and $x \in M_{i_{ab}}$, there
exists an energy zero path $\gamma$ connecting $x$ and $g(x)$.

As an initial step, consider the construction of the Morse-Novikov complex for the subcollection $M_{i_{l1}}, \cdots, M_{i_{la_l}}$ for a fixed $l$.
Now, we choose base points $x_{0,i_{lb}} \in M_{i_{lb}}$ so that
they lie in a single $G$-orbit. We can consider Novikov-Morse theory on each connected
component,
with base points given by $x_{0,i_{lb}}$'s.

Then, from the assumption, any $g$-image of $x_{0,i_{lb}}$ for $g\in G$
can be connected to one of $x_{0,i_{lb'}} \in M_{i_{lb'}}$ by energy zero path, say $\gamma_g$.
Then, we can define the group action as before. Namely, we concatenate $\gamma_g$ to
the front of the image under the naive group action. Also, one can 
see that the deck transformation group of one component is isomorphic (by conjugation) to that of the other component, hence the resulting Novikov rings can be identified, which we denote as $\Lambda_l$.

One can take the direct sum of Novikov complexes $(\oplus_{s=1}^{a_l} CN_*(M_{i_{ls}};\eta|_{M_{i_ls}}), \oplus \delta)$ so that it admits a natural $G$-action.

Now, let us illustrate how to add various complexes for different $l$'s. 
For this, we need to introduce the universal Novikov ring to properly compare various Novikov rings
of different subcollections. Then, we will take completed direct sum over all different subcollections
to define the Novikov complex (This is similar  to what has been done for
Novikov Floer theory in \cite{FOOO}, where the case of several connected components of
a path space are considered).

The universal Novikov ring is defined by
$$\Lambda= \left\{ \left. \sum a_i T^{\lambda_i} \right| a_i \in \R, \lambda_i \in \R, \lim_{i \to \infty} \lambda_i = \infty \right\},$$
and the ring homomorphism from $\Lambda_{[\eta]} \to \Lambda$ by
$$ \sum a_i h_i \mapsto \sum a_i T^{I_{\eta}(h_i)}.$$

For each Novikov complex $CN_*(M_{i_{l1}};\eta)$ with $G$-action, we can introduce the new Novikov complex of $M$ as
$$  \WH{\oplus}_{l=1}^j CN_*(M_{i_{l1}};\eta) \otimes_{\Lambda_{[\eta]}} \Lambda $$
which admits a $G$-action also.\\


\noindent{\bf(2)} We have considered a global quotient orbifold $[M/G]$, and have seen how to
set up its Novikov complex by working on $\WT{M}$ observing $G_\eta$-actions.
We can generalize the construction to any effective orbifold if we just use orbifold terms, but
it will not be explicit as in the global quotient case.
As we will not use it in this paper, we briefly outline the construction.
Consider a compact connected oriented orbifold $\chi$ with a closed Morse 1-form $\eta$ on $\chi$.
It is straight forward that $\eta$ defines a homomorphism from $I_\eta: \pi_1^{orb}(\chi) \to \R$. Then, as before its kernel ${\rm{Ker}} I_{\eta}^{orb} \left( \leq \pi_1^{orb}(\chi)\right)$ corresponds to the orbifold covering 
$\pi : \WT{\chi} \to \chi$ on which $\pi^*\eta$ is an exact $1$-form. Note that
$\WT{\chi}$ is no longer a smooth manifold, but rather an orbifold.

Consider the deck transformation group
$\Gamma^{orb} = \pi_1^{orb}(\chi)/\KIO$, and take the (one-sided) completion of the group ring of
$\Gamma^{orb}$ to define $\Lambda^{orb}_{[\eta]}$. 
Let $f:\WT{\chi} \to \R$ be a Morse function on the (orbifold) Novikov cover which integrates $\pi^\ast \eta$. If $f$ is of Morse-Smale type,
we apply the construction in \cite{CH} (for general orbifolds) to obtain the orbifold Morse complex of $f:\WT{\chi} \to \R$ with $\Lambda^{orb}_{[\eta]}$-coefficients. $CN_\ast (\WT{\chi};\eta)$ has a natural $\Lambda^{orb}_{[\eta]}$-module structure and the boundary operator $\delta$ which counts gradient trajectories is $\Lambda^{orb}_{[\eta]}$-linear. We may call the resulting complex the Novikov complex of $(\chi,\eta)$.

\section{Floer-Novikov theory and orientation spaces}\label{sec:FNtheory}
We recall the setup of Floer-Novikov theory briefly (following \cite{FOOO}), and also the notion of orientation spaces
as well as some gluing formulas for them, which are  well-known constructions.  In the last subsection, we will introduce
the spin $\Gamma$-equivalence (refining $\Gamma$-equivalence) to associate the canonical spin
structure to each generator of Floer-Novikov complexes.

\subsection{Floer-Novikov theory}\label{subsec:FNtheory}
Let $L_0$ and $L_1$ be Lagrangian submanifolds of a symplectic manifold $(M,\omega)$, transversally intersecting with each other. Let us assume that both $L_0$ and
$L_1$ are connected, compact, oriented and also spin for simplicity.
Consider the space of paths
$$\Omega(L_0, L_1) = \{ l : [0,1] \to M \,\, | \,\, l(0) \in L_0, l(1) \in L_1\}.$$
We denote by
$\Omega(L_0, L_1;l_0)$ its connected component containing $l_0 \in \Omega(L_0, L_1)$. 

We define an $1$-form $\alpha$ on $\Omega(L_0,L_1;l_0)$ as follows: for $\xi \in l^\ast TM = T_l \Omega (L_0,L_1;l_0)$, 
\begin{equation}\label{def:alpha}
\alpha (\xi) = \int_{[0,1]} \omega(l'(t), \xi(t)) dt.
\end{equation}
The $1$-form $\alpha$ is closed but not necessarily exact and hence, one needs to find a suitable  Novikov covering for $\alpha$.

First let us consider the universal covering $\WT{\Omega}_{univ}(L_0,L_1;l_0)$ of  $\Omega (L_0, L_1;l_0)$. A path in $\Omega (L_0, L_1 ;l_0)$ is a continuous map $w : [0,1] \times [0,1] \to M$ satisfying $w_0 (t) =l_0 (t)$ and $w_s \in \Omega(L_0,L_1;l_0)$ where $w_s(t) = w(s,t)$.
If $w_1(t) =l(t)$ for $l \in \Omega (L_0, L_1 ; l_0)$, the homotopy class of such $w$ gives a point in the  universal covering space $\WT{\Omega}_{univ}(L_0,L_1;l_0)$ lying over $l$.
 
Denote by $\OL{w}(s,t) := w(1-s,t)$.
Given two elements $(w,l)$ and $(w',l')$ of $\Omega (L_0, L_1;l_0)$, we can glue these to get a loop $C:=w \star \OL{w'}$ in $\Omega(L_0, L_1;l_0)$. So we have a loop
$$C : S^1 \times [0,1] \to M$$
which satisfies the Lagrangian boundary conditions
\begin{equation}\label{bdycond}
C(s,0) \in L_0, \,\, C(s,1) \in L_1.
\end{equation}
We can define its symplectic area $I_\omega(C) = \int_C \omega$ and Maslov index $I_{\mu}(C)$. The Maslov index $I_{\mu}$ is defined as usual by the difference of the Maslov indices of the loops in the Lagrangian Grassmannian, where the loops are obtained from two boundary components of $C$ after a symplectic trivialization of the pull-back bundle over $C$.
This gives  a homomorphism 
\begin{equation}\label{eq:alphamu}
(I_{\omega}, I_\mu) :\pi_1 (\Omega (L_0, L_1;l_0)) \to \RR \times \ZZ.
\end{equation}

The following definition of the $\Gamma$-equivalence between two paths in $\Omega(L_0, L_1;l_0)$ is due to \cite{FOOO}. 
\begin{definition}\label{def:gammaeq}
$(w,l)$ and $(w',l)$ are said to be  $\Gamma$-equivalent if 
$$I_{\omega} (w \star \OL{w'}) = 0 = I_{\mu} (w \star \OL{w'}).$$
 \end{definition}

The set of $\Gamma$-equivalence classes forms a covering $\WT{\Omega} ( L_0, L_1;l_0)$ associated to the the following subgroup  
of $\pi_1 (\Omega (L_0, L_1;l_0))$:
$${\rm{Ker}} (I_{\omega},I_\mu) = {\rm{Ker}} I_{\omega} \cap {\rm{Ker}} I_\mu.$$ 
Hence the deck transformation group $\Pi( L_0, L_1;l_0)$ of the covering 
\begin{equation}\label{def:pi}
\pi :\WT{\Omega} ( L_0, L_1;l_0) \to \Omega ( L_0, L_1;l_0),
\end{equation}
is given by the quotient $\pi_1 (\Omega (L_0, L_1;l_0))/{\rm{Ker}} (I_{\omega},I_\mu)$.

By the definition of $\Gamma$-equivalences, the homomorphisms $I_{\omega}$ and $I_{\mu}$ give rise to 
\begin{equation}\label{Emuabel}
(E, \mu) : \Pi( L_0, L_1;l_0) \to \RR \times \ZZ
\end{equation}
defined by $E(g) = I_\omega (C)$ and $\mu(g) = I_{\mu} (C)$. Here, the element $g \in \Pi(L_0, L_1;l_0)$ is represented by a map $C: S^1 \times [0,1] \to M$ satisfying \eqref{bdycond} above.
Note that $\Pi(L_0, L_1;l_0)$ is an abelian group. Indeed \eqref{Emuabel}
 gives  an injective homomorphism onto its image by the definition of $\Gamma$-equivalence.
Define $\Lambda^k(L_0, L_1 ; l_0)$ whose element is a formal sum
$$\sum_{\substack{\beta \in \Pi (L_0,L_1;l_0) \\ \mu_{l_0} (\beta)=k }} a_\beta [\beta]$$
with $a_{\beta} \in R$. It should satisfy the condition that the set $ \{ \beta | a_\beta \neq 0, E(\beta) < C \}$ is finite for any $C$.

From the construction, $\pi^\ast \alpha$ is exact (where $\pi$ is the covering map in \eqref{def:pi}). The Floer action functional $\mathcal{A}$ on $\WT{\Omega} ( L_0, L_1;l_0)$ is defined by
$$\mathcal{A} ([w,l]) = \int w^{\ast} \omega$$
for an element $[w,l]$ in $\WT{\Omega} ( L_0, L_1;l_0)$. Then, the direct computation shows that
$$\pi^\ast \alpha = - d \mathcal{A}.$$
i.e. the Floer action functional is a Morse function on the Novikov cover which integrates $- \pi^\ast \alpha$. 
In short, Floer-Novikov theory is Morse theory of $\mathcal{A}$ on $\WT{\Omega} ( L_0, L_1;l_0)$
with Novikov coefficients $\Lambda^R (L_0, L_1 ; l_0)$. Critical points of $\mathcal{A}$
are given by $[w, l_p] \in \WT{\Omega} ( L_0, L_1;l_0)$ where $l_p$
is a constant path at $p \in L_0 \cap L_1$. 


Like in \cite{FOOO}, we work with the universal Novikov ring $\Lambda_{nov}$ and its subring $\Lambda_{0,nov}$ as coefficient rings for Floer cohomology groups: 
$$\Lambda_{nov} (R) = \left\{ \sum_{i=0}^{\infty} a_i T^{\lambda_i} \mid a_i \in R, \lambda_i \in \RR, \lim_{i \to \infty} \lambda_i =\infty \right\},$$
$$\Lambda_{0,nov} (R) = \left\{\sum_{i=0}^{\infty} a_i T^{\lambda_i} \in \Lambda_{nov} (R) \mid \lambda_i \in \RR_{\geq 0 } \right\}.$$
Recall for a single Lagrangian $L$, the associated Novikov ring is defined by
$$\Lambda(L) = \left\{ \left. \sum_{\beta \in \Pi (L)} a_\beta [\beta]  \, \right| \, \# \{ \beta | a_\beta \neq 0, E(\beta) < C \} \,\,\mbox{is finite for each} \,\,C \right\} .$$
The homomorphism $I_L : \Lambda(L) \to \Lambda_{nov}$ given by
\begin{equation}\label{formalsum}
I_L \left( \sum_{\beta \in \Pi(L)} a_\beta [\beta] \right) = \sum_{\beta} a_\beta T^{\int_\beta \omega}
\end{equation}
identifies $\Lambda(L)$ with a subring of the universal Novikov ring. We use the same identification as \eqref{formalsum} to transfer $\Lambda^k (L_0, L_1;l_0)$ into the universal Novikov coefficients. 

\subsection{Orientation spaces for Lagrangian Floer theory}\label{oriline}
We recall orientation spaces for Lagrangian Floer theory.
The Cauchy-Riemann (CR for short) operators associated to
generators of the Floer-Novikov complex were introduced in \cite{FOOO} for Bott-Morse Lagrangian Floer theory.
The index spaces (Definition \ref{def:orspace}) of these operators will be called orientation spaces.
The orientation spaces are keys to analyze Fredholm indices, and orientations in the Lagrangian Floer theory. As so, it is natural to expect that they play a important role in understanding group actions on Lagrangian Floer theory.

We fix a reference path $l_0$, and specify pre-fixed choices of data in what follows.
For the chosen reference path $l_0:[0,1] \to M$, we write $l_0(0) =p_0 \in L_0$, $l_0(1) = p_1 \in L_1$.
Along the chosen reference path, we choose 
a fixed path $\lambda_{l_0}$($t \mapsto \lambda_{l_0}(t)$), where $\lambda_{l_0}(t)$ is an element of oriented Lagrangian subspace in $T_{l_0 (t)} M$ whose end points are given by
$\lambda_{l_0}(0) = T_{p_0}L_0$ and
$\lambda_{l_0}(1)=T_{p_1}L_1$.
We denote
$$\widetilde{\lambda}_{l_0}:=\bigcup_{t \in [0,1]} \{t\} \times \lambda_{l_0}(t) \to [0,1].$$
Before we explain additional choices needed to specify spin structure of the Lagrangian bundle,
we recall the definition of the associated CR operator. Consider a critical point  $[w, l_p] \in Cr(L_0, L_1 ; l_0)$ for $p$ a point in the Lagrangian intersection $L_0 \cap L_1$.  We choose a representative $(w, l_p)$ (instead of its $\Gamma$-equivalence class). We want to define the orientation space $\Theta^-_{(w, l_p)}$ at the end.

The map $w$ produces a path $\lambda_{w,\lambda_{l_0}}$ of Lagrangian subspaces in $TM$ 
from $T_{p} L_1$ to $ T_{p} L_1$ by concatenating $\lambda_{l_0}$ and Lagrangian paths in top and bottom edge of $\partial w$. More precisely, for
$$(s,t) \in (\{0\} \times [0,1] ) \cup ([0,1] \times \{0,1\}) \stackrel{homeo}{\cong} [0,1],$$
we get a path $\lambda=\lambda_{w, \lambda_{l_0}}$ such that,
\begin{equation}\label{eq:deflambda}
\lambda(0,t)= \lambda_{l_0} (t), \qquad \lambda (i,s) = T_{w(i,s)} L_i \;\; \textrm{for}\; i=0,1,
\end{equation}
as drawn in Figure \ref{[0,1]}. 

\begin{figure}[h]
\begin{center}
\includegraphics[height=1.6in]{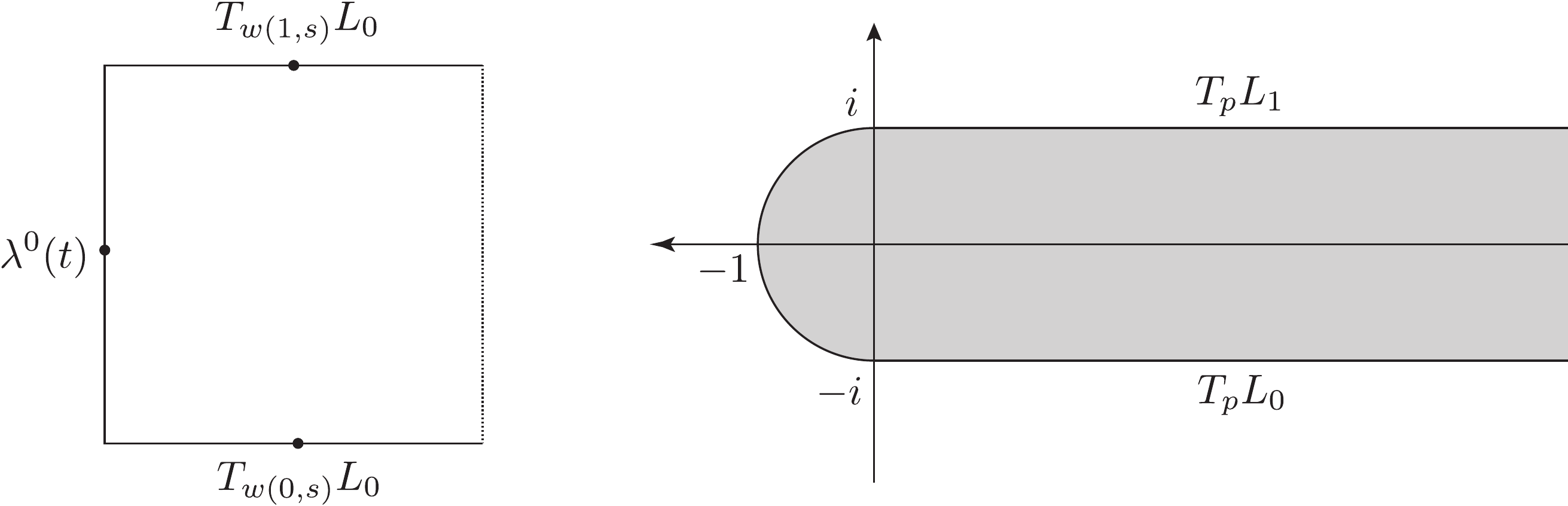}
\caption{path $\lambda_{w, \lambda^0}$}\label{[0,1]}
\end{center}
\end{figure}

Since $[0,1]^2$ is contractible, we can trivialize $w^\ast TM$
$$w^{\ast} TM \cong [0,1]^2 \times T_p M$$
so that $\lambda$ is a path in the oriented Lagrangian Grassmannian of $T_p M$. Put
$$Z_- = \{ z \in \CC \, | \, |z| \leq 1 \} \cup \{ z \in \CC \, | \, {\rm{Re}} \, z \geq 0, \, |{\rm{Im}} \, z| \leq 1 \}.$$
For $\lambda=\lambda_{w, \lambda_{l_0}}$, one defines a CR operator by
$$\bar{\partial}_{\lambda, Z_-}  : W^{1,p}_\lambda (Z_- ; T_p M) \to L^p (Z_- ; T_p M \otimes \Lambda^{0,1} (Z_-)),$$
whose Lagrangian boundary condition comes from $\lambda$. Here, $W^{1,p}_\lambda (Z_- ; T_p M)$ is the Banach space of  $L^{1,p}$ maps $\zeta_- :  Z_- \to T_p M$ with weighted norms:
\begin{enumerate}\label{def:zminusbdy}
\item $\zeta_- (\tau, i ) \in T_p L_i$, where we use $z=\tau + it$
\item $\zeta_- (z_-) \in \lambda (t)$, where $z_- (t) = e^{\pi i (-1/2 +t)} \in \partial Z_-$
\item $\displaystyle\int_{Z_-} e^{\delta |\tau|} ( | \nabla \zeta_- |^p + |\zeta_-|^p ) d \tau dt < \infty$.
\end{enumerate}
It is well known that $\bar{\partial}_{\lambda, Z_-} $ is a Fredholm operator. See \cite{FOOO} for more details.

 Finally, one assigns an orientation space to each generator $[w,l_p]$ as follows.
 \begin{definition}\label{def:orspace}
The orientation space  $\Theta_{(w,l_p)}^-$ for a pair $(w,l_p)$ is defined as
$$\Theta_{(w,l_p)}^- :=  {\rm{det}} (\bar{\partial}_{\lambda, Z_-}) = (\wedge^{top}  {\rm{Coker}} \, \bar{\partial}_{\lambda, Z_-})^* \otimes \wedge^{top} {\rm{Ker}} \,  \bar{\partial}_{\lambda, Z_-}.$$
\end{definition}

The canonical orientation on the moduli space associated to the CR operators with Lagrangian boundary conditions is determined from the spin structures of
the Lagrangian sub-bundles.  Later,
we will deal with isomorphisms between various orientation spaces, and
to find a canonical isomorphism, we need to also specify spin structures of the Lagrangian bundles
on the boundary components. Without these, there would be an ambiguity of signs whenever we try to compare, identify and glue orientation spaces. Signs are indeed very crucial information (for example different signs for the group action change the invariant sets).

As in chapter 8 of \cite{FOOO}, we fix the following choices of spin data.
Also we choose and fix a trivialization $\sigma:[0,1] \times \mathbb{R}^n \to \widetilde{\lambda}_{l_0} $, which
also provides a trivialization of $T_{p_0}L_0$ and $T_{p_1}L_1$.
The latter trivialization $\sigma_{t=i}: \mathbb{R}^n \to T_{p_i}L_i$
gives an embedding $SO(n)$ into the fiber of $P_{SO}(L_i)$ at $p_i$.
We choose one of its lifts $\iota_i : Spin(n) \to P_{spin}(L_i)$ between
two possible choices.

The trivialization $\sigma$ of $ \widetilde{\lambda}_{l_0} $ also
can be lifted, and we make a choice
\begin{equation}\label{spinl0}
\widetilde{\sigma}: [0,1] \times Spin(n) \cong P_{spin}(\lambda_{l_0}).
\end{equation}
Now, we can use $\widetilde{\sigma}, \iota_0, \iota_1$ to
glue $P_{spin}(L_0), P_{spin}(\lambda_{l_0}),  P_{spin}(L_1)$
at $p_0$ and $p_1$.
Namely, we glue $P_{spin}(L_i)$  and $P_{spin}(\lambda_{l_0}),$
using $\iota_i \circ \widetilde{\sigma}^{-1}$ at $p_i$ for $i=0,1$.

We emphasize that with these pre-fixed data, we have the canonical spin structure of $\lambda$ for the generator $(w,l_p)$:
for $T_{w(\gamma(t))}L_0$ with $0 \leq t \leq \frac{1}{3}$, we have $P_{spin}(L_0)$
and for $T_{w(\gamma(t))}L_1$ with $\frac{2}{3} \leq t \leq 1$, we have $P_{spin}(L_1)$,
and for $\lambda_{l_0}(t)$ for $\frac{1}{3} \leq t \leq \frac{2}{3}$, we have $P_{spin}(\lambda_{l_0})$. At $t=\frac{1}{3}$ (resp. $t = \frac{2}{3}$),
we glue $P_{spin}(\lambda_{l_0})$ and $P_{spin}(L_0)$ (resp. $P_{spin}(L_1)$)
by $\iota_0 \circ \WT{\sigma}^{-1}$ (resp. $\iota_1 \circ \WT{\sigma}^{-1}$).

\subsection{Gluing theorems}\label{subsec:Gluingthm}
We need to recall  some elementary facts on determinant (index) spaces of CR operators. See \cite[Section (11c)]{Se} for more explanations.

First, we consider the gluing of CR operators on $Z_-$ and the one on holomorphic strips, and investigate the relation between orientation spaces before and after the gluing. Let $S_1$ be given by $Z_-$, and let $S_2$ be a strip $\R \times [0,1]$.
Consider a $J$-holomorphic strip connecting two intersection points $p$ and $q$ and the induced Lagrangian sub-bundle data $\lambda'$ along $\partial S_2 = \R \times \{0,1\}$, which give
rise to a weighted CR operator $\bar{\partial}_{\lambda', S_2}$.

The standard gluing theorem for the weighted CR operators $\bar{\partial}_{\lambda, S_1}$ for $S_1$, and $\bar{\partial}_{\lambda', S_2}$ for $S_2$ says that 
\begin{equation}
 {\rm{det}} (\bar{\partial}_{\lambda,S_1})  \otimes  {\rm{det}}(\bar{\partial}_{\lambda', S_2})
\cong {\rm{det}} (\bar{\partial}_{\lambda \sharp \lambda', S_3}),
\end{equation}
where $S_3 = Z_-$ and by $\lambda \sharp \lambda'$, we mean the Lagrangian sub-bundle
data along $\partial S_3$ obtained by concatenating that of $\R \times \{1\} \subset \partial S_2$ backwards,
that of $S_1$, and that of $\R \times \{0\} \subset S_2$. (See (a) of Figure \ref{fig:SSS}.)

Hence, the determinant line bundle ${\rm{det}}(\bar{\partial}_{\lambda', S_2})$ for
the holomorphic strip is obtained as a tensor product  (compare with \eqref{eq:moropq})
$${\rm{det}} (\bar{\partial}_{\lambda,S_1}) \otimes {\rm{det}} (\bar{\partial}_{\lambda \sharp \lambda', S_3}).$$

\begin{figure}[h]
\begin{center}
\includegraphics[height=1.7in]{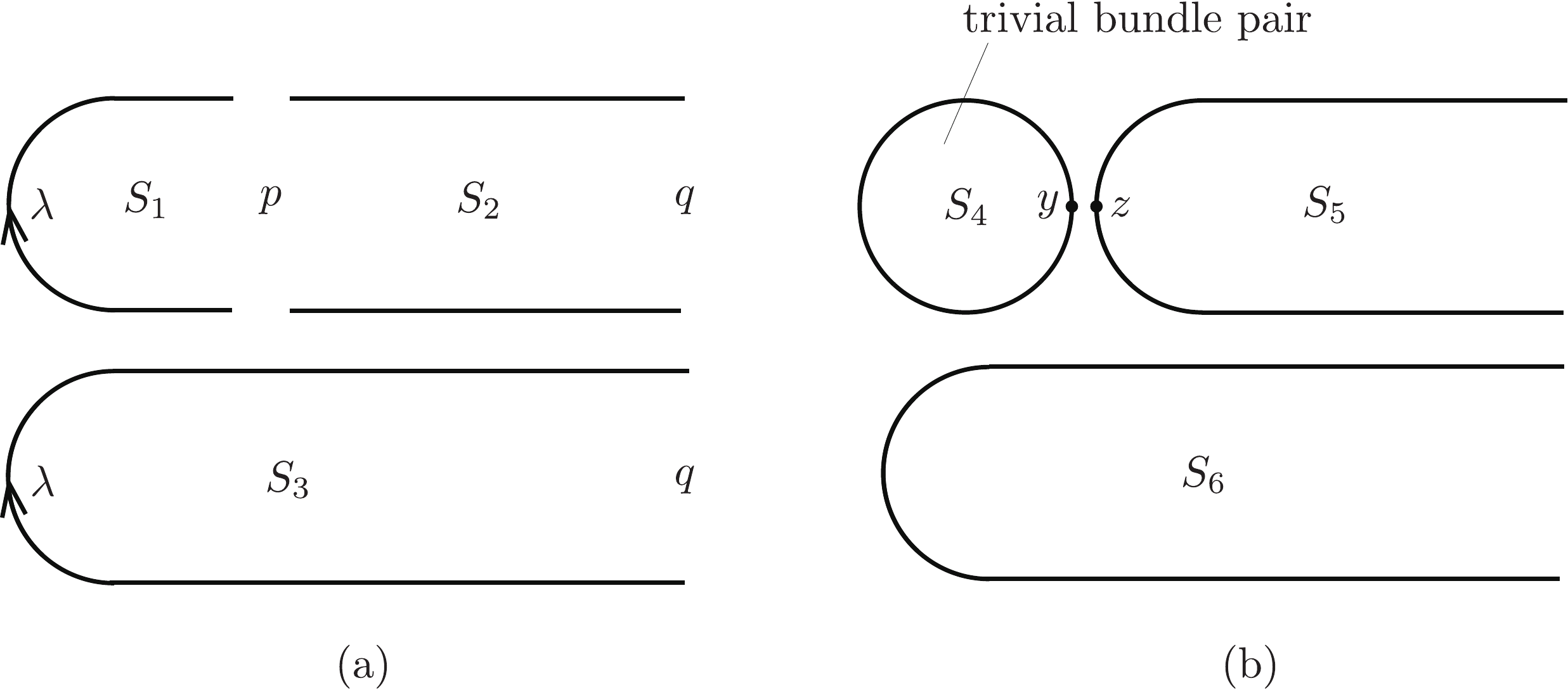}
\caption{(a) Gluing of $S_1$ and $S_2$ (b) Gluing of $S_4$ and $S_5$}\label{fig:SSS}
\end{center}
\end{figure}

In the discussion below we need another type of gluing result.
Let $S_4$ be a disc with a trivial symplectic bundle $E_4$ and a Lagrangian sub-bundle $F_4$
(which is not necessarily trivial) along $\partial S_4$
and consider $S_5 = Z_-$ and a trivial symplectic bundle $E_5$ and a Lagrangian sub-bundle $F_5$
along $\partial S_5$
which is fixed sub-bundle for $\{z\in Z_- \mid {\rm Re}\, (z) \geq 0\}$ as in \eqref{def:zminusbdy}.

Consider boundary marked points $y \in \partial S_4$, and $z \in \partial S_5$ and
an identification $E_{4,y} \cong E_{5,z}$ which identifies $F_{4,y}$ with $F_{5,z}$.
Then determinant spaces of the associated CR operators satisfy
\begin{equation}\label{zindexformula2}
 {\rm{det}} (\bar{\partial}_{S_4})  \otimes  {\rm{det}}(\bar{\partial}_{S_5})
\cong {\rm{det}} (\bar{\partial}_{S_6}) \otimes \wedge^{top} F_{5,z},
\end{equation}
where $S_6 = Z_-$ with a Lagrangian sub-bundle  data along $\partial S_6$ obtained by 
concatenating that of $\R \times \{1\} \subset \partial S_5$ backwards,
that of $S_4$ counter-clockwise, and that of $\R \times \{0\} \subset S_5$. (See (b) of Figure \ref{fig:SSS}.) Here, the  new term $\wedge^{top} F_{5,z}$ appears because the elements of the kernel
of Cauchy-Riemann problem on $S_4$ and $S_5$ should have the same boundary value in $F_{4,y} = F_{5,z}$ at $y=z$. 

Later, we will need the following lemma of Seidel.
\begin{lemma}\cite[Lemma 11.17]{Se}\label{lem:se11}
Let $\rho$ be a loop in the Lagrangian Grassmannian. Then a choice of pin structure determines
an isomorphism $$det(\OL{\partial}_{D,\rho}) \cong \wedge^{top}\rho(0).$$
\end{lemma}
One way to interpret this lemma is to say that the choice of pin structure gives rise to a canonical orientation of the determinant space
(relative to $\rho(0)$).
This lemma will be used in  \eqref{zindexformula2} for a Lagrangian loop $\rho$ of $\partial S_4$ so that we
have a canonical isomorphism between ${\rm{det}}(\bar{\partial}_{S_5})$ and ${\rm{det}} (\bar{\partial}_{S_6})$.


Let us briefly review the argument in \cite{Se} for Lemma \ref{lem:se11} in the case of a Maslov index zero loop. First
consider $\rho$ to be a constant loop with a trivial spin structure. Then,
the kernel is identified with $\rho(0)$ whereas the cokernel is trivial. For non-constant $\rho$ of index zero
with a trivial spin structure, we consider a homotopy from $\rho$ to a trivial one to obtain the identification $det(\OL{\partial}_{D,\rho}) \cong \wedge^{top}\rho(0).$
If we equip $\rho$ with a non-trivial spin structure, the identification is changed by reversing the sign.

\subsection{Spin-$\Gamma$-equivalences and Novikov theory}\label{subsec:SpinGamma}
We have defined the closed $1$-form $\alpha$ \eqref{def:alpha} on $\Omega(L_0,L_1;l_0)$,
and considered the Novikov covering
$\WT{\Omega}(L_0,L_1;l_0)$ by using $\Gamma$-equivalences on
bounding surfaces of paths to define action functional $\CA$. Now, we introduce a new notion of spin-$\Gamma$-equivalences.
This will be a finer equivalence relation than the $\Gamma$-equivalence, and hence it will have a possibly bigger deck transformation group, which results in the bigger Novikov ring $\Lambda^{spin}$.
But later, we will reduce the coefficient ring to the usual universal Novikov ring with help of a homomorphism $\Lambda^{spin}$ to $\Lambda$. In this process, $\Gamma$-equivalent, but
not spin-$\Gamma$ equivalent elements will be identified with opposite signs. 
 Hence, the Floer complex
is essentially the same complex as that of \cite{FOOO}.

Spin structures are essentially taken into account also in the definition of
Floer complexes in \cite{FOOO}, but the usage of spin $\Gamma$-equivalence here gives more canonical and explicit way of merging the spin structures of Lagrangian bundle data at critical points into the Floer complexes. 
Namely, $\Gamma$-equivalences are not very natural to assign orientations in the sense that $\Gamma$-equivalent generators with different
spin structures on the associated Lagrangian bundle should rather be thought to give the same generator {\em with opposite sign}. In \cite{FOOO}, the choice of spin structures is carried as extra data, but we would like to combine it with the $\Gamma$-equivalence by defining a finer equivalence relation. We emphasize that what to be done is a reformulation of orientation analysis of \cite{FOOO} which is written very carefully in Section 8.1.3.
Introduction of spin-$\Gamma$-equivalences provides a way to understand the Novikov complex and its orientation scheme at once, and we explain how to adapt orientation scheme of \cite{FOOO} into this setting.

Detailed constructions are to be given from now on.
We first define the spin-$\Gamma$-equivalence. If two pairs
$(w,l)$ and $(w',l)$ have the same end points $l$, then induced Lagrangian paths
$\lambda_{w,l}$ and $\lambda_{w',l}$ have the same end values. 
They can be glued to each other by first inverting $\lambda_{w',l}$ and identifying $\lambda_{w,l}(i)$ with $\lambda_{w',l}(i)$ for $i=0,1$. As a result, we obtain a Lagrangian bundle  over
a circle $S^1$, which we denote by $\lambda \sharp \lambda'_-$.
Their spin structures also can be glued too. 


Note that $\lambda \sharp \lambda'_-$ is a Lagrangian sub-bundle over a rectangle. On the horizontal edges, it restricts to the tangent bundle of Lagrangian submanifolds with
 given spin structures, $P_{spin}L_0$ and $P_{spin}L_1$. On vertical edges, it becomes $\lambda_{l_0}$ over the
base path $l_0$ with the spin structure $P_{spin}(\lambda_{l_0})$. How to glue these spin structures at vertices by $\iota_i \circ \widetilde{\sigma}^{-1}$ were explained 
in the paragraph after \eqref{spinl0}. This provides the desired canonical spin structure 
on $\lambda \sharp \lambda'_-$.

If $(w,l)$ and $(w',l)$ are $\Gamma$-equivalent, the Maslov index of $\lambda \sharp \lambda'_-$ vanishes, and hence the trivial spin structure of $\lambda \sharp \lambda'_-$ makes sense (see the explanation before Assumption \ref{as:spin} for the definition of trivial spin structure).

\begin{definition}\label{def:spingamma}
$(w,l)$ and $(w',l)$ are said to be {\em  spin-$\Gamma$-equivalent} if 
they are $\Gamma$-equivalent, and the induced  spin structure on $\lambda \sharp \lambda'_-$
over $S^1$ is a {\em trivial spin structure}.
The set of spin-$\Gamma$-equivalence classes is denoted by $\WT{\Omega} ( L_0, L_1;l_0)^{sp}$. 
\end{definition}

The spin-$\Gamma$ equivalence is linked to the covering $\WT{\Omega} ( L_0, L_1;l_0)^{sp}$ in the same way that we had the $\Gamma$-equivalence from a covering $\WT{\Omega} ( L_0, L_1;l_0)$ \eqref{def:pi}. In addition to $(I_{\omega}, I_{\mu})$ \eqref{eq:alphamu}, we have a homomorphism
$sp: {\rm Ker}\; \mu \to \Z/2$, which is defined by 0 for the trivial spin structure 
on $C \in \pi_1 \left( \Omega (L_0,L_1;l_0) \right)$ with $I_\mu (C) =0$, and $1$ otherwise. 
In fact, by using Lemma \ref{lem:se11}, we can extend $sp$ to the following map
\begin{equation}\label{def:sploop}
sp:  \pi_1(\WT{\Omega} ( L_0, L_1;l_0)) \to \Z/2.
\end{equation}
Let $C:S^1 \times [0,1] \to M$ be a cylinder representing an element of $\pi_1(\WT{\Omega} ( L_0, L_1;l_0))$ i.e. the boundary $S^1 \times \{0\}$
maps to $L_0$ and the boundary $S^1 \times \{1\}$ to $L_1$ by $C$. We cut the cylinder $S^1 \times [0,1]$ along $\{-1\} \times [0,1]$
to obtain a disc $D:D^2 \to M$, and choose a path of Lagrangian subspaces from $T_{C(-1,0)}L_0$ to $T_{C(-1,1)}L_1$ which
connects corresponding spin structures also. In this way, we get a Lagrangian loop along $\partial D^2$ with a spin structure from $C$.

With help of Lemma \ref{lem:se11}, we can determine whether the canonical orientation of the associated determinant space equals the
given orientation of $T_{C(1,0)}L_0$ or not. If it is equal, then we define $sp$ of $C$ to be zero, and one otherwise. Note that $sp (C)$ does not depend on the choice of paths connecting $T_{C(-1,0)}L_0$ and $T_{C(-1,1)}L_1$ since this path contributes to $sp$ twice after cutting.

The covering  $\WT{\Omega} ( L_0, L_1;l_0)^{sp}$ corresponds to the following subgroup  
of  $ {\rm{Ker}} I_\omega \cap {\rm{Ker}} I_\mu \subset \pi_1 (\Omega (L_0, L_1;l_0))$:
$${\rm{Ker}} (I_{\omega}, I_\mu, sp) = {\rm{Ker}} I_{\omega} \cap {\rm{Ker}}\, sp. $$ 
We denote by $\Pi( L_0, L_1;l_0)^{sp}$ the group of deck transformations of the covering 
\begin{equation}\label{def:pi2}
\pi^{sp} :\WT{\Omega} ( L_0, L_1;l_0)^{sp} \to \Omega ( L_0, L_1;l_0).
\end{equation}
It is not hard to check that 
${\rm{Ker}} (I_{\omega} ,I_\mu, sp)$ is a normal subgroup of $\pi_1 \left( \Omega (L_0,L_1;l_0) \right)$ and hence, the deck transformation group is given by the quotient group $\pi_1 (\Omega (L_0, L_1;l_0))/{\rm{Ker}} (I_{\omega},I_\mu,sp)$.
Note that $\Pi(L_0, L_1;l_0)^{sp}$ is also an abelian group since it is a quotient group of the abelian group $\Pi (L_0,L_1;l_0)$ \eqref{Emuabel}.

The Novikov ring for the coverings $\pi^{sp}$ is given as below.
Here, $R$ is a commutative ring with unity which includes $\Q$
(we need rational coefficients for multisections).

\begin{definition}\label{def:novl01}
$\Lambda_k^R (L_0, L_1 ; l_0)^{sp}$ is the set of all formal sums
$$\sum_{\substack{h \in \Pi(L_0, L_1 ; l_0)^{sp} \\ \mu(h) = k}} a_h [h],$$
such that $a_h \in R$ and 
$$| \{ h \in \Pi (L_0, L_1;l_0)^{sp} \,| \, E(g) \leq C, a_h \neq 0\}| = \infty.$$
We write
$$\Lambda^R (L_0, L_1 ; l_0)^{sp} = \bigoplus_k \Lambda_k^R (L_0, L_1 ; l_0)^{sp}.$$

\end{definition}
Compare it with the Novikov ring $\Lambda^R (L_0, L_1 ; l_0)$ defined in Subsection \ref{subsec:FNtheory}.

\begin{definition}
The homomorphism $j$
\begin{equation}\label{defflcmx4}
j:\Lambda^R (L_0, L_1 ; l_0)^{sp} \to \Lambda^R (L_0, L_1 ; l_0),
\end{equation}
is defined as the map induced from the  natural projection $\Pi (L_0,L_1; l_0)^{sp} \to \Pi (L_0,L_1;l_0)$ with an additional multiplication of $sp$ 
defined in \eqref{def:sploop}.
\end{definition}

Later, we will use the coefficient ring $\Lambda^R (L_0, L_1 ; l_0)$ ( or the universal Novikov ring)
using this homomorphism $j$.

Now, new (naive) generators of the Floer complex are given by
\begin{equation}\label{defflcmx3}
Cr (L_0, L_1;l_0)^{sp} = \left\{ [w,l_p] \in \left. \WT{\Omega} (L_0, L_1;l_0)^{sp} \, \right| \, p \in L_0 \cap L_1 \right\},
\end{equation}
where $l_p$ is the constant path at $p$.
If an element $c \in \Pi (L_0, L_1;l_0)^{sp}$ is represented by the loop $C$ of $\pi_1(\Omega(L_0, L_1;l_0)^{sp})$, then the action of $\Pi (L_0, L_1;l_0)^{sp}$ on $Cr (L_0, L_1;l_0)$ 
is defined by 
$$c \cdot [w,l_p] =[C \star w,l_p].$$
From this action, one can naturally define a $R[\Pi (L_0, L_1;l_0)^{sp}]$-module which is generated by elements in $Cr (L_0, L_1;l_0)^{sp}$. The actual definition of the cochain complex should also involve the orientation spaces associated to the generators in \eqref{defflcmx3}.
 More precisely, consider the formal sum
$$\mathfrak{r} = \sum_{\mu([w,l_p])=k} a_{[w,l_p]} x_{[w,l_p]}$$
where $a_{[w,l_p]} \in R$ and $x_{[w,l_p]} \in \left|\Theta_{[w,l_p]}^- \right|_R $. 

\begin{remark}
For the precise definition of $\Theta_{[w,l_p]}^-$, see Subsection \ref{oriline}.  Spin-$\Gamma$-action on orientation spaces
is defined using the isomorphism of orientation spaces in \eqref{zindexformula2}. We omit the details, and
refer readers to Section \ref{GactFloerNov} where similar operation on orientation spaces are explained carefully.
\end{remark}

Let 
$$supp(\mathfrak{r}):=\left\{ [w,l_p] \in \left. Cr(L_0, L_1;l_0)^{sp} \, \right| \, a_{[w,l_p]} \right\} \neq 0.$$
We call the sum $\mathfrak{r}$ a Floer cochain of degree $k$ if
$$\# (supp(\mathfrak{r} \cap \{ [w,l_p] \,|\, \mathcal{A} ([w,l_p]) \leq \lambda \})) < \infty$$
for any $\lambda \in \RR$.

 We denote the set of all Floer cochains of degree $k$ by $CF_R^k (L_0, L_1;l_0)^{sp}$ and we define 
\begin{equation}\label{defflcmx2}
CF_{R,l_0}^{sp,\ast}(L_0,L_1) := \bigoplus_{k} CF_R^k (L_0, L_1; l_0)^{sp}.
\end{equation}
From the $\Pi (L_0, L_1;l_0)^{sp}$-action $Cr (L_0, L_1;l_0)^{sp}$, we get a natural $\Lambda^R(L_0, L_1;l_0)^{sp}$-module structure of $CF_{R,l_0}^{sp,\ast}(L_0,L_1)$.
 
\begin{definition}
 The  Floer complex of the pair $(L_0, L_1)$ is defined as
\begin{equation}\label{defflcmx1}
CF_{R,l_0}^{\ast}(L_0,L_1) :=CF_{R,l_0}^{sp,\ast}(L_0,L_1) \otimes_{\Lambda^R(L_0, L_1;l_0)^{sp}} \Lambda^R(L_0, L_1;l_0)
\end{equation}
 using the homomorphism $j$ of \eqref{defflcmx4}.
\end{definition}
We will use the universal Novikov coefficients by considering ring homomorphism
$${\Lambda^R(L_0, L_1;l_0)^{sp}} \to \Lambda^R(L_0, L_1;l_0) \to \Lambda_{nov} (R)$$
in many applications such as morphism spaces of Fukaya categories. This will enable us to compare various of modules arising from Lagrangian Floer theory over the single coefficient field.

We will revisit the $\Lambda^R(L_0, L_1;l_0)^{sp}$-module structure on $CF_{R,l_0}^{sp,\ast}(L_0,L_1)$ with more details in Subsection \ref{Galpha} and discuss the new Novikov ring which fits into orbifold Floer cohomology in Subsection \ref{subsec:newNov}.

\section{Spin profiles}
Consider $G$-invariant Lagrangian submanifolds and suppose further that they admit $G$-invariant spin structures. Although spin structures are $G$-invariant, group actions on the frame bundles of Lagrangians do not lift in general to the one on spin bundles. 
For example, consider a reflection $A \in O(n)$, and its lift $\WT{A} \in Pin(n)$.
Then, $A^2=Id$, but we have $\WT{A}^2 = (-1)^{r(r-1)/2}e$ where $r$ is the rank of $Ker(Id + A)$ \cite[Lemma 11.4]{Se}. Hence, the $\Z/2$-action given by $A$ on $\R^n$ induces the action on $Pin(n)$ only up to signs, and the sign error indeed has an important geometric origin. For the general group actions, we will regard these sign differences as a group cohomology class, and call them spin profiles.

\subsection{Spin structures and spin profiles}\label{sec:spins}
Before presenting the definition of spin profiles, we briefly recall basic facts on spin structures for reader's convenience.

Let $Cl(\R^n)$ be the Clifford algebra. There is a natural embedding $\R^n \to Cl(\R^n)$ and vectors satisfy the relation $v^2 = ||v||^2 e$ for the unit $e \in Cl(\R^n)$.
Then, $Pin(n) \subset Cl(\R^n)^\times$ is defined as the multiplicative subgroup generated by $S^{n-1} \subset \R^n$, and $Spin(n) \subset  Cl(\R^n)^\times$ the subgroup of $Pin(n)$ whose elements are given by products of even number of generators in $S^{n-1}$. Both of homomorphisms $Pin(n) \to O(n)$ and $Spin(n) \to SO(n)$ are induced by the twisted adjoint action on $\R^n$, sending $w \mapsto -vwv$ for each generator $v \in S^{n-1}$.

A spin structure on $L$ is a principal $Spin(n)$ bundle
$\pi:P_{spin}(L) \to L$ which is a fiberwise double cover of the principal orthonormal frame bundle
$\pi':P_{SO}(L) \to L$. Let $pr:P_{spin}(L) \to P_{SO}(L)$ be the covering map such that $\pi = \pi' \circ pr$. $H^1(L;\Z/2)$ acts on the set of isomorphism classes of spin structures on $L$ and the action is free and transitive. Thus, it is an affine space modeled on $H^1(L;\Z/2)$.
In fact, Milnor \cite{Mi} has shown that a spin structure of a vector bundle on a simplicial complex can be understood as a trivialization of the vector bundle on the 1-skeleton, which can be extended to the 2-skeleton. Then, the action of $H^1 (L;\Z/2)$ can be interpreted as changing the trivialization of the vector bundle along the non-trivial loops in the 1-skeleton.
For example, homotopy classes of trivializations of a rank $n$ vector bundle on $[0,1]$ (with fixed trivialization at the end points) could differ by a loop of $\pi_1(SO(n))$, but trivial and non-trivial loops have different lifts into $Spin(n)$. 

We are particularly interested in the vector bundle on a circle. If the vector bundle $F \to S^1$ is spin, then there are two choices of spin structures. However, there is no preferred choice unless we know that the
vector bundle $F$ can be extended to $F' \to D^2$ with $\partial D^1 =S^1$. If $F$ has extension, there is a unique spin structure on $F \to S_1$ which extends to $F' \to D^2$. We call this preferred choice the {\em trivial spin structure}.
In our setting, any loop in the Lagrangian Grassmannian of Maslov index zero is indeed contractible. For such a loop, we have a choice of the trivial spin structure.

Let us first make the following necessary assumption.
\begin{assumption}\label{as:spin}
Let $L$ be a connected spin (orientable) Lagrangian submanifold with a choice of spin structure and suppose the $G$-action on $L$ is orientation preserving. We assume $G$ preserves the isomorphism class of the given spin structure. i.e. a pulled-back spin structure by $g$ is isomorphic to the original one for each $g \in G$. Such spin structures are said to be $G$-invariant.
\end{assumption}

\begin{remark}
We do not consider the case that spin structure is not $G$-invariant.
\end{remark}
\begin{remark}
Recall that $L$ is said to be relatively spin, if there exist a class $w \in H^2(M,\Z/2)$ such that
the restriction of $(w + w_2(L))$ to $L$ is zero, and hence there exist a bundle $V$ on 3-skeleton of $M$,
such that $(V \oplus TL)|_L$ becomes spin. Floer homology between two relatively spin Lagrangian
submanifolds can be defined only if they become relatively spin with the same $V$.

With $G$-actions, we can work with $w \in H^2_G(M,\Z/2)$ so that there exist
a $G$-bundle $V$ on 3 skeleton of $M$ with $(V \oplus TL)|_L$ being spin and
$G$-invariant.  We leave further details to the readers, and consider the case that $L$ is spin from now on.
\end{remark}


Given a $G$-action on $L$, let us denote by $A_g:L \to L $ its action $A_g(x) = g \cdot x$ for
$g\in G$ and $x \in L$, and let a $G$-invariant metric be given on $M$.
We denote the orthonormal frame bundle of $L$ by $P_{SO}(L)$, which 
has an induced $G$-action denoted by $A_g^{SO}:P_{SO}(L) \to P_{SO}(L)$.
In fact, $A_{g}^{SO}$ provides an isomorphism 
\begin{equation}\label{eq:agso}
(A_{g^{-1}})^*P_{SO}(L) \to P_{SO}(L).
\end{equation}

Let us fix a spin structure of $L$ satisfying the above assumption \ref{as:spin}. 
We would like to lift  $A_g^{SO}:P_{SO}(L) \to P_{SO}(L)$ to a map  $A_g^{spin}$ between spin bundles
\begin{equation}\label{eq:agspin}
A_g^{spin}:P_{spin}(L) \to P_{spin}(L).
\end{equation}
Such a lift exists by the assumption \ref{as:spin} and the standard covering theory argument. \eqref{eq:agspin}
Indeed, an explicit isomorphism of bundles over $P_{SO(n)}$
$$T_g : (A_{g^{-1}})^\ast P_{spin} (L) \stackrel{\cong}{\to} P_{spin} (L)$$
from \eqref{eq:agspin} can be viewed as a lift of $A_g^{SO}$ in
\eqref{eq:agso}. Note that the lift is not unique, as it can be multiplied by $(-1) \in Spin(n)$ to produce another one. We always choose the lift $A_e^{spin}$ for the identity element $e \in G$ to be the identity map on $P_{spin} (L)$.

%

For $g,h \in G$, we have a natural map given by the composition
\begin{equation}
(A_{(gh)^{-1}})^* P_{spin}(L)  = (A_{g^{-1}})^*(A_{h^{-1}})^* P_{spin}(L)
\stackrel{(A_{g^{-1}})^* T_h}{\to} 
(A_{g^{-1}})^* P_{spin}(L) \stackrel{T_g}{\to} P_{spin}(L).
\end{equation}
But it does not necessarily equal $T_{gh}$, but possibly differs by multplication of $\pm 1 \in Spin(n)$ if the group action on the frame bundle does not lift to the one on the spin bundle.  
Note that this difference $\pm 1$ is locally constant function on
$x \in L$.

\begin{definition}\label{def:spf}
The function $\textnormal{spf}_L:G \times G \to Z/2$ is defined implicitly by the relation
\begin{equation}\label{eq:spf}
T_g \circ (A_{g^{-1}})^*T_h = (-1)^{\textrm{spf}_L(g,h)}T_{gh},
\end{equation}
where  $\pm1$ means the multiplication by $\pm 1 \in Spin(n)$.
\end{definition}

If $L$ has several connected components, we make the following assumption which enables us to get a constant $\textnormal{spf}_L$, not depending on connected components. See (ii) of Subsection \ref{subsec:genspfs} below for more discussions on this issue.
\begin{assumption}\label{ass:Ldisconnspf}
If $L$ is not connected (but $L/G$ is connected), then we assume that component-preserving subgroups of $G$ are the same for every connected components of $L$ and normal in $G$. We assume further that each connected component has the same spin profile with respect to the action of the component-preserving subgroup.
\end{assumption}

We will show that $\textnormal{spf}_L:G \times G \to Z/2$ defines an second group cohomology class which is independent of choices of $T_g$'s. Thus, for example, if we have an induced group action on spin bundles, $\textnormal{spf}_L$ is cohomologous to $0$.

\begin{prop}
$\textnormal{spf}_L$ defines  group cohomology class in $H^2(G;\Z/2)$, which
we call a {\em spin profile} of  $L$.
The function $\textnormal{spf}_L$ depends on the choices of  the lifts $A_g^{spin}, \forall g \in G$,
but the resulting group cohomology class $[\textnormal{spf}_L]$ is independent of these choices.
\end{prop}
\begin{proof}
We first show that $\textnormal{spf}_L$ is indeed a cocycle for the
group cohomology $H^*(G,\Z/2)$. 

Let $G'$ be the set of all liftings of $A_g^{SO}$ for every $g \in G$. Clearly, $G'$ forms a group and the map $G' \to G$ sending $\pm A_g^{spin} \to g$ is $2$ to $1$. The kernel of $G' \to G$ consists of the liftings of the identity, which are $\pm 1 \in spin(n)$. Therefore we get the following exact sequence of groups
\begin{equation}\label{GextG'}
0 \to \ZZ / 2\ZZ \to G' \to G \to 1.
\end{equation}
The choice of lifts $A_g^{spin} \in G'$ for each $g\in G$
defines a section $\sigma:G \to G'$ by setting $\sigma(g)=A_g^{spin}$), which may {\em not} be a group homomorphism. Recall that $\sigma(e) = A_e^{spin}$ is assumed to be the identity.

As $\ZZ / 2 \ZZ$ is abelian, it naturally admits a $G$-module structure by conjugation: for $h \in \ZZ / 2\ZZ$
we write this conjugation action as $h^g = \sigma(g) h \sigma(g)^{-1}$, which is independent of the choice of $\sigma$. 
From general theory of group extension, extensions $G'$ as in \eqref{GextG'} are classified by a cohomology class in $H^2 (G, \ZZ / 2 \ZZ)$  given by so called the factor set 
$$[g,h] = \sigma(g) \sigma(h) \sigma(gh)^{-1}.$$
Since $[g,h] : P_{spin} (L) \to P_{spin} (L)$ covers the identity from $P_{SO} (L)$ to itself, it should be either 
$\pm 1 \in spin(n)$. Thus, the factor set is indeed a map
$$[\, , \,] : G \times G \to \ZZ / 2\ZZ$$
which exactly agrees with $\textnormal{spf}_L$ in the definition \ref{def:spf}.

The following lemma is standard, whose proof is omitted. (see \cite[section 6.6]{W}.)
\begin{lemma} We have
\begin{enumerate}
\item $[g,e]=[e,g] =0$ for all $g \in G$,
\item $[g,h]^f - [fg,h]  +[f,gh] - [f,g] = 0$ for all $f,g,h \in G$.
\end{enumerate}
\end{lemma}
Here, we used additive notation for $\ZZ / 2\ZZ$ to emphasize that it is abelian.
%

From the lemma we get a cohomology class $[\, , \,]$ in $H^2 (G; \ZZ /2\ZZ)$. Now we show that the class does not depend on the choice of a section $\sigma$. Let $\sigma'$ be another section (which is also regarded as a lifting of each $g$ in $G$ to an isomorphism $P_{spin} (L) \to P_{spin} (L)$ covering $g : L \to L$). Then, there exists $\beta : G \to \ZZ / 2 \ZZ$ such that
$$ \sigma'(g) = \beta(g) \sigma(g)$$
for all $g \in G$. The factor set $[\, , \,]'$ corresponding to $\sigma'$ is
\begin{eqnarray*}
[g,h]'&=& \beta(g) \sigma(g) \beta(h) \sigma(h) \sigma(gh)^{-1} \beta(gh)^{-1}\\
&=& \beta(g) + \sigma(g) \beta(h) \sigma(g)^{-1} + \sigma(g) \sigma(h) \sigma(gh)^{-1} + \beta(gh)^{-1}\\
&=& [g,h] + \beta(h)^g - \beta(gh) + \beta(g).
\end{eqnarray*}
Therefore, $[\, , \,]$ and $[\, , \,]'$ are cohomologous and give the same cohomology class $[{\rm spf}_L] \in H^2 (G; \ZZ / 2\ZZ)$.

\end{proof}

\subsection{Comparison of ${\rm spf_L}$ for two different spin structures on $L$}
Once a $G$-invariant spin structure (satisfying assumption
\ref{as:spin}) is fixed, the spin profile of $L$ is uniquely determined as seen in the previous subsection. We now analyze the dependence of the spin profiles $[{\rm spf}_L]$ on the choice of spin structures on $L$.

\begin{prop}\label{prop:diffspinspf}
Let $s$ and $t$ be two $G$-invariant spin structures and denote the
resulting spin profiles by $[{\rm spf_L^s}]$ and $[{\rm spf_L^t}]$, respectively. Then,
there exists an exact sequence
\begin{equation}
H^1 (L;\ZZ /2\ZZ)^G  \to H^2 (G;\ZZ /2\ZZ) \to H^2 ( \pi_1^{orb} [L /G ] ; \ZZ /2)
\end{equation}
such that the first map sends
the difference of two spin structures $s-t \in H^1 (L;\ZZ /2\ZZ)^G$ to the
 difference of two
spin profiles $[{\rm spf_L^s}]- [{\rm spf_L^t}]$.
\end{prop}
The proof will occupy the remainder of this subsections.

If $L$ is spin,  the fibration $SO(n) \to P_{SO} (L) \to L$ induces the following short exact sequence
\begin{equation}\label{sesSO}
1 \to \ZZ / 2\ZZ \to \pi_1 (P_{SO} (L)) \stackrel{p}{\to} \pi_1 (L) \to 1.
\end{equation}
The inclusion of the fiber induces the first injection $\pi_1(SO(n)) \to \pi_1(P_{SO} (L))$. One can check
that this is in fact a central extension.

Let $s:  \pi_1 (L) \to \pi_1 (P_{SO} (L))$ be a homomorphism, which is
a section of $p$, i.e $p\circ s = id$.
\begin{lemma}\label{lem:spinseciso}
A choice of spin structures on $L$ is equivalent to a choice of a section $s$ of the map $p$
of \eqref{sesSO}.
\end{lemma}
\begin{proof}
If we are given a spin structure $P_{spin} (L)$ on $L$, then from the diagram
\begin{equation*}
\xymatrix{ & 1 \ar[r]\ar[d] & \pi_1(P_{spin} (L)) \ar[r]^{\quad q} \ar[d]^{\pi} & \pi_1 (L)\ar[r]\ar[d]^{=} &1 \\
1 \ar[r]& \ZZ /2 \ZZ \ar[r]& \pi_1 (P_{SO} (L)) \ar[r]^{\quad p} & \pi_1 (L) \ar[r] & 1,
}
\end{equation*}
$s:=\pi \circ q^{-1}$ is a desired section. Note that $p \circ s = p \circ \pi \circ q^{-1} = q \circ q^{-1} =id.$

Conversely, suppose there exists a section $s : \pi_1 (L) \to \pi_1 (P_{SO} (L))$ of \eqref{sesSO}. Then, the image of $s$ is an index 2 subgroup of $\pi_1 (P_{SO} (L))$, which determines a 2-fold covering of $P_{SO} (L)$. Denote this covering by $P_{spin} (L)$ and let $F$ be a general fiber of $P_{spin} (L) \to L$.
\begin{equation*}
\xymatrix{ & \pi_1 (F) \ar[r]\ar[d] & \pi_1(P_{spin} (L)) \ar[r]^{\quad q} \ar[d]^{\pi} & \pi_1 (L)\ar[r]\ar[d]^{=} & \pi_0 (F) \\
1 \ar[r]& \ZZ /2 \ZZ \ar[r]& \pi_1 (P_{SO} (L)) \ar[r]^{\quad p} & \pi_1 (L) \ar[r] & 1
}
\end{equation*}
From the construction, $q$ in the diagram above is an isomorphism and hence $\pi_1 (F) = \pi_0 (F) =1$. Then, \cite[Theorem 1.4]{LM} tells us that $P_{spin} (L)$ is indeed a spin structure on $L$.
\end{proof}

A choice of such a section geometrically means a choice of trivialization of $TL$ along each based loop in $\pi_1 (L)$. If $t : \pi_1 (L) \to \pi_1 (P_{SO} (L))$ is another section, then their difference $s-t$ is a group homomorphism from $\pi_1(L)$ to $\ZZ / 2\ZZ$. i.e an element of
$$\Hom (\pi_1(L), \ZZ /2 \ZZ) = \Hom (H_1(L), \ZZ / 2\ZZ) = H^1 (L;\ZZ/2\ZZ).$$

Now, the bundle  $P_{SO} (L) \to L$ is $G$-equivariant, and hence we can consider an induced map
of orbifold fundamental groups 
$$p^{orb} :\pi_1^{orb} \left[ P_{SO} (L) / G\right] \to \pi_1^{orb} [L/G].$$
Note that $p^{orb}$ is surjective. Indeed given any path $\gamma$ from $x_0$ to $g(x_0)$ for $g \in G$,
$P_{SO}(L)|_{\gamma}$ is a trivial $SO(n)$-bundle over $\gamma$. Thus, if  $\WT{x}_0$ is a base point in $P_{SO} (L)$ lying over $x_0$,
we can find a path from $\WT{x}_0$ to $g\WT{x}_0$ which covers $\gamma$.

The following lemma is straightforward.
\begin{lemma}\label{lem:spgpdia}
We have the commutative diagram of group homomorphisms as below.
\begin{equation*}
\xymatrix{  & 1 \ar[d] & 1 \ar[d] & 1 \ar[d] & \\
1   \ar[r] & \ZZ /2 \ZZ \ar[r]\ar[d] & \ZZ /2 \ZZ \ar[r]\ar[d] & 1 \ar[r]\ar[d] & 1\\
1 \ar[r] & \pi_1 (P_{SO} (L)) \ar[r]\ar[d]_{p} & \pi_1^{orb} \left[ P_{SO} (L) / G \right] \ar[r]\ar[d]_{p^{orb}} & G\ar[r]\ar[d]^{=} & 1\\
1 \ar[r] & \pi_1 (L) \ar[r]\ar[d] \ar@/_/[u]_{s} & \pi_1^{orb} [L/G] \ar[r]\ar[d]  & G \ar[r]\ar[d] & 1\\
& 1  & 1& 1 & \\
}
\end{equation*}
\end{lemma}

Analogously to Lemma \ref{lem:spinseciso}, we show that the splitting of the second column in the above diagram is equivalent to having a $G$-action on the spin bundle induced by the section $s$ of \eqref{sesSO}. In other words, the obstruction to lift the $G$-action on $P_{SO} (L)$ to $P_{spin} (L)$  lies in the second group cohomology class determining the group extension in the second column.

\begin{lemma}\label{obst_porb}
Fix a spin structure on $L$ which is induced by a section 
$$ s: \pi_1(L) \to \pi_1 (P_{SO} (L) )$$
of \eqref{sesSO}. 
The $G$-action on $L$ (and $P_{SO} (L)$) can be lifted to $P_{spin} (L)$ if and only if there exists a section $s^{orb}$ of $p^{orb}$ which fits into the following diagram.
\begin{equation*}
\xymatrix{ \pi_1 (L) \ar[d]^{s} \ar[r] &  \pi_1^{orb} [L/G] \ar[d]^{s^{orb}}\\
\pi_1 (P_{SO} (L)) \ar[r]& \pi_1^{orb} \left[ P_{SO} (L) / G \right] }
\end{equation*}
where horizontal maps come from natural exact sequences
$$ 1 \to \pi_1 (L) \to \pi_1^{orb} [L/G] \to G \to 1$$
and
$$ 1 \to \pi_1 (P_{SO} (L)) \to \pi_1^{orb} \left[ P_{SO} (L) / G \right] \to G \to 1.$$
\end{lemma}
\begin{proof}
Firstly, suppose there is a section $s^{orb}$. We consider the universal cover $\Omega$ of $P_{SO} (L)$ which also can be regarded as a covering of $[P_{SO} (L)/G]$. Since $\Omega$ is simply connected, it is the orbifold universal cover of $\left[ P_{SO} (L) /G \right]$ and hence $\pi_1^{orb} \left[ P_{SO} (L) /G \right]$ acts on it. $ \pi_1^{orb} [L/G]$ and its subgroup $\pi_1(L)$ also act on $\Omega$ via $s^{orb}\left( \pi_1^{orb} [L/G] \right) \leq \pi_1^{orb} \left[ P_{SO} (L) /G \right]$ and $s^{orb} (\pi_1 (L)) = s (\pi_1 (L)) \leq \pi_1^{orb} \left[ P_{SO} (L) /G \right]$. We take the quotient of $\Omega$ by $\pi_1 (L)$-action so that the resulting quotient space is exactly the spin bundle $P_{spin} (L)$ associated to $s$. Then, clearly $\pi_1^{orb} [L/G] / \pi_1 (L) \cong G$ acts on $P_{spin} (L)$.


On the other hand, if we have a $G$-action on $P_{spin} (L)$ which makes $P_{spin} (L) \to L$ and $P_{spin} (L) \to P_{SO} (L)$ $G$-equivariant, then we get a group homomorphism $\pi_1^{orb} \left[ P_{spin} (L) \right] \to \pi_1^{orb} [L/G]$ and $\pi_1^{orb} \left[ P_{spin} (L) /G \right] \to \pi_1^{orb} \left[ P_{SO} (L) /G \right]$. The first map should be an isomorphism since the only element possibly contained in the kernel is a loop in the fiber and the fiber is simply connected. Thus,
\begin{equation}
\xymatrix{  \pi_1^{orb} [L/G] \ar[dr]^{\cong} \ar[rr] && \pi_1^{orb} \left[ P_{SO} (L) /G \right] \ \\
&\pi^{orb} \left[ P_{spin} (L)/G \right] \ar[ur] & } 
\end{equation}
gives a desired section.
\end{proof}

To prove Proposition \ref{prop:diffspinspf}, we need a slightly different commutative diagram, which is obtained from the
above. The following lemma is given for this purpose, whose proof is elementary and omitted.
\begin{lemma} 
Let $s(\pi_1(L))$ be the image of the composition of inclusions  $\pi_1(L) \stackrel{s}{\to} \pi_1(P_{SO}L) \to \pi_1^{orb} \left[P_{SO} (L) /G\right]$.
If the section $s$ corresponds to the $G$-invariant spin structure on $L$(see Lemma \ref{lem:spinseciso}), then $s(\pi_1(L))$ is
a normal subgroup of  $\pi_1^{orb} \left[P_{SO} (L) /G\right]$.
\end{lemma}

Consider the central column of the diagram in Lemma \ref{lem:spgpdia}, and divide the
2nd and the 3rd rows by $s(\pi_1(L))$ and $\pi_1(L)$ respectively to 
obtain the following diagram
\begin{equation}\label{diaspgp}
\xymatrix{  & 1 \ar[d] & 1 \ar[d] & 1 \ar[d] & \\
1   \ar[r] & 1 \ar[r]\ar[d] & \ZZ /2 \ZZ \ar[r]\ar[d] &  \ZZ /2 \ZZ \ar[r]\ar[d] & 1\\
1 \ar[r] & s(\pi_1 (L)) \ar[r]\ar[d]_{\cong} & \pi_1^{orb} \left[ P_{SO} (L) / G \right] \ar[r]\ar[d]_{p^{orb}} & G_s' \ar[r]\ar[d] & 1\\
1 \ar[r] & \pi_1 (L) \ar[r]\ar[d]  & \pi_1^{orb} [L/G] \ar[r]\ar[d]  & G \ar[r]\ar[d] & 1\\
& 1  & 1& 1 & .\\
}
\end{equation}
Note that the 2nd column is independent of the choice of spin structure $s$. However, the 3rd column
depends on the choice of $s(\pi_1(L))$, and only makes sense if $s$ is a $G$-invariant spin structure
(so that $s(\pi_1(L))$ is a normal subgroup). 

Recall that the obstruction to lift $G$-action to the spin bundle lies in the existence of a section of $p^{orb}$. The bottom left square diagram of \eqref{diaspgp} shows that such a section exists at least on the subgroup $\pi_1 (L)$ of $\pi_1^{orb} [L/G]$. Therefore, the obstruction descends to a group extension in the last column. (See \cite[(Ex 6.6.4)]{W}.) 
Therefore, the extension of $G$ in the last column gives rise to the second cohomology class $[{\rm spf}_L^s] \in H^2(G,\Z/2)$, the spin profile of the spin bundle determined by $s$. For a different $G$-invariant spin structure $t : \pi_1(L) \to \pi_1 (P_{SO} (L))$, we get a similar diagram
for the extension $G_t'$, and a possibly different spin profile $[{\rm spf}_L^t] \in H^2(G,\Z/2)$.

\begin{remark}
The discussion in the first paragraph of the proof of Lemma \eqref{obst_porb} shows that $G_s'$ genuinely acts on $P_{spin} (L)$, but its quotient $G$ may not.
\end{remark}

We recall from \cite[(Ex 6.6.4)]{W}  that from the 2nd and the 3rd columns of the diagram \eqref{diaspgp} canonically induce a homomorphism $ H^2 (G;\ZZ /2\ZZ) \to H^2 ( \pi_1^{orb} [L /G ] ; \ZZ /2)$
which sends the cohomology class of the extension of the 3rd column to the one of the 2nd column.  It implies that 
the cohomology class $[{\rm spf}_L^s]$ maps to the extension class of the 2nd column in 
$ H^2 ( \pi_1^{orb} [L /G ] ; \ZZ /2)$ under this induced homomorphism. But, the image of $[{\rm spf}_L^s]$ is independent of the choice of spin structure $s$.
This proves that the image of  $[{\rm spf}_L^s]$  and  $[{\rm spf}_L^t]$ coincide under
this homomorphism.

From \cite[(6.8.3)]{W}, we have the following exact sequence called {\em Lyndon-Hochschild-Serre spectral
sequence} from the 3rd row of the diagram \eqref{diaspgp}:
\begin{equation}\label{eq:LHSspec}
H^1 (L;\ZZ /2\ZZ)^G =H^1 (\pi_1 (L) ; \ZZ /2\ZZ)^G \to H^2 (G;\ZZ /2\ZZ) \to H^2 ( \pi_1^{orb} [L /G ] ; \ZZ /2).
\end{equation}
Therefore, the difference of two spin profiles $[{\rm spf}_L^s]- [{\rm spf}_L^t]$ which maps to zero under
the last homomorphism, is in the image of $H^1 (L;\ZZ /2\ZZ)^G$. As both $s$ and $t$ are $G$-invariant
spin structures, the difference of two spin structures $s-t$ lies in $H^1 (L;\ZZ /2\ZZ)^G$. The first map in \eqref{eq:LHSspec} sends it
to the difference of two spin profiles $[{\rm spf}_L^s]- [{\rm spf}_L^t]$. This proves the proposition \ref{prop:diffspinspf}.

In addition, we explain how $s-t$ associates ${\rm spf}^s_L$ and ${\rm spf}^t_L$ in a more geometric way.
For each $g \in G$, choose a path $\gamma_g$ from $x_0$ to $g(x_0)$ and a trivialization of $TL|_{\gamma_g}$ i.e. an orthonormal frame of $TL$ along $\gamma_g$. Then, the frame gives a path connecting $\WT{x_0}$ and $g (\WT{x_0})$ in $P_{SO}(L)$. The composition 
$$l_{g,h} :=h(\gamma_g ) \star \gamma_h \star \overline{\gamma_{gh}} \in \pi_1 (L)$$
(representing the element $[\gamma_g][\gamma_h][\gamma_{gh}]^{-1}$ in $\pi_1^{orb} \left[P_{SO} (L) /G \right]$) is a loop based at $x_0$ on which a trivialization of $TL$ is fixed. If this trivialization agrees with $t(l_{g,h})$, then we define $[g,h]_t=1$ and $[g,h]_t =0$ otherwise. $[\, , \, ]_t$ is exactly the factor set representing $t$ in the cohomology. Now, it is not difficult to check that the factor set corresponding to $s$ is given by 
$$[g,h]_s = (s-t)(l_{g,h})+  [g,h]_t$$
where $(s-t) (l_{g,h})$ is the evaluation of an $1$-cocycle $s-t \in H^1 (L;\ZZ /2\ZZ)$ at $l_{g,h}$.

\subsection{Generalizations}\label{subsec:genspfs}
\noindent{\bf (i)}
We finish the section with a short discussion on orientation reversing actions.
Suppose that $L$ is oriented, and a half of $G$-action on $L$ is orientation reversing.
Then, $(A_g)^* TL$ has opposite orientation compared to $TL$ for an orientation reversing $g$.

We recall from \cite{KT} the following well-known fact on 
an isomorphism between a vector bundle and itself with the reversed orientation.
Let $F \to B$ be an oriented vector bundle with transition functions $g_{ij}$ on a cover $\{U_i\}_{i \in I}$
of $B$. Choose maps $h_i:U_i \to O(n)\setminus SO(n)$, and define the new transition maps by $h_i \circ g_{ij} \circ h_j^{-1}$. Then a new vector bundle $F^{op}$ is obtained which is isomorphic to $F$ via the maps $h_i$. Different choices of $h_i$ yield isomorphic bundles.
Similarly, by choosing $h_i :U_i \to Pin(n) \setminus Spin(n)$, we can find an isomorphism 
from spin structures $P_{spin}(F)$ to $P_{spin}(F^{op})$. This provides
an isomorphism between sets of spin structures of $F$ and $F^{op}$.

Back to our case, consider an oriented Lagrangian submanifold $L$ and its spin structure $P_{spin}(L)$.
\begin{assumption}\label{as:spin2}
We require that the spin structure $(A_g)^*P_{spin}(L)$ (of $L^{op}$) is isomorphic
to the given spin structure $P_{spin}(L)$(of $L$) for each orientation reversing $g\in L$ in the sense described above.
\end{assumption}
For each $g$, we choose an isomorphism
$$T_g:(A_{g^{-1}})^*\big(P_{spin}L \big) \to P_{spin}L.$$
Then, the spin profile of $L$ is defined by \eqref{eq:spf} as well.
\begin{remark}
Seidel  \cite{Se} used an isomorphism between pin structures
$$ (A_g)^* \big(P_{pin}L \otimes \wedge^{top} TL \big)\cong P_{pin}L.$$
The additional factor $\wedge^{top} TL$ is not necessary for us since we only deal with orientable Lagrangian submanifolds here and hence $\wedge^{top} TL$ is a trivial bundle.\\
\end{remark} 


\noindent{\bf (ii)}
Here, we explain more details on spin profiles when $L$ is not connected (but a $G$-orbit).
In such a case, we have made Assumption \ref{ass:Ldisconnspf}.

As observed before, $\textnormal{spf}_L$ is a locally constant function, but not necessarily a constant function when $L$ is not connected,
since  signs in \eqref{eq:spf} may also depend on connected components of $L$.
We do not know how to handle the general case,  but we add more arguments under the Assumption \ref{ass:Ldisconnspf} together with \ref{as:spin disconn} below. First of all, from Assumption \ref{ass:Ldisconnspf} there exists a normal subgroup $G_0$ of $G$ such that it preserves each connected component of $L$, and $G \setminus G_0$ parametrizes the set of components.
\begin{assumption}\label{as:spin disconn}
The following exact sequence splits
$$1 \to G_0 \to G \to G/G_0 \to 1.$$
\end{assumption}
For the rest of the paper, we will only deal with disconnected Lagrangian submanifolds satisfying the above assumption also. So,one can express $G$ as a semi-direct product $G_1 \ltimes G_0$ where $G_1 \cong G / G_0$. For later use, let $\phi : G_1 \to {\rm Aut} (G_0)$ determine the semi-direct product $G_1 \ltimes G_0$. 
The main advantage here is that the quotient group $G/G_0$ also can be viewed as a subgroup of $G$ and hence, acts on $L$ purely by permuting connected components. In particular, we can label connected components of $L$ by $G_1(\cong G/G_0) \leq G$. 

We fix a $G_0$-invariant spin structure on the connected component $C_0$ corresponding to $1 \in G_1$, and pull it back to the component $C_\alpha$ corresponding to $\alpha (\in G_1)$ by $\alpha^{-1}$. As before we take a lift of $G_0$-action on the spin bundle on $C_0$ which gives rise to a spin profile $\textnormal{spf}_0 \in H^2 (G_0 ; \ZZ/2)$. Likewise, we get a spin profile $\textnormal{spf}_\alpha \in H^2 (G_0 ; \ZZ/2)$ for each connected component $C_\alpha$ associated to $\alpha=(a,1) \in G_1 (\leq G)$.
Assumption \ref{ass:Ldisconnspf} implies that $\textnormal{spf}_\alpha=\textnormal{spf}_0$ for all $\alpha \in G_1$.

For $g=(g_1,g_0) \in G=G_1 \ltimes G_0$ and $x=\alpha x_0 \in C_\alpha$, note that the connected component containing $g \cdot x$ is labeled by $(g_1 a, 1)$. So, $P_{spin} (L)_x = P_{spin} (L)_{x_0}$ and $P_{spin} (L)_{g \cdot x} = P_{spin} (L)_{( a^{-1}g_1^{-1},1) (g_1,g_0) \cdot x} =P_{spin} (L)_{(1, \phi_a (g_0) ) \cdot x_0}$. Now we define a map $P_{spin} (L)_x  \to P_{spin} (L)_{g\cdot x}$ by the following diagram:
\begin{equation*}
\xymatrix{
P_{spin} (L)_x  \ar[r] \ar@{=}[d]& P_{spin} (L)_{g\cdot x}  \ar@{=}[d]\\
P_{spin} (L)_{x_0} \ar[r] & P_{spin} (L)_{\phi_a (g_0) \cdot x_0}
}
\end{equation*}
where we abbreviate $(1, \phi_a (g_0)) \cdot x_0$ to $\phi_a ( g_0) \cdot x_0 (\in C_0)$ to emphasize that the second row only depends on the choice of lift of $G_0$-action to the spin bundle on $C_0$.

\begin{prop}
Under Assumption \ref{ass:Ldisconnspf}, the above choice of lifts produces $\textnormal{spf}_L$ which does not depend on connected components of $L$.
\end{prop}

\begin{proof}
The condition $\textnormal{spf}_\alpha=\textnormal{spf}_0$ implies that $\textnormal{spf}_0$ is invariant under the conjugation action of $\alpha$ on $G_0$. The rest of the proof follows easily by the direct computation and is omitted.
\end{proof}

Analyzing component-preserving subgroups is subtle in general since this involves a study of conjugacy relation among subgroups in $G$. For example, two component-preserving subgroups $G_1$ and $G_2$ are always conjugate to each other, but possibly conjugated by more than one element in $G$. We have only dealt with the case when $G$ can be decomposed nicely into ``component preserving" and ``permuting" subgroup so that there is no complicated issue on conjugacy classes. We postpone more general discussion to  \cite{CH2}.

\section{Equivariant Fukaya Categories for exact symplectic manifolds}\label{sec:eqfuex}
The Floer action functionals are single-valued for exact symplectic manifolds, and hence Novikov theory is not needed in such cases. Equivariant Fukaya categories can be constructed in a relatively simple way (than the general case), and in fact this only requires a choice of a group cohomology class in $H^2(G,\Z/2)$.
We remark that Seidel defined equivariant Fukaya category for $G=\Z/2$, and it corresponds
to the non-trivial group cohomology class $H^2(\Z/2,\Z/2)$ in our language.

In this section, definitions of equivariant Lagrangian branes and equivariant Fukaya categories for exact symplectic manifolds will be provided. We will use notations and constructions by Seidel in \cite{Se} to which we refer readers for more details.
At the end, we consider $G$-invariant parts of morphisms in these categories, which defines orbifolded Fukaya categories (or, its first approximation as
we should add the theory of bulk deformation by twisted sectors).

Let $(M,\omega)$ be an exact symplectic manifold  with an action of a finite group $G$.
Suppose that we have a $G$-invariant quadratic volume form $\eta^2$ on $M$, and a $G$-invariant almost complex structure $J_M$.
Let $\alpha_M:Gr(TM) \to S^1$ be a square phase map of $\eta$ defined by $\eta(v_1\wedge \cdots \wedge v_n)^2/
|\eta(v_1\wedge \cdots \wedge v_n)|^2$ for any basis of $V$ in the Lagrangian Grassmannian  $Gr(TM)$.

A Lagrangian brane in the sense of \cite{Se} is a triple  
$$L^\sharp  = (L, \alpha^\sharp,P^\sharp),$$
where $L$ is an exact Lagrangian submanifold with a grading $\alpha^\sharp: L \to \R$ satisfying $\alpha_M(T_xL) =exp(2 \pi i \alpha^\sharp)$ and
a pin structure $P^\sharp$ on it. For simplicity, we assume that $L$ is oriented, and $P^\sharp$ is
a spin structure on $L$. 

Let $s$ be a group cohomology class $H^2(G,\Z/2)$.
\begin{definition}\label{def:seqvbrane}
A triple $(L, \alpha^\sharp,P^\sharp)$ is called an {\em $s$-equivariant brane} if
\begin{enumerate}
\item $L$ is preserved by the $G$-action on $M$;
\item $\alpha^\sharp$ is $G$-invariant;
\item The isomorphism class of $P^\sharp$ is $G$-invariant and has a spin profile $s$.
\end{enumerate}
\end{definition}

We review each condition. If $L$ is not preserved by $G$-action, then one can consider
its equivariant family $\cup g \cdot L$ as an equivariant immersion from $|G|$-copies of $L$ to $M$.
For simplicity, we require that each intersection $g(L) \cap L$ is either empty or equal to $L$ for all $g \in G$. 

Condition on $G$-grading can be met as follows. If $L$ is connected, then
 $\alpha^\sharp \circ g - \alpha^\sharp$ is a locally constant function.
 Thus, $\alpha^\sharp \circ g - \alpha^\sharp=c$ on $L$, but since $g^{|g|}=1$ we have $c=0$.
That is, $\alpha$ is automatically $G$-invariant. If $L$ has several connected components, we  first fix a grading on one component which is invariant under the action of the subgroup preserving the given component, and then define the gradings on other components by the group action.

As revealed before, a spin structure $P^\sharp$ has an associated spin profile.
If it is not equal to $s$, one may try to change $P^\sharp$ to another spin structure
and match the associated spin profile with $s$ making use of the Proposition \ref{prop:diffspinspf}. However, it is not possible in general and $(3)$ is, in fact, a critical condition.

The setup of Floer cochain complexes in \cite{Se} briefly goes as follows.
Let $(L_0^\sharp, L_1^\sharp)$ be two exact Lagrangian branes such that $L_0$ and $L_1$ intersect transversely at $y \in L_0 \cap L_1$. (In fact, transversal intersection conditions are not necessary as we are able to choose perturbation data $(H_{01}, J_{01})$ instead. We return to this point in a moment.) Then, the pair of linear Lagrangian branes ($k=0,1$) at $y$ is given by 
$$\Lambda_{k,y}^\sharp = \big( \Lambda_{k,y} = (TL_k)_y, \alpha_{k,y}^\sharp=\alpha_k^\sharp(y), P_{k,y}^\sharp= (P_k^\sharp)_y \big).$$
Their index and orientation space are defined as in \cite[(11.25)]{Se} and denoted by
$$i(y) = i(\Lambda_{0,y}^\sharp, \Lambda_{1,y}^\sharp), o(y) = o(\Lambda_{0,y}^\sharp, \Lambda_{1,y}^\sharp).$$
Given a regular Floer datum, the graded Floer cochain group is defined as
$$CF^k(L_0^\sharp, L_1^\sharp) = \bigoplus_{o(y)=k} | o(y)|_R.$$
(See \cite[(12.16)]{Se}.)
\begin{remark}
If one changes the Pin structure $P_k^\sharp$ by multiplying $(-1)^{\epsilon_k}e$ for $k=0,1$,  then the resulting Pin structure is isomorphic to the original one, but the orientation space $o(y)$ is altered by the multiplication of $(-1)^{\epsilon_1 + \epsilon_2}$.
\end{remark}

A Floer datum in \cite{Se} for a pair of objects $(L_0, L_1)$  consists of
$$H_{01} \in C^\infty([0,1], C^\infty_c(M,\R)), \qquad J_{01} \in C^{\infty}([0,1], \mathcal{J}),$$
where $\mathcal{J}$ is the space of compatible almost complex structures on $M$. They are required to satisfy that $\phi^1 (L_0)$ is transverse to $L_1$ where $\phi^1$ is the time one map of the Hamiltonian vector field of $H_{01}$.
Given  $g \in G$, one can 
push forward the Floer datum to get $g(H_{01}), g(J_{01})$ in an obvious way

We are now ready to define $s$-equivariant Fukaya categories (modifying the construction of Fukaya categories by Seidel).
\begin{def-thm}\label{thm:exeqfuk}
Let $(M,\omega)$ be an exact symplectic manifolds. For $s \in H^2(G,\Z/2)$, we define the {\em $s$-equivariant Fukaya category} as follows.
Objects of the $s$-equivariant Fukaya category are $s$-equivariant branes.
A morphism between two $s$-equivariant branes $L_0^\sharp$ and $L_1^\sharp$ is given by $CF^*(L_0^\sharp, L_1^\sharp)$, on which $G$ acts linearly.

There exist $\AI$-operations
$$m_k : CF^*(L_0^\sharp, L_1^\sharp) \times \cdots \times CF^*(L_{k-1}^\sharp, L_k^\sharp) \to CF^*(L_0^\sharp, L_k^\sharp),\quad k=1,2 \cdots$$ 
which are compatible with the $G$-action. i.e. for $k=1,2 \cdots$
\begin{equation}\label{strictGAinfty}
m_k(gx_1,\cdots, gx_k) = gm_k(x_1,\cdots,x_k).
\end{equation}
\end{def-thm}

\begin{remark}
The definition above should include a $G$-equivariant flat vector bundle for each $s$-equivariant brane, but we omit it for simplicity and postpone it to Section \ref{sec: orbibundle}.
\end{remark}

On the technical level, we should be able to handle the problem of equivariant transversality involved in the $\AI$-operations on Fukaya categories. Assuming transversality for a moment, we show that $G$-actions on morphisms can be defined canonically if both Lagrangian branes have the same spin profile $s$.
\begin{lemma}\label{lem:transassume}
Consider two $s$-equivariant branes $(L_0^\sharp, L_1^\sharp)$. Suppose there exists a $G$-invariant regular Floer datum for these branes. Then, there exists a natural $G$-action on the Floer cochain group $\bigoplus_{o(y)=k} |o(y)|_R$.
\end{lemma}
\begin{proof}
The equivariance of perturbation data implies that for each $g \in G$, the pulled-back of each equivariant brane $L_i^\sharp$ under the $g$-action is
isomorphic to $L_i^\sharp$ itself for $i=0,1$. Hence, it is easy to see that 
for the intersections $y \in L_0 \cap L_1$ and  $y':=g(y)  \in g(L_0) \cap g(L_1) = L_0 \cap L_1$, 
$$i(y) = i(y'), \quad o(y) \cong o(y').$$
The precise isomorphism on determinant lines $o(y) \cong o(y')$ has to be considered carefully.
Recall that determinant line is defined from a Lagrangian path $\lambda$ from $T_yL_0$ to $T_yL_1$ (from the grading of
each Lagrangian) whose spin structure is determined from the spin structure of each Lagrangian.
Since grading is $G$-invariant, $g$-action image of the Lagrangian path $\lambda$ for $o(y)$ can be used
to define $o(y')$. Here $g(\lambda)$ has an induced canonical  spin structure from $L_0$ and $L_1$ (not from the $g$-action). This defines an isomorphism
$$o_g :  o(y) \to o(y').$$
 But the (canonical) spin structure for $o(y')$ from $L_0$ and $L_1$ at $y'$ can be compared to the  spin structure of $o(y)$ via group action, from the choice of  a lift  $T_{g, L_i} : (A_{g^{-1}})^*P_i \to P_i$ for each $g$ ($i=0,1$).
Since the lifts $T_{g, L_i} $ for $g \in G$ do not define a group action due to sign error \eqref{eq:spf}, the maps $\{g \mapsto o_g \, | \, g \in G\}$ 
do not define a group action on $ \bigoplus_{o(y)=k} |o(y)|_R$, either.
Instead, we have
$$o_g \circ o_h = (-1)^{\textrm{spf}_{L_0}(g,h) +\textrm{spf}_{L_1}(g,h)} o_{gh}.$$
However, both $L_0, L_1$ have the same spin profile $s =\textrm{spf}_{L_0}=\textrm{spf}_{L_1}$ by the definition of $s$-equivariant branes. Hence, the sign error cancels out and the maps $\{ g \mapsto o_g\, | \,g \in G\}$ define a $G$-action on the Floer cochain group.
\end{proof}

In what follows, we shall describe how to take care of equivariant transversality issue. Indeed, one can adapt the following algebraic technique introduced by Seidel in the exact case.

\begin{lemma}\cite[Lemma 4.3]{Se2}\label{hompert}
Let $\mathcal{D}$ be an $A_\infty$-category with an action of a finite group $G$, and $D=H(\mathcal{D})$, the associated cohomology level category. Suppose that we have another graded linear category $C$ with a $G$-action, and an equivalence $F: D \to C$, which is $G$-equivariant. Then, there is an $A_\infty$-category $\mathcal{C}$ with $H(\mathcal{C})$ isomorphic to $C$, which carries an action of $G$, and an equivariant $A_\infty$-functor $\mathcal{F} : \mathcal{D} \to \mathcal{C}$ such that $H(\mathcal{F}) =F$.
\end{lemma}
Seidel observed that one can make $G$ act freely on the set of objects of $\mathcal{D}$ via introducing additional labels to objects in $\mathcal{C}$ with elements of a group $G$ (possibly making several copies of the same object). Now, it is much easier for $\mathcal{D}$ to become $G$-equivariant. Then, the homological perturbation type lemma above is applied to obtain an equivariant $\AI$-structure on the category $\mathcal{C}$.

Let us spell out how to use this lemma to construct the $\AI$-operations in the $s$-equivariant Fukaya category without a transversality assumption. First, we define the $\AI$-category $\mathcal{D}$ and a free $G$-action on it. 
\begin{definition}
An object of $\mathcal{D}$ is given 
by $(L^\sharp, g)$ for $g \in G$. Here, $L^\sharp = (L, \alpha^\sharp, P^\sharp)$ is an $s$-equivariant brane of $M$ and $(L^\sharp, g) $ denotes an $s$-equivariant brane given by 
$(L, \alpha^\sharp, (g^{-1})^*P^\sharp)$. ($L$ and $\alpha^{\sharp}$ are $G$-invariant as usual.) 
A morphism between $(L_0^\sharp, g)$ and $(L_1^\sharp, h)$ is
$$CF^*((L_0^\sharp, g), (L_1^\sharp, h)).$$
We fix the perturbation datum $(H_{01,g}, J_{01,g})$ of $((L_0^\sharp,1), (L_1^\sharp,g) )$ for each $g \in G$ and use the push forward datum $(g(H_{01,g^{-1}h}), g(J_{01,g^{-1}h}) )$ for an arbitrary pair $((L_0^\sharp, g), (L_1^\sharp, h))$
\end{definition}

$\AI$-operations on $\mathcal{D}$ are defined in a standard way with aid of similarly chosen perturbation data for higher operations \cite{Se}.

\begin{lemma}\label{lem:gactionond}
$\AI$-category $\mathcal{D}$ admits an action of $G$.
\end{lemma}
\begin{proof}
On the object level, $g \in G$ sends $(L^\sharp, h)$ to $(L^\sharp, gh)$. The $G$-action on the set of objects is obviously free due to the additional index $h$ in $(L^\sharp, h)$. Now we clarify how $G$ acts on the morphism spaces. For simplicity, let us look at the following $g$-action 
$$CF ((L^\sharp,1), (L^\sharp, h)) \to CF ((L^\sharp,g), (L^\sharp, gh)) $$
and other cases can be dealt in the same way.

Let $L_0 \cap L_1$ transversally intersect at $y$ and denote $g(y)  \in g(L_0) \cap g(L_1) = L_0 \cap L_1$ by $y'$. Then, we have $i(y) = i(y')$ from the $G$-invariant gradings. Then, the group action $o_g: o(y) \to o(y')$ is defined in an obvious way. Recall that perturbation data for the left and the right hand sides are $(H_{01,h}, J_{01,h})$ and $(g(H_{01,h}), g(J_{01,h}))$ respectively. Moreover since the spin structure for the brane $(L^\sharp, k)$ is defined as the pulled-back one via $k$, the $g$-action is compatible with $\AI$-operations. Therefore, the map $g \mapsto o_g$ defines a group action on $\mathcal{D}$. No ambiguity arises since $g$ always sends an object to another. 
\end{proof}

\begin{remark}
We do not have to identify $(g^{-1})^*P^\sharp$ and $P^\sharp$ for $\mathcal{D}$ as they are now spin structures for different objects $(L^\sharp,1)$ and $(L^\sharp,g)$ in $\mathcal{D}$. 
Thus, $\mathcal{D}$ can actually be enlarged to include equivariant branes with various spin profiles.  
\end{remark}

According to Lemma \ref{hompert}, it suffices to define a graded linear category $C$  with a $G$-action and a $G$-equivariant equivalence $F:D \to C$ in order to define the $s$-equivariant Fukaya category. The graded linear category $C$ is a refined version of Donaldson-Fukaya category in the sense that objects are given with gradings and pin structures.
\begin{definition}\label{Donald}
Fix $s \in H^2(G,\Z/2)$. The linear category $C$ is defined as follows.
Objects of $C$ are $s$-equivariant branes $L^\sharp$'s. A morphism between $L_0^\sharp$ and $L_1^\sharp$ is given by $HF^*(L_0^\sharp, L_1^\sharp)$. The composition of morphisms is defined by the triangle product as usual.
\end{definition}
Note that morphisms are given by the Floer cohomology groups, rather than Floer cochain complexes. Our next task is to endow $C$ with a group action. 
We carry out the construction of a group action on $C$, making use of the idea of weak group actions on chain complexes. It is employed in \cite{CH} to deal with the Morse complex of a non-invariant function on a global quotient and we, in fact, proved that the weak group action induces a strict group action on the Morse homology. Here, we need a slight generalization since
the product structure is also required to be compatible with the $G$-action at least in a weak sense.

\begin{lemma}\label{condhomopert}
There exist a $G$-action on the linear category $C$ and a $G$-equivariant functor $D \to C$, where $D$ is the cohomology level category of $\mathcal{D}$.
\end{lemma}
\begin{proof}
The sign issue arising from spin profiles can be handled as in Lemma \ref{lem:transassume}, so we will discuss the question on perturbation data only. 

Recall that the well-definedness of Donaldson-Fukaya category uses the following canonical isomorphisms. Let $(H_{01}, J_{01}), (H_{01}', J_{01}')$ be two perturbation data for a pair 
$(L_0^\sharp, L_1^\sharp)$. From now on, we omit $J$-terms for notational simplicity. An one parameter family $H_{01}^c$ connecting two such data gives rise to a continuation map $\Phi_{H_{01}^c}$ from the Lagrangian Floer complex for $H_{01}$ to that for $H_{01}'$, which
induces  a canonical isomorphism 
between Floer cohomology groups $HF^*(L_0^\sharp, L_1^\sharp;H_{01} )$
and  $HF^*(L_0^\sharp, L_1^\sharp;H_{01}')$.

Given three perturbation data $H_{01}$, $H_{01}'$ and $H_{01}''$, the continuation map from $H_{01}$ to $H_{01}''$ is homotopic to the composition of two continuation maps, one from $H_{01}$ to $H_{01}'$ and the other from $H_{01}'$ to $H_{01}''$. This implies that the canonical isomorphisms on cohomology groups are compatible with compositions.

We consider a naive $G$-action on Floer complex which is a chain map, and compose it with
the continuation map $\Phi_{H_{01}^g}$ from some family $H_{01}^g$ from $g(H_{01})$ to $H_{01}$:
$$\Psi_g:CF^*(L_0^\sharp, L_1^\sharp;H_{01}) \stackrel{g}{\to}
CF^*(L_0^\sharp, L_1^\sharp;g(H_{01})) \stackrel{\Phi_{H_{01}^g}}{\to} CF^*(L_0^\sharp, L_1^\sharp;H_{01}).$$
Then, $\Psi_g$ is clearly a chain map, and we claim that it defines a weak action on the chain complex
$CF^*(L_0^\sharp, L_1^\sharp;H_{01})$, in the following sense:
\begin{equation}\label{eq:chhtp}
\Psi_g \circ \Psi_h - \Psi_{gh} = \sigma_{g,h} \circ m_1 + m_1 \circ \sigma_{g,h}.
\end{equation}
The proof of the claim goes as follows. 
First, one can check that the family $H_{01}^g \circ h^{-1}$ defines a continuation homomorphism
$h \circ \Phi_{H_{01}^g} \circ h^{-1}$ where $\Phi_{H_{01}^g}$ is a continuation defined by $H_{01}^g$. Thus, 
$$\Psi_g \circ \Psi_h =  \Phi_{H_{01}^g} \circ g \circ  \Phi_{H_{01}^h} \circ h =
 \Phi_{H_{01}^g} \circ (g \circ  \Phi_{H_{01}^h} \circ g^{-1}) \circ (gh) = 
  \Phi_{H_{01}^g} \circ ( \Phi_{H_{01}^h\circ g^{-1}}) \circ (gh).$$
The composition  $\Phi_{H_{01}^g} \circ ( \Phi_{H_{01}^h\circ g^{-1}})$ of
two continuation homomorphisms is chain homotopic to the continuation
map $\Phi_{H_{01}^{gh}}$ via homotopy $\WT{\sigma}_{g,h}$ and
we define $\sigma_{g,h} := \WT{\sigma}_{g,h} \circ (gh)$ to
obtain the formula \eqref{eq:chhtp}.

This shows that the $G$-action is well-defined on morphism spaces of the category $C$. Now, we show that the composition is compatible with the $G$-action.
Naive $G$-action gives the following commutative diagram:
\begin{equation}\label{mkg}
\xymatrix{
CF(L_0, L_1;H_{01}) \times CF(L_{1}, L_2;H_{1,2}) \ar[r]^{\qquad \qquad  m_2} \ar[d]^g & CF(L_0, L_2;H_{o2}) \ar[d]^g \\
CF(L_0, L_1; g(H_{01})) \times CF(L_{1}, L_2; g(H_{1,2}) ) \ar[r]^{\qquad \qquad m_2^g} & CF(L_0, L_2; g( H_{02}) ). 
}
\end{equation}
In order to move back from the second line to the first line, we consider continuation homomorphisms $\Phi_{H_{01}^g}, \Phi_{H_{12}^g}$
and $\Phi_{H_{02}^g}$, all of which we denote by $\Phi_1$ for simplicity. It is not hard to show that  there exist $\Phi_2$ (the second component of $\AI$-homomorphism(functor) between two  different perturbation data) satisfying the following relation:
\begin{equation}
\Phi_1m_2(x, y) + \Phi_2(m_1(x),y) + (-1)^{\deg'x}\Phi_2(x,m_1(y)) =m_2(\Phi_1(x),\Phi_1(y)) + m_1(\Phi_2(x,y)).
\end{equation}
This shows that the product is well-defined in Donaldson-Fukaya categories independent of perturbation data, and also
implies its compatibility with the $G$-action when applied to $(g(x), g(y))$.
The functor from $D \to C$ can be easily constructed using the isomorphisms $\{T_g\}_{g \in G}$ from the pull-back spin bundles to the original one.
\end{proof}

\section{Energy zero subgroups for non-exact cases}
From now on, we drop the exactness assumption and discuss the $G$-Floer-Novikov theory of a general symplectic manifold (see Section \ref{sec:FNtheory}). As observed in $G$-Novikov theory, we
need to find a subgroup of {\em energy zero} elements.
We will call this set $G_\alpha$ of energy zero elements of $G$,  which is an analogue of 
$G_\eta$ in $G$-Novikov theory in Section \ref{sec:general base}. For each energy zero element, we will assign a spin
structure on the Lagrangian bundle data associated to it, using the pre-fixed isomorphisms
$\{T_g\}_{g \in G}$. The composition of
energy zero elements will reveal the necessity of the condition that two Lagrangian should have the same spin profiles to define a group action on their Floer cochain complex.
  
Let $L_0$ and $L_1$ be $G$-invariant Lagrangian submanifolds of $M$, transversally intersecting with each other. Assume that both $L_0$ and
$L_1$ are connected, compact, oriented and they admit spin structures which are preserved by the $G$-action in the sense of Assumption \ref{as:spin}.
In case of orientation reversing action, we assume that both $L_0$ and $L_1$ have
orientation reversing $G$-actions and work under Assumption \ref{as:spin2}.
We set the notations as follows:
\begin{itemize}
\item
$A_{g,L_i}$ : the action of $g \in G$ on $L_i$,
\item
$A_{g,L_i}^{SO}$ : the induced action of $g$ on the frame bundle of $L_i$,
\item
$A_{g,L_i}^{spin}$ : the specified lift of $A_{g,L_i}^{SO}$ on the spin bundle of $L_i$ \\
(or $T_{g,L_i} : \left(A_{g,L_i}\right)^\ast P_{spin} (L_i) \stackrel{\cong}{\to} P_{spin} (L_i)$ ) 
\end{itemize}
for $i=0,1$. The choice of $A_{g,L_i}^{spin}$ above gives rise to spin profiles $\textnormal{spf}_{L_i}$ of $L_i$, correspondingly.

We choose a base path $l_0$ of the space of paths $\Omega(L_0, L_1)$ so that $g \cdot l_0 \neq h \cdot l_0$ for $g \neq h$. (This is not essential, but simplifies expositions). The naive $G$-action on $\Omega(L_0, L_1)$ will be written as 
$g(l)(t) = g( l(t))$ for $l \in \Omega(L_0,L_1)$.
Consider the connected component $\Omega(L_0, L_1;l_0)$ containing $l_0 $. 
Note that $g(l)$ may not lie in $\Omega(L_0, L_1;l_0)$. Hence, we take a subgroup preserving the connected component
as follows.
\begin{definition}\label{def:compreserve}
We define the subgroup $G_{l_0}$ of $G$ as
$$G_{l_0} := \{g \in G \mid  g(l_0) \in \Omega(L_0, L_1;l_0) \}.$$
\end{definition}
We will reduce it  further to the smaller subgroup, $G_\alpha$ of energy zero elements in $G_{l_0}$. For $g \in G_{l_0}$, we have a surface connecting $l_0$ and $g(l_0)$.
Namely, there exists a  surface $w_g:[0,1] \times [0,1] \to M$, such that 
\begin{eqnarray}\label{boundings1}
w_g(0,t) &=& l_0(t)\\
w_g(1,t) &=& g(l_0)(t) \nonumber \\
w_g(s,0) &\in& L_0 \nonumber \\
w_g(s,1) &\in& L_1\nonumber.
\end{eqnarray}
The symplectic area and the Maslov index of $w$ are defined as follows. The symplectic area of $w$ is simply defined by $\CA(w_g, g(l_0))=\int w_g^*\omega$. For the Maslov index of $w$, we construct a loop of Lagrangian subspaces along $w_g(\partial [0,1]^2)$ in a canonical way. Recall that we have the pre-fixed Lagrangian path $\WT{\lambda}_{l_0}$
along $l_0$. Then, $g(\WT{\lambda}_{l_0})$ is a Lagrangian path along $g(l_0)$ from $T_{g(p_0)}L_0$ to $T_{g(p_1)}L_1$. (We use the orientation reversal $g(\WT{\lambda}_{l_0})^{op}$ if $g$ is orientation reversing). Along $w_g(s, i)$ for $i=0, 1$, we consider $T_{w_g(s,i)}L_i$
for $s \in [0,1]$. The concatenation
of the above paths produces a bundle $\CL_{w_g}$ on $\partial [0,1]^2$ which is a Lagrangian sub-bundle of $w_g^*TM$ on $[0,1]^2$
The Maslov index of $\CL_{w_g}$ is well-defined, and denoted by $\mu(w_g, g(l_0))$.

 We will require $w_g$ to
have zero symplectic area and zero Maslov index, but
given a general $g \in G_{l_0}$, such $w_g$
does not necessarily exist. 
\begin{definition}\label{def:galpha}
We define $$G_\alpha = \{ g \in G \mid \exists w_g \,\, \textrm{with} \,\, 
\CA(w_g, g(l_0)) = \mu(w_g, g(l_0))=0\}.$$
Elements of $G_\alpha$ are called {\em energy zero} elements.
\end{definition}
Given $g \in G_\alpha$, there can be several choices of $w_g$ with
zero symplectic area and zero Maslov index. However, all $w_g$'s satisfying
 $\CA(w_g, g(l_0)) = \mu(w_g, g(l_0))=0$ are $\Gamma$-equivalent
 to each other. (See Definition \ref{def:gammaeq} for the $\Gamma$-equivalence relation.)  
 \begin{lemma}
$G_\alpha$ forms a subgroup of $G$.
\end{lemma}
\begin{proof}
Let $g$ and $h$ be elements of $G_\alpha$. Then, there are $(w_g, g( l_0))$ and $(w_h, h( l_0))$ whose energies and Maslov indices are both zero. We make a concatenation 
\begin{equation}\label{eq:conwg}
w_{gh} := w_g \star (g( w_h)).
\end{equation} 
Then, $(w_{gh} , (gh)(l_0))$ connects $l_0$ and $(gh)( l_0)$ and hence, $gh$ is accompanied with the path $w_{gh}$ satisfying 
$$\mathcal{A} ((w_{gh} , (gh)( l_0)) = \mu (w_{gh}, (gh)( l_0 )) =0$$
since $\mathcal{A}$ and $\mu$ are invariant under $G$-action. 
\end{proof}
 
 Now, we will associate a canonical spin structure of $\CL_{w_g}$ to each $w_g$ so that we can define the spin-$\Gamma$-equivalence relation among $w_g$'s as in \ref{def:spingamma}.
We will see that there may be at most two spin-$\Gamma$-equivalence classes for a single $\Gamma$-equivalence class.

Recall that $\CL_{w_g}|_{ \{0\} \times [0,1]} = \lambda_{l_0}$ has a preferred trivialization
$\sigma:[0,1] \times \mathbb{R}^n \to \widetilde{\lambda}_{l_0} $
and that the spin structure $P_{spin}(\lambda_{l_0})$ for $\lambda_{l_0}$ is fixed by one of its lifting $\widetilde{\sigma}: [0,1] \times Spin(n) \cong P_{spin}(\lambda_{l_0})$ \eqref{spinl0}.  
Over $[0,1] \times \{0\}$ (resp. $[0,1] \times \{1\}$)
we have $P_{spin} (L_0)$ (resp. $P_{spin} (L_1)$). 
Gluing of two spin structures at $(0,i)$ is done by $\iota_i \circ \WT{\sigma}^{-1}$ for $i=0,1$ as explained in \eqref{spinl0}.
It remains to specify spin structure of $\CL_{w_g}$ along $\{1\} \times [0,1]$
and how to glue it to $P_{spin} (L_i)$'s.

Let us assume for a moment that the $g$-action preserves orientations of $L_0$ and $L_1$. By $g$-action, the reference path $l_0$ is mapped to $g(l_0)$. Then, the Lagrangian path $g(\widetilde{\lambda}_{l_0})$ over $g(l_0)$ from $T_{g(p_0)}L_0$ to $T_{g(p_1)}L_1$ can be trivialized by $g \circ \sigma$:
$$[0,1] \times \mathbb{R}^n \cong g(\widetilde{\lambda}_{l_0}).$$
We take $(A_{g^{-1}})^* \big( P_{spin}(\widetilde{\lambda}_{l_0}) \big)$ as a spin structure $P_{spin}(g(\widetilde{\lambda}_{l_0}))$. From $\widetilde{\sigma}$, we get an induced isomorphism
\begin{equation}\label{eq:isogl0}
\widetilde{g\sigma}:[0,1] \times Spin(n) \cong P_{spin}(g(\widetilde{\lambda}_{l_0})).
\end{equation}
%
%
%

We need to define an analogue of $\iota_i$ at $g(p_i)$ to identify two spin bundles there. ($l_0 (i) = p_i$ for $i=0,1$.) In fact, $T_{g, L_i}: (A_{g^{-1}})^* P_{spin}L_i \to P_{spin}L_i$ comes into play to glue $P_{spin}(g(\widetilde{\lambda}_{l_0}))$ with $P_{spin}(L_i)$ at $g(p_i)$.
Note that $\iota_i:Spin(n) \to P_{spin}(L_i)_{p_i}$ defines a
natural embedding, which we again denote by 
$$\iota_i: Spin(n) \to   (A_{g^{-1}})^* P_{spin}L_i |_{g(p_i)}.$$
Consider its composition with $T_{g,L_i}$:
$$ T_{g,L_i} \circ \iota_i : Spin(n) \to  P_{spin} (L_i)_{g(p_i)}.$$
We see that $T_{g,L_i} \circ \iota_i \circ (\widetilde{g\sigma})^{-1}$ glues $P_{spin}(L_i)$ and $P_{spin}(g(\widetilde{\lambda}_{l_0}))$ at $g(p_i)$. (See Figure \ref{spinwg}) Therefore, we have described the canonical spin structure for $\CL_{w_g}$, which depends on the pre-fixed choice of $T_{g, L_i}$.

If $g$ is an orientation reversing action for both $L_0$ and $L_1$,
we instead use the orientation reversal $(A_{g})_*\WT{\lambda}_{l_0}^{op} = g(\WT{\lambda}_{l_0})^{op}$ along $g(l_0)$, whose end points can be identified with $T_{g(p_0)} L_0$
and  $T_{g(p_1)} L_1$. Let $\sigma_{SO} : [0,1] \times SO(\R^n) \to \WT{\lambda}_{l_0}$ be the trivialization of the frame bundle of $\WT{\lambda}_{l_0}$  induced from $\sigma$. If $(\R^n)^{op}$ denotes $\R^n$ with the opposite orientation, then $\sigma_{SO}$ leads to the trivialization ${\sigma}_{SO}^{op}:[0,1] \times SO((\R^n)^{op}) \to  \WT{\lambda}_{l_0}^{op}$. Take an isomorphism between $SO(\R^n)$ and $SO((\R^n)^{op})$, given
by an element $O \in O(n) \setminus SO(n)$.  
\begin{figure}[h]
\begin{center}
\includegraphics[height=1in]{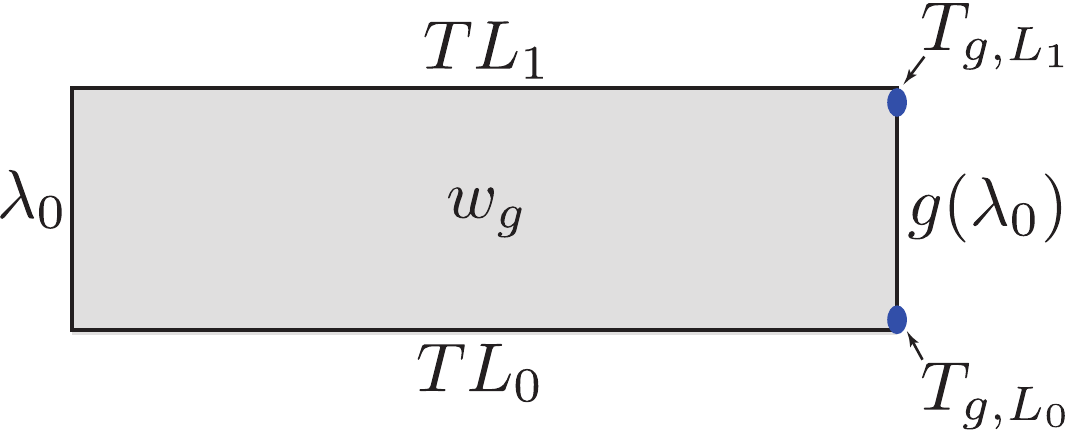}
\caption{The canonical spin structure for $\CL_{w_g}$}\label{spinwg}
\end{center}
\end{figure}

From the extension
$\WT{\sigma}:[0,1] \times Spin(\R^n) \stackrel{\cong}{\to} P_{spin}(\WT{\lambda}_{l_0})$, we have
$\WT{\sigma}^{op}:[0,1] \times Spin((\R^n)^{op}) \stackrel{\cong}{\to} P_{spin}(\WT{\lambda}_{l_0}^{op})$
which  are isomorphic to the former by multiplying a lift $\WT{O} \in Pin(n) \setminus Spin(n)$.
We also have $\WT{g\sigma}^{op}:[0,1] \times Spin((\R^n)^{op}) \stackrel{\cong}{\to} P_{spin}(g\WT{\lambda}_{l_0}^{op})$. We next take
$A_{g^{-1}}^*(P_{spin}(\WT{\lambda}_{l_0})^{op}))$ as the spin structure of 
$g(\WT{\lambda}_{l_0})^{op}$, and glue $P_{spin}(L_i)$ and $P_{spin}(g(\widetilde{\lambda}_{l_0})^{op})$ using 
$T_{g,L_i} \circ \iota_i \circ (\widetilde{g\sigma} \circ \WT{O})^{-1}$.

%

So far, we have described the canonical spin structure associated to the Lagrangian sub-bundle $\CL_{w_g}$ over the boundary of $w_g$.
As the bundle lies over $\partial [0,1]^2 \cong S^1$, there are only two possible spin structures. 
\begin{definition}
For each $w_g$, we associate a number $sp(w_g) \in \Z/2$ which is zero if the canonical spin
structure is trivial and $1$ otherwise.
\end{definition}

To handle orientation spaces in defining a $G_\alpha$-action in the next section, we should not only consider $w_g$ but also with its $sp$ value, i.e. $(w_g, sp(w_g))$ (or $(-1)^{sp(w_g)}w_g$ for short).



Lastly, we discuss the relation between group operations and $sp$.
We want to compare  the product $sp(w_g) \cdot sp(w_h)$ and  $sp(w_{gh})$ of the
concatenation $w_{gh}$ \eqref{eq:conwg}.
\begin{prop}\label{prop:compadm}
We have
$$sp(w_g) \cdot sp(w_h) = (-1)^{\textnormal{spf}_{L_0}(g,h)+\textnormal{spf}_{L_1}(g,h)}sp(w_{gh}).$$

\end{prop}
\begin{figure}[h]
\begin{center}
\includegraphics[height=2in]{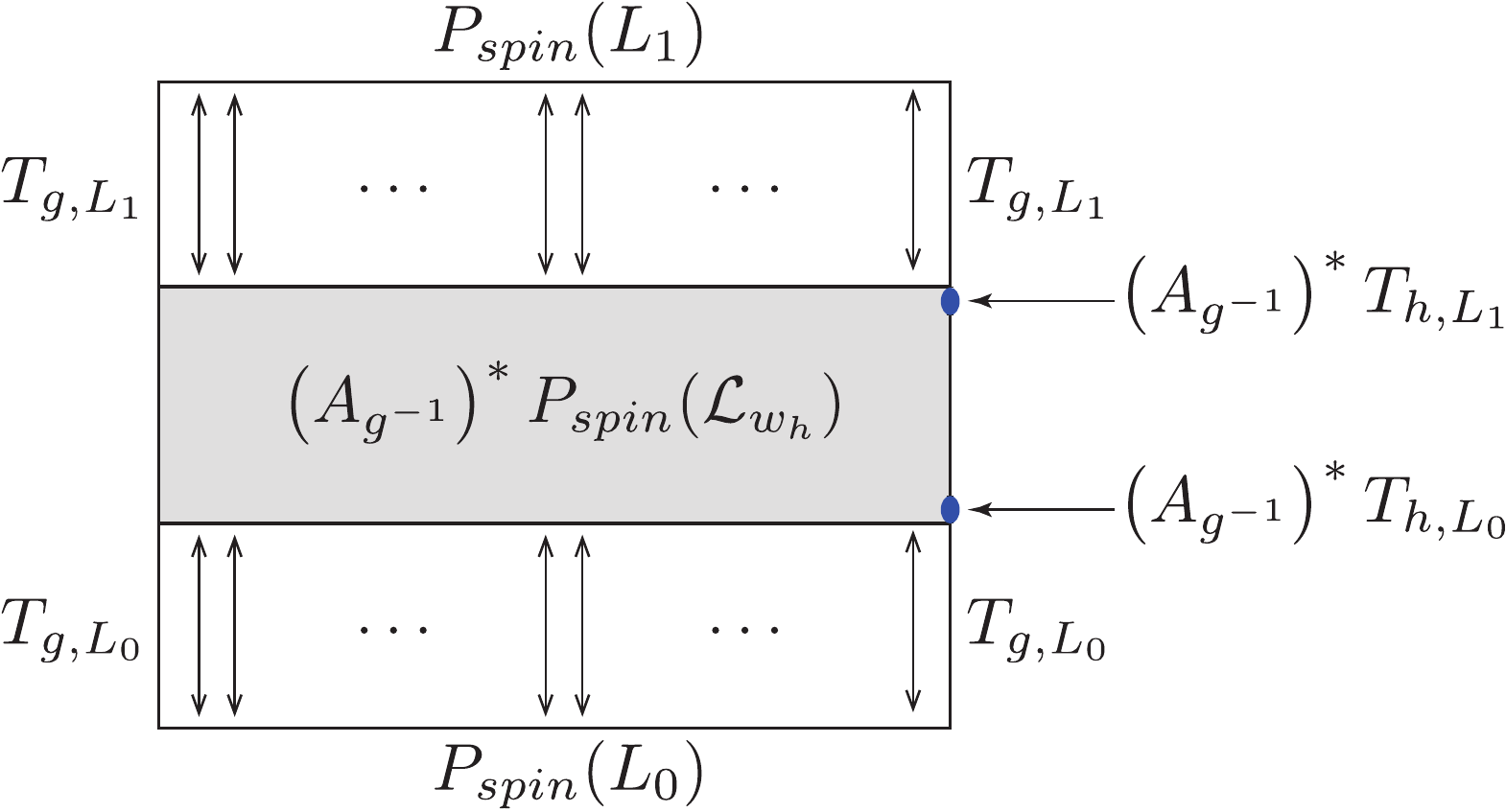}
\caption{Construction of $ P_{spin}(g(\CL_{w_h}) )$ }\label{cocycsp}
\end{center}
\end{figure}
\begin{proof}
For $g$ and  $h$ in $G_\alpha$, we choose paths $w_g$ and $w_h$ of energy and index zero, respectively. They are equipped with Lagrangian bundles $\CL_{w_g}$ and $\CL_{w_h}$ on their boundaries and also the spin structures $P_{spin}(\CL_{w_g}), P_{spin}(\CL_{w_h})$ obtained by $T_g$ and $T_h$ respectively. 

The concatenation of $w_g$ and $w_h$ gives $w_{gh}(= w_g \star g( w_h) )$ up to $\Gamma$-equivalence. $w_{gh}$ has two canonical spin structures. One is the canonical spin structure associated to the surface $w_{gh}$ which is used to define
$sp(w_{gh})$. The other is induced from spin structures of $w_g$ and $w_h$ described as follows.

Consider the image $g(w_h)$ of $w_h$ and the pull-back bundle $(A_{g^{-1}})^*P_{spin}(\CL_{w_h})$ (on $g(w_h)$) with respect to the $g$-action. If we restrict  $(A_{g^{-1}})^*P_{spin}(\CL_{w_h})$ to horizontal edges, it coincides with $(A_{g^{-1}})^*P_{spin}(L_i)$. We apply the isomorphism $T_{g,L_i}$ over the whole horizontal edges (Figure \ref{cocycsp})
 so that we obtain a new spin structure over the boundary of $g(w_h)$ which is isomorphic to $(A_{g^{-1}})^*P_{spin}(\CL_{w_h})$.  
 Let us denote the resulting spin structure
as $P_{spin}(g(\CL_{w_h}))$. Note that the restriction of $P_{spin}(g(\CL_{w_h}))$ to horizontal edges
is given by the restriction of the spin structure $P_{spin}L_i$ on $L_i$. Since $P_{spin}(g(\CL_{w_h}))$ is isomorphic to $(A_{g^{-1}})^*P_{spin}(\CL_{w_h})$ on $g(w_h)$ 
\begin{equation}\label{spactiongh}
sp (P_{spin}(g(\CL_{w_h}))) = sp( P_{spin}(\CL_{w_h})).
\end{equation}

We look into the gluing at each vertical ends of $P_{spin}(g(\CL_{w_h}))$ more carefully.
On the vertical edges of $g(w_h)$, we have $g(l_0)$ on the left and $gh(l_0)$ on the right.
$P_{spin}(g(\CL_{w_h}))$ agrees with $P_{spin}(g(\WT{\lambda_{l_0}}))$ (resp. $P_{spin}(gh(\WT{\lambda_{l_0}}))$) over $g(l_0)$ (resp. $gh(l_0)$) with a trivialization \eqref{eq:isogl0} 
$$\widetilde{g\sigma}:[0,1] \times Spin(n) \cong P_{spin}(g(\widetilde{\lambda}_{l_0}))
$$
(resp. $\widetilde{gh \sigma}:[0,1] \times Spin(n) \cong P_{spin}(gh (\widetilde{\lambda}_{l_0}))
$).
The horizontal and the vertical part of $P_{spin}(g(\WT{\lambda_{l_0}}))$ at the end of $g(l_0)$ is glued by
$T_{g,L_i} \circ \iota \circ \WT{g\sigma}^{-1}$.
On the other hand,
\begin{equation}\label{eq:newglue1}
T_{g,L_i} \circ (A_{g^{-1}})^*T_{h,L_i} \circ \iota \circ \WT{gh\sigma}^{-1}.
\end{equation}
is applied to glue the horizontal and the vertical component of $P_{spin}(g(\WT{\lambda_{l_0}}))$ at the end of $gh(l_0)$. (See the right vertical edge of Figure \ref{cocycsp}.)

The spin structure $P_{spin}(g(\CL_{w_h}))$
can be glued to $P_{spin}(\CL_{w_g})$ along $g(l_0)$ since the spin structure $P_{spin}(\CL_{w_g})$ over the right vertical edge have exactly the same trivializations and gluings as that of $P_{spin}(g(\CL_{w_h}))$ over the left vertical edge. 
Denote the glued spin structure by $P_{spin}(\CL_{w_g}) \star P_{spin}(g(\CL_{w_h}))$. By construction, it coincides with the canonical spin structure for $w_{gh}$ if we replace $T_{g,L_i} \circ (A_{g^{-1}})^*T_{h,L_i}$ in \eqref{eq:newglue1} by $T_{gh, L_i}$ for gluing on the right vertical edge.
Hence, the total difference between these two spin structures are
precisely the product of  $(-1)^{\textnormal{spf}_{L_i}(g,h)}$ for $i=0,1$ (see \eqref{eq:spf}),
and we have proved the claim.
\end{proof}

Note that if $L_0$ and $L_1$ have the same spin profiles, then differences between $P_{spin}(\CL_{w_g}) \star P_{spin}(g(\CL_{w_h}))$ and $P_{spin}( \CL_{w_{gh}} )$ at two vertices of
the right vertical edge cancel out, which will allow us to define group actions in the next section.

\section{Group actions on Floer-Novikov complexes}\label{GactFloerNov}
In this section, we will define $G_\alpha$-action (Definition \ref{def:galpha}) on the Floer Novikov complex of $(L_0,L_1)$ when spin profiles of $L_0$ and $L_1$ coincide.
\begin{assumption}\label{as:spin3}
We assume that 
$$\textnormal{spf}_{L_0}= \textnormal{spf}_{L_1}.$$
\end{assumption}
The assumption can be weakened to only require that $[\textnormal{spf}_{L_0}]= [\textnormal{spf}_{L_1}]$.
In fact, if they define the same class in $H^2 (G;\Z /2)$, then one can modify one of them to fulfill Assumption \ref{as:spin3} by changing the lifts $A_g^{spin}$'s appropriately.
We emphasize again that this condition is necessary to define a group action on the Floer complex.

 For $g \in G_\alpha$, we will use the bounding surface $w_g$ connecting
the base path $l_0$ and $g(l_0)$ of zero symplectic area and zero Maslov index, but such $w_g$ might carry a non-trivial canonical spin structure.
If $sp(w_g)$ is $0$ for all $g$, then we define $G_\alpha$
action on the covering space $\WT{\Omega} ( L_0, L_1;l_0)^{sp}$ without sign corrections, and it gives rise to an action on the Floer complex. However, in general, one should consider ``$(-1)^{sp(w_g)} w_g$".

As the action of ``$(-1)$" on the covering space does not make sense, we cannot define the action of $(-1)^{sp(w_g)} w_g$ directly on the covering
space. We will first make an assumption that $sp (w_g) =0$ for all $g \in G$ (Assumption \ref{ass:sp0}) to give a clear view of the action of $G_\alpha$ on the covering space, and later in Subsection \ref{Galpha}, construct the precise $G_{\alpha}$-action on the Floer-Novikov complex without Assumption \ref{ass:sp0}.
Actual $G_\alpha$ action should be considered as an action between the orientation spaces (see Subsection \ref{oriline}) of the Floer-Novikov complex and these are spaces that actions of the sign factors $(-1)^{sp(w_g)}$ can be made.

\subsection{Group actions on the Novikov covering $\WT{\Omega} ( L_0, L_1;l_0)^{sp}$}
Let us make the following simplifying assumption just in this subsection.
\begin{assumption}\label{ass:sp0}
For each $g \in G_\alpha$, we can find $w_g$ with 
\begin{equation}\label{eq:wgsp0}
sp(w_g)=0(=\CA(w_g, g(l_0)) = \mu(w_g, g(l_0)) ).
\end{equation}
\end{assumption}
Recall that we have defined the universal covering space of the path space $\Omega( L_0, L_1;l_0)$ by $\WT{\Omega} ( L_0, L_1;l_0)$ in Subsection \ref{subsec:FNtheory}, and refined it to get a smaller covering
$\WT{\Omega}( L_0, L_1;l_0)^{sp}$ by introducing spin $\Gamma$-equivalence relation
in Subsection \ref{subsec:SpinGamma}.

The $G$-action on $M$ gives a naive $G$-action on path spaces such that
\begin{equation}\label{def:honestgaction}
g((w,l)) \in \WT{\Omega}_{univ}( L_0, L_1;g(l_0)), \quad g((w,l)) =(g(w),g(l)),
\end{equation}
for $(w,l) \in  \WT{\Omega}_{univ}( L_0, L_1;l_0)$, where $g(w)(s,t) := g(w(s,t))$. Since this naive $G$-action changes the base path from $l_0$ to $g(l_0)$, it does {\em not}
give the correct $G$-action for  Floer-Novikov complex.
Alternatively, we define the $G_\alpha$-action on $\WT{\Omega}(L_0, L_1;l_0)^{sp}$
in the following way. For  $g \in G_\alpha$,  we take $w_g$ satisfying \eqref{eq:wgsp0}. Then, the $G_\alpha$-action on $\WT{\Omega}(L_0, L_1;l_0)^{sp}$ is defined
by first taking the naive $g$-action and then attaching $w_g$.  
\begin{definition}
We define the action of  $G_\alpha$, 
 $G_\alpha  \times \WT{\Omega} (L_0, L_1;l_0)^{sp} \to \WT{\Omega} (L_0, L_1;l_0)^{sp}$, by
\begin{equation}\label{def:galphaaction}
(g,  [w,l] ) \mapsto [w_g \star (g ( w)), g ( l)] =: g \cdot [w,l].
\end{equation}
(See Figure \ref{Fig:Galphaact}.)
\end{definition}
\begin{figure}[h]
\begin{center}
\includegraphics[height=1.8in]{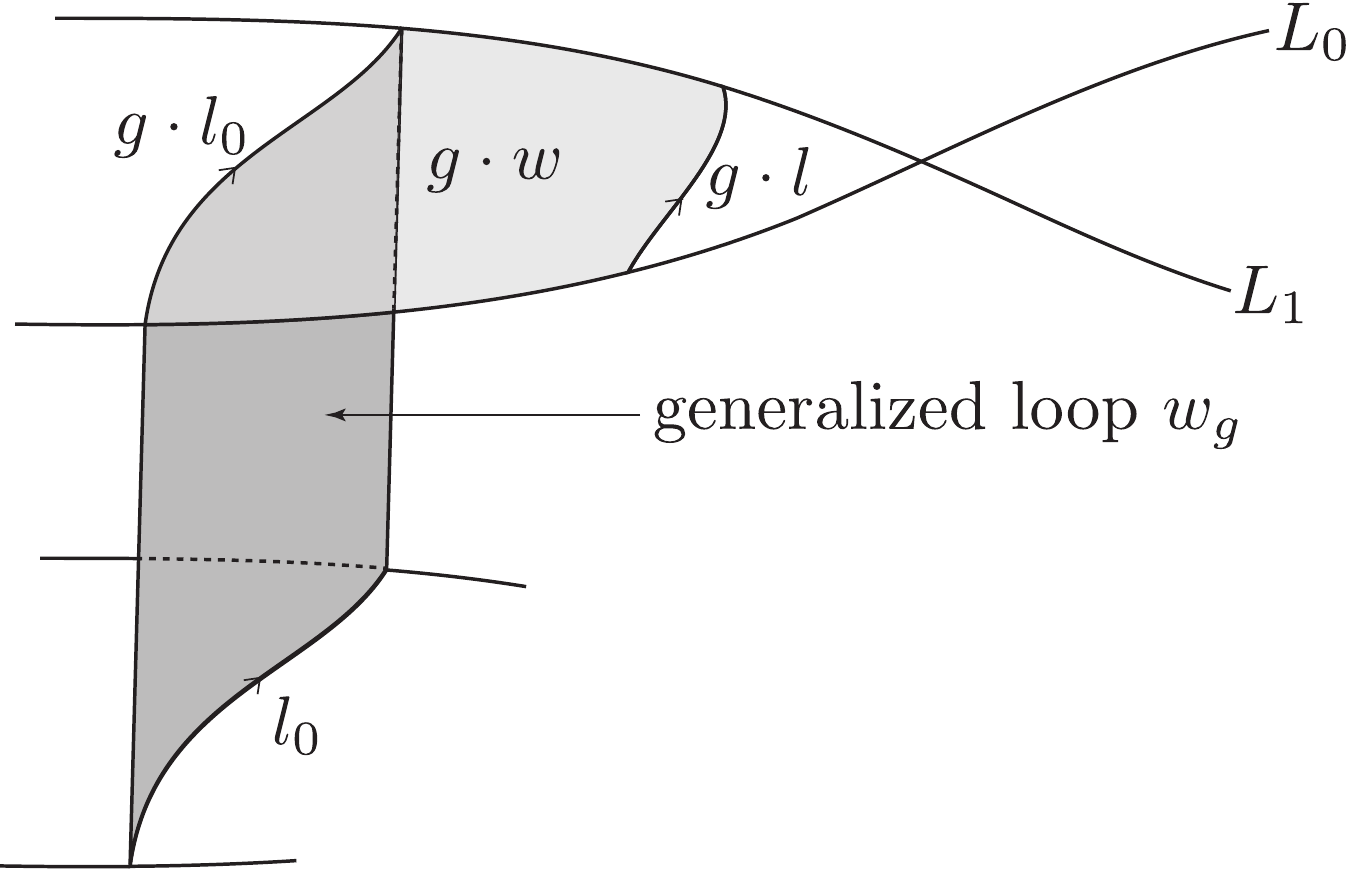}
\caption{$G_\alpha$-action on $\WT{\Omega} (L_0, L_1;l_0)^{sp}$}\label{Fig:Galphaact}
\end{center}
\end{figure}
\begin{lemma}
\eqref{def:galphaaction} provides a well-defined $G_\alpha$-action on $\WT{\Omega} (L_0, L_1;l_0)^{sp}$
\end{lemma}
\begin{proof}
One may choose a different $w_g'$ satisfying \eqref{eq:wgsp0} to define $g$-action, or
choose a different representative $[w',l]$ of $[w,l]$, but it is not hard to
see that these provide the same map \eqref{def:galphaaction} up to spin $\Gamma$-equivalence.
Note that
we have proved Proposition \ref{prop:compadm} under the assumption \ref{as:spin3}, which implies that
if $w_g$ and $w_h$ satisfy \eqref{eq:wgsp0}, $w_{gh}= w_g \star g(w_h)$ also
satisfies \eqref{eq:wgsp0}. This shows that \eqref{def:galphaaction} indeed defines a group action. 
\end{proof}

Now, we prove an analogue of Lemma \ref{lem:deccomm}.
\begin{lemma}
The covering map  $\pi^{sp} : \WT{\Omega} (L_0, L_1;l_0)^{sp} \to \Omega (L_0, L_1;l_0)$ is $G_\alpha$-equivariant.  Moreover the action of $ G_\alpha$ and that of deck transformation group $\Pi (L_0, L_1;l_0)$ on $\WT{\Omega} (L_0, L_1;l_0)^{sp}$ commute with each other. 
\end{lemma}
\begin{proof}
The projection map sends $[w_g \star (g ( w)), g ( l)]$ to $g(l)$,
hence the first claim follows.

Suppose that an element $c \in \Pi (L_0,L_1;l_0)$ is represented by a loop $C$ at $l_0$ in the path space $\Omega$. Then,
$$g \cdot c \cdot [w,l]= [w_g \star (g  (C)) \star (g ( w)), g (l) ].$$ 
On the other hand,
$$c \cdot g \cdot [w,l]= [C \star w_g \star (g ( w)), g(l)].$$
Their difference is represented by the cylinder 
\begin{equation}\label{eq:cywgc}
w_g \star (g (C)) \star \overline{w_g} \star \overline{C}.
\end{equation}
Note that  $w_g\star g (C) \star \OL{w_g}$ and $\overline{C}$ have the opposite energies, Maslov indices, and  $sp$-values. Hence \eqref{eq:cywgc} has the trivial spin structure as
well as vanishing energy and Maslov index,
which  implies 
\begin{equation}\label{gccg}
g \cdot c \cdot [w,l] = c \cdot g \cdot [w,l].
\end{equation}
\end{proof}
Let us review the Floer-Novikov setting briefly.
From the construction, $(\pi^{sp})^\ast \alpha$ is exact for the covering map $\pi^{sp}$ in \eqref{def:pi}.  The Floer action functional $\mathcal{A}$ on $\WT{\Omega} ( L_0, L_1;l_0)^{sp}$ is defined by
$\mathcal{A} ([w,l]) = \int w^{\ast} \omega$ for $[w,l]$ in $\WT{\Omega} ( L_0, L_1;l_0)^{sp}$.
Then, it is not difficult to check that $\mathcal{A}$ is well-defined and $(\pi^{sp})^\ast \alpha = - d \mathcal{A}.$

The following lemma demonstrates that $\CA$ is compatible with the $G_\alpha$-action.

\begin{lemma}$\left. \right.$
\begin{enumerate}
\item
$\CA$ is invariant under $G_\alpha$-action on $\WT{\Omega} (L_0, L_1;l_0)^{sp}$.
\item
The set $Cr(L_0, L_1 ; l_0)$ of critical points of $\mathcal{A}$ consists of the pairs $[w,l_p]$ where $l_p$ is the constant path with $p \in L_0 \cap L_1$. The set $Cr(L_0, L_1 ; l_0)$ is invariant under the action of $\Pi(L_0, L_1;l_0)^{sp}$ and $G_\alpha$. 
\end{enumerate}
\end{lemma}
\begin{proof}
The $G_\alpha$-invariance of $\mathcal{A}$ follows from the $G$-invariance of $\omega$ and the definition of $G_\alpha$-action. 
In particular, the set of critical points of $\mathcal{A}$ is preserved by $G_\alpha$. The standard variation argument  as in \cite{FOOO}  shows that $Cr(L_0, L_1 ; l_0)$ consists of $[w,l_p]$'s which directly implies that it is invariant under $\Pi(L_0, L_1;l_0)^{sp}$.
\end{proof}

Let $p \in L_0 \cap L_1$ and suppose that the constant path $l_p$ lies in the component $\Omega(L_0,L_1;l_0)$.
We consider the relation between $G_\alpha$ and the isotropy group $G_p$,
which is an analogue of Lemma \ref{locgpinG}.
\begin{lemma}\label{locgpinGfl}
For $p \in L_0 \cap L_1$, we have
$$G_p \subset G_\alpha.$$
In general, if there exists a path from $L_0$ to $L_1$ which is fixed by $g \in G$,
then $g$ belongs to $G_\alpha$. 
\end{lemma}
\begin{proof}
If $g \in G_p$, then $g \in G_{l_0}$ since it preserves the connected component of $\Omega(L_0,L_1)$ containing
$l_p$. To see that $g \in G_\alpha$, take $(u, l) \in \WT{\Omega}_{univ}(L_0,L_1;l_0)$,
where $u$ is a bounding surface from $l_0$ to $l_p$.
Then, $g(u)$ is a bounding surface from $g(l_0)$ to $g(l_p) = l_p$. Setting $w_g := u \star \OL{g(u)}$, we get a bounding surface  from $l_0$ to $g(l_0)$ which has vanishing energy and Maslov index. The general case can be proved similarly.
\end{proof}

\begin{definition}
We denote by   $(G_{\alpha})_{[w,l]}$
the isotropy group at $[w,l] \in \WT{\Omega} (L_0, L_1;l_0)$ for the $G_\alpha$-action.
Also,  we denote by $(G_{\alpha})_{\pm[w,l]}$ the set of elements of $G_\alpha$ which either act trivially on the spin $\Gamma$-equivalence class $[w,l]$ or send
$[w,l]$ to a $\Gamma$(but not spin $\Gamma$)-equivalent $[w',l]$.
\end{definition}
Obviously $(G_{\alpha})_{[w,l]} \subset (G_{\alpha})_{\pm[w,l]}$.

Analogously to the Lemma \ref{lem:isolocalgroup2}, one might expect that $(G_\alpha)_{[w,l_p]}=G_p$, but unfortunately it is not true in general. In fact, we are not interested in the isotropy group $(G_{\alpha})_{[w,l]}$, but rather
in $(G_{\alpha})_{\pm[w,l]}$.
If $[w,l_p]$ are spin $\Gamma$-equivalent to $[w',l_p]$, their orientation spaces will be identified in a canonical way. (See Lemma \ref{cylinder}.) If they are $\Gamma$-equivalent but not spin $\Gamma$-equivalent, the identification of orientation spaces are changed by reversing the sign. We see that in either case, $(G_{\alpha})_{\pm[w,l]}$ sends the orientation space of $[w,l_p]$ to
itself. In this setting, we have the following lemma.
\begin{lemma}\label{comparelocgp}
We have
$$(G_\alpha)_{\pm[w,l_p]}=G_p.$$
\end{lemma}
\begin{proof}
For $g \in G_p$, take $w_g: = w \star \OL{g(w)}$ as in the proof of lemma \ref{locgpinGfl}.
We claim that $g \in (G_\alpha)_{[w,l_p]}$ if $sp(w_g)=0$.
Observe that the $g$-action on $[w,l_p]$ is nothing but $[w_g \star g(w), l_p]$. Since $w_g \star g(w) = w \star \OL{g(w)} \star g(w)$
which is spin $\Gamma$-equivalent to $w$, we have proved the claim.

If $sp(w_g)=1$, we should use $(-1) \cdot w_g$ to define an action of $g (\in G_\alpha)$, but
we can not define such an action directly. From Assumption \ref{ass:sp0},
there is certain $w_g'$ satisfying \eqref{eq:wgsp0} which, then, is
$\Gamma$-equivalent but not spin-$\Gamma$-equivalent to $w_g$.
Note that $w_g' \star g(w)$ is $\Gamma$-equivalent to $w_g \star g(w)$,
and hence to $w$. Thus, $g \in (G_\alpha)_{\pm[w,l_p]}$.

The other direction $(G_\alpha)_{\pm[w,l_p]}\subset G_p$ is easy to check.
\end{proof}

\subsection{$G_\alpha$-action on the critical points and their orientation spaces}\label{Galpha}
The $G_\alpha$-action on the orientation spaces are crucial ingredient for defining a group action on Floer theory of Lagrangian intersections. This signifies that as in the Morse case, we do {\em not} want the group action on the set of
critical points of the Novikov-Floer complex, but rather on the set of the orientation spaces associated to critical points.

So far, we have considered the $G_\alpha$-action on the Novikov covering $\WT{\Omega}(L_0,L_1;l_0)^{sp}$ under Assumption \ref{ass:sp0}. From now on, we drop the assumption, and define $G_\alpha$ action directly on the complex $CF_{R,l_0}^{sp, *}(L_0,L_1)$ in
\eqref{defflcmx2}. 
The key part of the construction is how to define an action on the orientation spaces
associated to critical points.

We first go back to the definition \ref{def:orspace}, and show that the orientation space $\Theta_{[w,l_p]}^-$ can be assigned to each spin $\Gamma$-equivalence class $[w,l_p] \in  Cr(L_0, L_1 ; l_0)^{sp}$ without ambiguity.
\begin{lemma}\label{cylinder}
Orientation bundles $\Theta_{[w,l_p]}^-$ for different representatives of $[w,l_p]$ can
be canonically identified. Hence, the assignment 
$\Theta_{[w,l_p]}^-$ to each $[w,l_p] \in  Cr(L_0, L_1 ; l_0)^{sp}$ is
well-defined.
\end{lemma}
\begin{proof}
Let two different representatives $(w',l_p)$ and $(w,l_p)$ represent the same spin-$\Gamma$-equivalence class.
Regarding $w$ and $w'$ as paths in $\Omega(L_0,L_1;l_0)$,
$w'$ is homotopic to $w' \star \OL{w} \star w$.
Thus, $w'$ can be thought of as the concatenation of $w' \star \OL{w}$ and $w$.
Let $\lambda_{w'\star \OL{w}}$ be the associated Lagrangian bundle data
of $w' \star \OL{w}$ along the boundary.
Note that by the definition of spin $\Gamma$-equivalences, $w' \star \OL{w}$
has energy zero, and $\mu(\lambda_{w'\star \OL{w}})=0$. Furthermore, the associated canonical spin structure is trivial.

To relate $\Theta_{(w,l_p)}^-$ and $\Theta_{(w',l_p)}^-$, 
we consider gluing of $\dbar_{\lambda_{w'\star \OL{w}}, D^2}$ to $\dbar_{\lambda_w,Z_-}$ at a point. (See Figure \ref{cylori}.) Here, the domain of the $\bar{\partial}$-operator associated with $\lambda_{w'\star \OL{w}}$ is replaced by $D^2$ (instead of a rectangle) to use gluing theorems in Subsection \ref{subsec:Gluingthm} more conveniently. 
\begin{figure}[h]
\begin{center}
\includegraphics[height=1.5in]{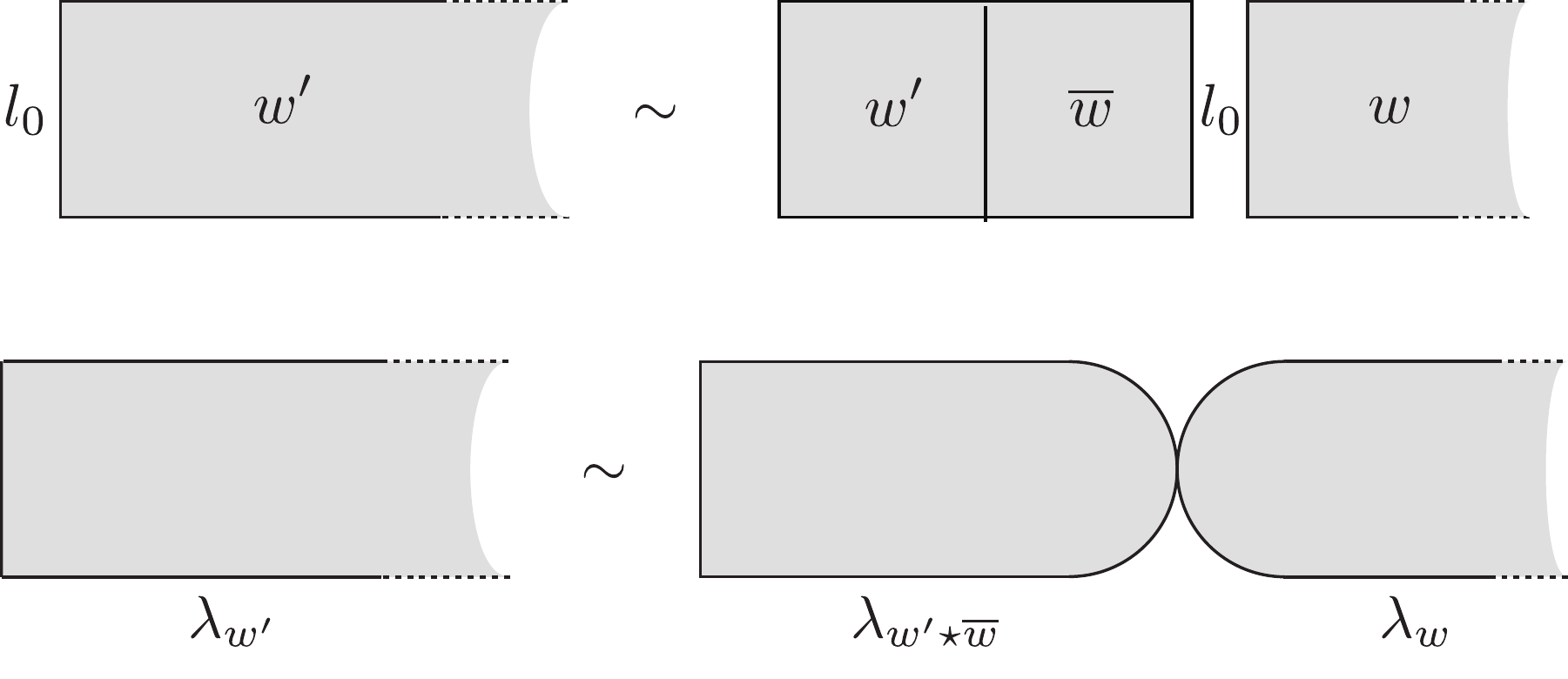}
\caption{Gluing of $\dbar_{\lambda_{w'\star \OL{w}}, D^2}$ and $\dbar_{\lambda_w,Z_-}$}\label{cylori}
\end{center}
\end{figure}
The index formula \eqref{zindexformula2} tells us that
\begin{equation}\label{eq:iden1}
{\rm{det}} (\bar{\partial}_{\lambda_w, Z_-}) \otimes {\rm{det}} (
\dbar_{\lambda_{w'\star \OL{w}}, D^2})
\cong {\rm{det}} (\bar{\partial}_{\lambda_{w'}, Z_-}) \otimes \wedge^{top} \lambda_z
\end{equation}
for some $z \in \partial Z_-$. 
From Lemma \ref{lem:se11} with the trivial spin structure on $\lambda_{w'\star \OL{w}}$, we have an identification 
\begin{equation}\label{eq:iden2}
{\rm{det}} (\bar{\partial}_{\lambda_{w'\star \OL{w}}, D^2}) \cong \wedge^{top} \lambda_z.\end{equation}
Combining \eqref{eq:iden1} and \eqref{eq:iden2}, we get the canonical identification
$${\rm{det}} (\bar{\partial}_{\lambda_w, Z_-}) \cong
{\rm{det}} (\bar{\partial}_{\lambda_{w'}, Z_-}),$$
which is denoted by $\Theta_{[w,l_p]}^-$.
\end{proof}

Now, we explore the group action on the orientation space $\Theta_{[w,l_p]}^-$ associated
to $[w,l_p] \in Cr(L_0,L_1;l_0)^{sp}$. It will be soon revealed that the group action below is only well-defined if the spin profile of $L_0$ and $L_1$ are the same.

\begin{definition}\label{def:aact}
Given $g\in G_\alpha$, take $w_g$ satisfying 
$\CA(w_g, g(l_0)) = \mu(w_g, g(l_0))=0$.
Then,  the action on Floer complex $$G_\alpha \times CF_{R,l_0}^{sp,*}(L_0,L_1) \mapsto CF_{R,l_0}^{sp,*}(L_0,L_1) $$
is defined by the linear map $ \Phi^{\Theta}_g$ on generators
\begin{equation}\label{action1}
 \Phi^{\Theta}_g:\Theta_{[w,l_p]}^- \mapsto \Theta_{[w_g 
\star g(w),l_{g(p)}]}^-.
\end{equation}
$\Phi^{\Theta}_g$ is defined as a composition
\begin{equation}\label{action2}
\Theta_{[w,l_p]}^- \;\;\stackrel{A^{\Theta}_g}{\to}\;\; \Theta_{[g(w),l_{g(p)}]}^-\;\;
\stackrel{\textrm{gluing}}{\to} \;\;\Theta_{[w_g \star g(w),l_{g(p)}]}^-
\stackrel{(-1)^{sp(w_g)}}{\to}\;\;\Theta_{[w_g \star g(w),l_{g(p)}]}^-.
\end{equation}
(See Figure \ref{GactOri}.)
\end{definition}
\begin{figure}[h]
\begin{center}
\includegraphics[height=3in]{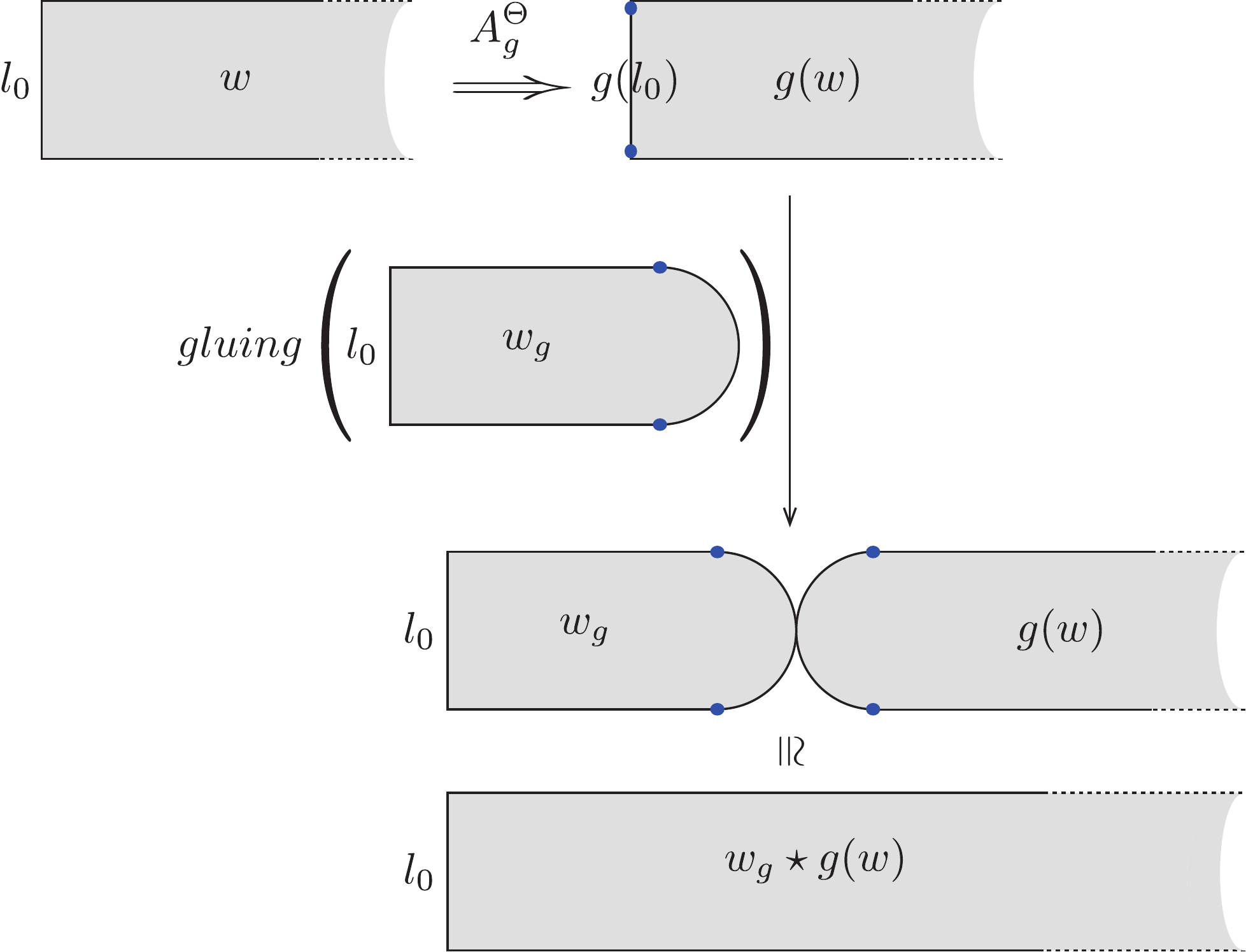}
\caption{First two steps of $\Phi^{\Theta}_g$}\label{GactOri}
\end{center}
\end{figure}
The action defined is independent of choice of $w_g$. Observe that we multiply $(-1)^{sp(w_g)}$ if $w_g$ has a non-trivial spin structure.

Let us closely look into each part of the map \eqref{action2}. We only explain the
case when $G_\alpha$ is orientation preserving. One can proceed similarly
for the other case, which we omit. \\

\noindent{\bf (1)} The first map $A^{\Theta}_g$ is the map given by the naive $G$-action.
Namely, taking $g$-action to $[w,l_p]$ and the associated Lagrangian bundle $\lambda_w$, we get a generator $[g(w), l_{g(p)}]$ and the Lagrangian bundle $\lambda_{g(w)}$ along $\partial Z_-$.
Then, we consider two Cauchy-Riemann problems on $Z_-$ with Lagrangian boundary
conditions given by $\lambda_w$ and  $\lambda_{g(w)}$ in $w^*TM$ and $g(w)^*TM$ respectively.
Note that the naive $g$-action sends the bundle pair $(w^*TM, \lambda_w)$
to $(g(w)^*TM,\lambda_{g(w)})$. This induces an action of $g$ between two
Cauchy-Riemann problems ($J$ is $G$-invariant), and in particular, one can send kernel and cokernel of 
$\dbar_{\lambda_w,Z_-}$ to those of $\dbar_{\lambda_{g(w)},Z_-}$.
They define a map $A_g^{\Theta}$ from ${\rm{det}} (\bar{\partial}_{\lambda_w, Z_-})$
to ${\rm{det}} (\bar{\partial}_{\lambda_{g(w)}, Z_-})$ by Definition \ref{def:orspace}.\\

\noindent{\bf (2)} 
The second map {\em gluing} is needed from $G$-Novikov theory, since $A_g^{\Theta}$ image  gives a path from $g(l_0)$ not $l_0$.
Hence, we restricted to the subgroup $G_\alpha$, and consider the gluing map of determinant spaces.
To define this map, we need to specify the choices of spin structures.
For this, we need to define the spin structure of $\lambda_{g(w)}$ carefully.
The desired spin structure of $\lambda_{g(w)}$ should have the following property.
We already have a  prescribed spin structure on  $\lambda_{w_g}$. Hence, if we glue  $\lambda_{w_g}$, with $\lambda_{g(w)}$,
we will obtain one spin structure on the concatenation $\lambda_{w_g \star g(w)}$. 
But $w_g \star g(w)$ also has another spin structure induced from the spin structures of $L_0$ and $L_1$.
We would like to have these two spin structures being equal to each other.

For this purpose, we define the spin structure of $\lambda_{g(w)}$ as follows.
 Along top and bottom edges of $Z_-$ we have spin structures induced from
 $P_{spin} (L_1)$, and $P_{spin} (L_0)$. Since $g(\lambda_{l_0})$ lies over the circular boundary of $Z_-$ for $g(w)$, it is natural to consider $P_{spin}(g(\lambda_{l_0}))$ along this region. Then, we glue it to $P_{spin} (L_i)$ making use of $T_{g,L_i} \circ \iota_i \circ \WT{g\sigma}^{-1}$ for $i=0,1$ as before. Observe that the circular boundary part of $Z_-$ has the same spin structure as the right vertical edge of $w_g$. What we obtain after gluing $w_g$ and $g(w)$ is indeed the canonical spin structure of $\lambda_{w_g \star g(w)}$ (produced by the spin structure of
 $L_0$, $L_1$, $P_{spin}(l_0)$ and  $\iota_i \circ \sigma^{-1}$).
This is because contributions 
$T_{g,L_i} \circ \iota_i \circ \WT{g\sigma}^{-1}$ appear twice in the gluing above (once in the spin structure of $w_g$ (right edge) and next in the spin structure of $g(w)$ (left edge)), and hence they cancel out.

We are now ready to define the second map $gluing$. $gluing$ is an isomorphism 
$$\Theta_{[g(w),l_{g(p)}]}^-\;\;
\to \;\;\Theta_{[w_g \star g(w),l_{g(p)}]}^-.$$
obtained from  the gluing of $w_g$ and $g(w)$ at a point
using \eqref{zindexformula2}.
To be specific, we take $S_4 = D^2$ (instead of $[0,1]^2$) and $F_4 = \lambda_{w_g}$,
and $S_5=Z_-$ and $F_5 = \lambda_{g(w)}$, and
$S_6=Z_-$ and $F_6= \lambda_{w_g \star g(w)}$.\\

\noindent{\bf (3)}  The last component, a multiplication by $(-1)^{sp(w_g)}$ is needed to have a well-defined group action.
Namely, we want the group action to be independent of $w_g$ chosen. In case that $sp(w_g)$ is non-trivial, 
the gluing (last step in the figure) in fact has a hidden contribution
of the sign factor $(-1)^{sp(w_g)}$  from Lemma \ref{lem:se11}:
$${\rm{det}}(\dbar_{\lambda_{w_g}, D^2}) \stackrel{(-1)^{sp(w_g)}}{\to} \Lambda^{top} F_{5,z}$$
(for some $z \in \partial S_4$). Namely, we can identify 
${\rm{det}}(\dbar_{\lambda_{w_g}, D^2})$ with ``$(-1)^{sp(w_g)}\Lambda^{top} F_{5,z}$".
Therefore, from \eqref{zindexformula2},
we obtain an canonical isomorphism
\begin{equation}
{\rm{det}}(\dbar_{\lambda_{g(w)}, Z_-}) \cong {\rm{det}}(\dbar_{\lambda_{w_g \star g(w)}, Z_-}).
\end{equation}
Hence, to cancel out this effect, we have added a multiplication by $(-1)^{sp(w_g)}$. \\

As  explained in {\bf(2)}, the resulting (glued) spin structure on $\lambda_{w_g \star g(w)}$ is the
canonical spin structure, induced from those of $L_0$, $L_1$ and $\lambda_{l_0}$.
Therefore, $\Phi^{\Theta}_g$ gives a map between orientation spaces which are determined by canonical spin structures of $[w,l_p]$ and $[w_g \star g(w),l_{g(p)}]$. 

To see that $ \Phi^{\Theta}$ defines a group action on the Floer complex, notice that
$A_{g}^\Theta \circ A_{h}^\Theta = A_{gh}^\Theta$, and both $\Phi^{\Theta}_{g} \circ  \Phi^{\Theta}_{h}$ and
$ \Phi^{\Theta}_{gh}$ are maps between orientation spaces which are determined by
canonical spin structures. The former map is twisted by $(-1)^{sp(w_g) + sp(w_h)}$
whereas the latter map is twisted by $(-1)^{sp(w_{gh})}$. From the Proposition \ref{prop:compadm}, 
two twistings coincide precisely when spin profiles of $L_0$ and $L_1$ agree with each other. 

Lastly, the $G_\alpha$-action on the Floer cochain complex is the linear extension of $G_\alpha$-action on $Cr(L_0,L_1;l_0)$ and this is compatible with what actually happens in geometry from the identity \eqref{gccg}. (See Definition \ref{def:novl01} for the precise definition of $\Lambda^R(L_0, L_1;l_0)^{sp}$.)
\begin{prop}\label{prop:galphacf}
$G_\alpha$-action on $CF_{R,l_0}^\ast(L_0,L_1)$ is $\Lambda^R(L_0, L_1;l_0)^{sp}$-linear.
\end{prop}

\subsection{Orbifold Novikov ring}\label{subsec:newNov}
We  present an enlarged  Novikov ring $\Lambda_{orb}^R(L_0, L_1;l_0)$ for orbifold Floer cohomology which is an analogue of $\Lambda_{[\eta]}^{rob}$ in Subsection \ref{subsec:OrbFund}.  
Heuristically, $\Lambda_{orb}^R(L_0, L_1;l_0)^{sp}$ contains elements in $G \setminus G_\alpha$ as there exists a path from $l_0$ to $g(l_0)$ for $g \notin G_\alpha$ which gives rise to a genuine loop of paths between $[L_0/G]$ and $[L_1/G]$ which increases (or decreases) the energy non-trivially.

Consider the set of homotopy classes of ``generalized" paths in $\Omega(L_0, L_1;l_0)$ which begin at $l_0$ and end at $g(l_0)$ for $g \in G$. Then, we take the quotient of it by the set of classes which can be represented by energy, Maslov and $sp$-value zero paths. ($I_\omega$ and $I_\mu$ are extended to generalized paths in the obvious way.) Denote this quotient by $\Pi(L_0, L_1 ; l_0)_{orb}^{sp}$. In particular, $\Pi(L_0, L_1 ; l_0)_{orb}^{sp}$ contains an element represented by a path $w$ whose end point is $g(l_0)$ for $g \in G \setminus G_\alpha$. There is an exact sequence
$$1 \to \Pi(L_0, L_1 ; l_0)_{orb} \to \Pi(L_0, L_1 ; l_0)_{orb}^{sp} \to G/ G_\alpha \to 1$$
analogously to the third line of the diagram \eqref{fundgpdia}.

\begin{definition}
$\Lambda_{orb} (L_0, L_1;l_0)^{sp}$ is defined by the completion of the group ring  of $\Pi(L_0, L_1 ; l_0)_{orb}^{sp}$ over $R$.
\end{definition}

By mimicking the construction in Definition \ref{def:obnoveta}, the $G_\alpha$-invariant part $CF_{R,l_0}^\ast (L_0,L_1)^{G_\alpha}$ admits a structure of $\Lambda_{orb} (L_0, L_1;l_0)^{sp}$-module.  Note that $\Lambda_{orb} (L_0, L_1;l_0)^{sp}$ has the $\ZZ$-grading induced by $I_\mu$ and there is a natural inclusion $\Lambda^R(L_0, L_1;l_0)^{sp} \to \Lambda_{orb}^R(L_0, L_1;l_0)^{sp}$ which fits into the diagram
\begin{equation*}
\xymatrix{ \Lambda^R(L_0, L_1;l_0)^{sp} \ar[r] \ar[d] & \Lambda_{nov} (R)\\
 \Lambda_{orb}^R(L_0, L_1;l_0)^{sp} \ar[ur] & }
\end{equation*} 
where $\Lambda_{orb}^R(L_0, L_1;l_0)^{sp} \to \Lambda_{nov}$ is obtained by taking the energy and the Maslov index, and with an additional multiplication by $(-1)^{sp}$. We omit further details as this new coefficient ring will not be visible if we use the universal Novikov field.  From now on, we will only use $\Lambda_{nov} (R)$ as our coefficient.


\section{Equivariant Transversality}\label{sec:Ku}
Lagrangian Floer theory is built upon subtle data of $J$-holomorphic discs and strips. In order to
have good moduli spaces of such maps, an appropriate perturbation scheme is required. Furthermore, we need equivariant transversality for the sake of group actions on Floer theory, and
 it is difficult to achieve both equivariance and transversality at once using simple methods.

We have shown that the algebraic technique due to Seidel can be hired
to overcome this problem in section \ref{sec:eqfuex} for exact symplectic manifolds and Lagrangian submanifolds. Note that the bubbling off of discs does not occur in this case.
Seidel developed a domain-dependent perturbation scheme for $J$-holomorphic strips and polygons which gives rise to the desired $\AI$-relations.

For the general case, however, disc bubblings do occur and the domain dependent perturbation is not good enough. Fukaya and Ono \cite{FO} introduced a kind of an abstract perturbation scheme so called Kuranishi structure.
Fukaya-Oh-Ohta-Ono \cite{FOOO} have developed it further to achieve transversality.
Note that Kuranishi structure is, in fact, a tool designed for attaining equivariant transversality more systematically with help of multi-sections. Hence, adding a finite group action to Kuranishi structure is not presumed to cause more difficulty. We will explain how to achieve equivariant transversality in this setup.

\subsection{Brief Review of Kuranishi structure}
We refer readers to \cite{FO}, \cite [Appendix A]{FOOO}, \cite{FOOO3} for the
definition and properties of Kuranishi structure. What we will carry out in this section is roughly as follows. Suppose a space $\CM$ admits a Kuranishi structure and a $G$-action. We shall show that there is an induced Kuranishi structure on $\CM/G$ whose multi-valued perturbation gives rise to $G$-equivariant multi-sections on $\CM$ if chosen suitably.

Let us explain in more detail. First we briefly explain the terms.
\begin{definition}\cite[A1.1]{FOOO}
A Kuranishi neighborhood of a point $p$ in a compact metrizable space $X$ is a quintuple $(V_p, E_p, \Gamma_p, \psi_p, s_p)$ such that
\begin{enumerate}
\item $V_p$ is a finite dimensional smooth manifold (possibly with boundary
or corner);
\item $E_p$ is a finite dimensional real vector space;
\item $\Gamma_p$ is a finite group acting smoothly and effectively on $V_p$ and linearly on $E_p$;
\item $s_p$ is a $\Gamma_p$-equivariant map $V_p \to E_p$;
\item $\psi_p:s_p^{-1}(0)/\Gamma_p \to U_p \subset X$ is a homeomorphism
to a neighborhood $U_p$ of $p \in X$.
\end{enumerate}
\end{definition}

Given a pair of neighborhoods $(V_p, E_p, \Gamma_p, \psi_p, s_p)$
and $(V_q, E_q, \Gamma_q, \psi_q, s_q)$ of $p \in X$ and $q \in 
\psi_p(s_p^{-1}(0)/\Gamma_p)$, a coordinate change \cite[A1.3]{FOOO}
is denoted as $(\hat{\phi}_{pq},, \phi_{pq}, h_{pq})$. Here,
$h_{pq}:\Gamma_q \to \Gamma_p$ is an injective homomorphism, and $\phi_{pq}:V_{pq} \to V_p$ is an $h_{pq}$-equivariant
smooth embedding from $\Gamma_q$-invariant open neighborhood $V_{pq}$ of the
origin in $V_q$ to $V_p$, and $\hat{\phi}_{pq}$ is a compatible equivariant
embedding of vector bundles $E_q \times V_{pq} \to E_p \times V_p$.
These are required to satisfy additional properties which we omit.
Note that a coordinate change may exist only in one direction.

\begin{definition}\cite[A1.5]{FOOO}
A Kuranishi structure on $X$ assigns a Kuranishi neighborhood 
$(V_p, E_p, \Gamma_p, \psi_p, s_p)$ for each $p \in X$ and
a coordinate change $(\hat{\phi}_{pq}, \phi_{pq},h_{pq})$ for
each $q \in 
\psi_p(s_p^{-1}(0)/\Gamma_p)$ such that
\begin{enumerate}
\item $\dim V_p - \textrm{rank}\;  E_p$ is independent of $p$, called virtual
dimension of $X$;
\item For $r \in \psi_q(s_q^{-1}(0)\cap V_{pq}/\Gamma_q),
q \in \psi_p(s_p^{-1}(0)/\Gamma_p)$, there exists $\gamma_{pqr} \in \Gamma_p$
satisfying
$$h_{pq}\circ h_{qr} =\gamma_{pqr} \cdot h_{pr} \cdot \gamma_{pqr}^{-1},\;\;
\phi_{pq}\circ \phi_{qr} = \gamma_{pqr} \cdot \phi_{pr},\;\;
\hat{\phi}_{pq}\circ \hat{\phi}_{qr} = \gamma_{pqr} \cdot \hat{\phi}_{pr}.$$
\end{enumerate}
\end{definition}

If trying to perturb $\{s_p\}$ to make it transverse to zero section, one
needs to construct perturbations inductively. Since a coordinate change may
exist only in one direction, a clever choice of subcollection of Kuranishi neighborhoods is required.
This subcollection should be ordered for inductive perturbations, and cover $X$. Such a choice
is called a {\em good coordinate system}. (See \cite[A.1.11]{FOOO}.)
One can always choose a good coordinate system, whose detailed construction
has been given in part II of \cite{FOOO3}. The proof goes by the induction on the dimension of $V_p$, and in each dimension, one glues several Kuranishi neighborhoods to obtain bigger Kuranishi neighborhoods and finally get a good coordinate system.
There is a notion of tangent bundle of a space $X$ with a Kuranishi structure, 
and also a notion of orientations on it, which we omit. 

We now recall multi-sections of a Kuranishi structure. 
For  $(V_p, E_p, \Gamma_p, \psi_p, s_p)$, consider the product of  $l$ copies of $E_p$,
denoted as $E_p^l$, and endow it with the action of
symmetric group $S_l$. Denote the quotient space by $\mathcal S^l(E_p):=E_p^l/S_l$,
which has an induced action of $\Gamma_p$.
There exists a $\Gamma_p$-equivariant map
$$
tm_m: \mathcal S^l(E_p)
\to \mathcal S^{lm}(E_p),
$$
which sends $[a_1,\ldots,a_l]$ to
$$
[\,\underbrace {a_1,\ldots,a_1}_{\text{$m$ copies}},\ldots,
\underbrace {a_l,\ldots,a_l}_{\text{$m$ copies}}].
$$
\begin{definition}\cite[A1.19]{FOOO}
An $n$-multisection $s$ of $\pi : E_p \times V_p \to V_p$ is a $\Gamma_p$-equivariant
map $V \to \mathcal S^n(E_p)$. It is said to be {\em liftable} if there exist $\WT{s}=(\WT{s}_1,\cdots
\WT{s}_n):V_p \to E_p^n$ such that its composition with $\pi:E_p^n \to \mathcal S^n(E_p)$ is $s$.
Here $\WT{s}$ needs not be $\Gamma_p$-equivariant. Each $\WT{s}_i$ is called a branch of $s$.
\end{definition}
Liftable multisections will be considered always.
Given an $n$-multi-section, we obtain an $nm$-multi-section by composing it with $tm_m$ map.
An $n$-multisection $s$ is equivalent to an $m$-multisection $s'$ if their induced $mn$-multisections
are the same. We identify equivalent multi-sections from now on. A lifted multisection is said to be transversal to zero if each of its branches is
transversal to the zero section. Compatibilities of multi-sections can be handled in the same way as in \cite{FOOO} and we omit the details.

A family of multisections $s_\epsilon$ is said to converge to a Kuranishi map $s=\{ s_p \}$ as $\epsilon \to 0$ if there exists $n$ such that $s_\epsilon$ is represented by an $n$-multisection $s_\epsilon^n$ and $s_\epsilon^n$ converges to a representative of $s$. 
\begin{lemma}\cite[A1.23]{FOOO}\label{lem:multisecch}
Suppose that a good coordinate system of the Kuranishi structure over $X$ is given, and that the Kuranishi structure has a tangent bundle. 
Then, there exists a family of multisections $\{ s'_{p,\epsilon} \}$ such that it converges to $\{ s_p \}_{p \in P}$ and $s'_{p, \epsilon}$'s are transversal to $0$ for all $\epsilon >0$.
 \end{lemma}

\subsection{Kuranishi structure for equivariant transversality}
A notion of group actions on a space with Kuranishi structure is introduced in \cite[A1.3]{FOOO}.
\begin{definition}\label{def:aukr}
 Let $X$ be a space with Kuranishi structure.
 A homeomorphism $\phi:X \to X$ is called an automorphism of Kuranishi structure if the following holds:
 For $p \in X$ and $p'=\phi(p)$, there exist  Kuranishi neighborhoods
 $(V_p, E_p, \Gamma_p, \psi_p, s_p), (V_{p'}, E_{p'}, \Gamma_{p'}, \psi_{p'}, s_{p'})$ of $p$ and $p'$
 with a group isomorphism $\rho_p:\Gamma_p \to \Gamma_{p'}$, a $\rho_p$-equivariant diffeomorphism $\phi_p:V_p \to V_{p'}$ and a $\rho_p$-equivariant bundle isomorphism $\WH{\phi}_p$ which 
 covers $\phi_p$. They are required to satisfy 
 $s_{p'} \circ \phi_p = \WH{\phi}_p \circ s_p$ and 
 $\psi_{p'} \circ \UL{\phi}_p = \phi \circ \psi_p$ for an induced homeomorphism
 $\UL{\phi}_p:s_p^{-1}(0)/\Gamma_p \to s_{p'}^{-1}(0)/\Gamma_p$ from $\phi_p$.
 
 These data should be compatible with the coordinate changes of Kuranishi structure as in \cite[A1.47]{FOOO}.
 \end{definition}
 We write $((\rho_p,\phi_p,\WH{\phi}_p),\phi)$ for an automorphism of Kuranishi structure.
  An automorphism $((\rho_p,\phi_p,\WH{\phi}_p),\phi)$  is said to be conjugate to 
 $((\rho_p',\phi_p',\WH{\phi}_p'),\phi')$ if $\phi=\phi'$ and if there exists $\gamma_p \in \Gamma_{\psi(p)}$  for each $p$ satisfying 
 $$\rho_p' = \gamma_p \cdot \rho_p \cdot \gamma_p^{-1}, \;\; \phi_p' = \gamma_p \cdot \phi_p, \;\; \WH{\phi}_p' = \gamma_p \cdot \WH{\phi}_p.$$
 
 \begin{definition}\label{def:gpkr}
 Consider an action of a finite group $G$ on $X$ and suppose $X$ is compact.
It is called an action on a space with Kuranishi structure if
each $g\in G$-action from $X$ to itself lifts to an automorphism $\phi_*$ of the Kuranishi structure, and
the composition $g_*$ and $h_*$ is conjugate to $(gh)_*$.
\end{definition}

The following lemma of \cite{FOOO} is essential in our application.
\begin{lemma}\cite[A1.49]{FOOO}\label{lemxg}
If a finite group $G$ acts on a space $X$  with Kuranishi structure, then the quotient space $X/G$ has a Kuranishi structure. 

If $X$ has a tangent bundle and the action preserves it, then the quotient space has a tangent bundle. If $X$ is oriented and the action preserves the orientation, then the quotient space has an orientation.
\end{lemma}
\begin{proof}
We only give a brief sketch, here.
Let $g \in G$ and $p \in X$. We put $G_p= \{ g \in G| g\cdot p = p \}$.
Take a Kuranishi neighborhood of $p$
$$(V_p, E_p, \Gamma_p, \psi_p, s_p)$$
such that $V_p$ is $G_p$-invariant.
We want to extend $\Gamma_p$ using $G_p$ to obtain another finite group $\Gamma_{[p]}$
which acts on $V_p$. Then, $\Gamma_{[p]}$ will serve as a local group for a Kuranishi neighborhood of $[p]$ in $X/G$.
We assume that the $G_p$-action on $V_p$ is effective, and otherwise, it can be made effective
by adding a finite dimensional representation to $V_p$ and $E_p$.

Both $G_p$ and $\Gamma_p$ can be considered as a subgroup of  the group of diffeomorphisms of $V_p$.
The extension $\Gamma_{[p]}$ is the group generated by these two subgroups $G_p, \Gamma_p$ in ${\it Diff}(V_p)$. 
From Definition \ref{def:gpkr}, there exist $ \gamma_{g_1,g_2} \in \Gamma_p$ for $g_1$ and $g_2$ in $G_p$ 
such that $$g_1 \circ g_2 = \gamma_{g_1,g_2} \circ (g_1g_2).$$
Also from Definition \ref{def:aukr}, we have
$$g \circ \gamma = \rho_{g,p}(\gamma) \circ g$$
for $g \in G_p$ and $\gamma \in \Gamma_p$.
This implies that the extension group $\Gamma_{[p]}$ satisfies the following exact sequence
$$1 \to \Gamma_p \to \Gamma_{[p]} \to G_p \to 1.$$
The $\Gamma_{[p]}$-action lifts to an action on $E_p \times V_p$, and $\psi_p:s_p^{-1}(0) /\Gamma_p \to X$
induces $\psi_{[p]}:s_p^{-1}(0) /\Gamma_{[p]} \to X/G$. Thus, we get a Kuranishi neighborhood
$(V_p, E_p, \Gamma_{[p]}, \psi_{[p]}, s_p)$ of $[p] \in X/G$.
\end{proof}

This lemma will be used to find a compatible good coordinate system and equivariant sections thereof.

First, let us consider a good coordinate system on $X$ given by a partially ordered
set $(\frak{P}, \leq)$, which parametrizes a collection of Kuranishi neighborhoods covering $X$. (Technically, we require that this system additionally satisfy properties in \cite[Definition 5.3]{FOOO3}.)
Namely, given $\frak{p} \in \frak{P}$  we have a corresponding Kuranishi neighborhood $(V_{\frak p}, E_{\frak p}, \Gamma_{\frak p}, \psi_{\frak p}, s_{\frak p})$, and partial order describes directions of
embeddings.

We say that a good coordinate system on $X$ is compatible with the $G$-action on $X$ if
the following holds:
\begin{enumerate}
\item For $\frak{p} \in \frak{P}$ and  $g\in G$, 
there exist $\frak{p}' \in \frak{P}$ such that $g_*$ sends the Kuranishi neighborhood for $ \frak{p}$
to that of $\frak{p}'$ (this defines a map $g_* :\frak{P} \to \frak{P}$ such that $\frak{p'} = g_\ast (\frak{p})$);
\item The map $g_*$ preserves the partial order of $\frak{P}$.
\end{enumerate}

We owe the following lemma to Kenji Fukaya.
\begin{lemma}\label{lem:fukaya}
Let $X$ be a space with Kuranishi structure together with a $G$-action on it.
Then, there exist a good coordinate system compatible with the $G$-action on $X$.
\end{lemma}
\begin{proof}
A good coordinate system is chosen by the induction on the dimension of $V_p$ and by gluing several Kuranishi neighborhoods (in order to get a bigger one satisfying desired properties).
From the previous lemma, $X/G$ is also a space with Kuranishi structure, and hence admits a good coordinate system. Then, we glue the corresponding charts in $X$ in the same way we construct the good coordinate system for $X/G$. This produces a desired good coordinate system on $X$ which is compatible with the $G$-action.
\end{proof}

We now explain how to choose multi-sections.
If $\Psi$ is an automorphism of a space with Kuranishi structure $\Psi$ and $s$ is a multi-section, then we can pull back $s$ by $\Psi$ in an obvious way. Write the pulled-back multi-section by $\Psi^*s$.
\begin{definition}
A multisection  $s$ is said to be $G$-equivariant if
for each $g \in G$ the corresponding automorphism of Kuranishi structure $g_*$
sends $s$ to itself.
\end{definition}
If a multisection $s$ is $G$-equivariant, the zero set of $s$ admits the obvious $G$-action.
Equivariant multi-sections can be constructed from a multi-section for $X/G$.
Virtual fundamental chains obtained in this way admit natural $G$-actions, and evaluation maps become $G$-equivariant as well.
 \begin{lemma}
Let $X$ be 
a space with Kuranishi structure equipped with $G$-action and fix a $G$-compatible good coordinate system.
Then the family of multisection $\{ s'_{p,\epsilon} \}$ in Lemma \ref{lem:multisecch} can be chosen to be $G$-equivariant.
\end{lemma}
\begin{proof}
This is again done by choosing a multi-section for $X/G$ first. We examine how to do it in local charts. Let $(V_p, E_p, \Gamma_{[p]}, \psi_{[p]}, s_p)$ be a Kuranishi neighborhood for $[p] \in X/G$. Given a multi-section $s_{p', \epsilon} : V_p \to \mathcal S^n(E_p)$ (which is $\Gamma_{[p]}$-equivariant by definition), the same multi-section $s_{p', \epsilon}$ can be regarded as a multi-section for the Kuranishi neighborhood $(V_p, E_p, \Gamma_p, \psi_p, s_p)$ of $p \in X$ since it is $\Gamma_p$-equivariant also.
The $\Gamma_{[p]}$-equivariance implies that the multi-section $s_{p', \epsilon}$ is
$G_p$-equivariant.
\end{proof}

We finally deal with transversality of evaluation maps from moduli spaces. We want to choose multi-sections suitably so that evaluation maps become submersive. Fukaya, Oh, Ohta and Ono \cite{FOOOT2} defined Lagrangian Floer theory on de Rham model using continuous families of multi-sections, and constructed smooth correspondences using such gadgets. 
It is straightforward to modify the construction of smooth correspondences by using a continuous family of multi-sections which fits into our $G$-equivariant setting.
Let us illustrate it for the local chart, and leave details to the reader.

Consider a chart  
$(V_p, E_p, \Gamma_p, \psi_p, s_p)$ from Lemma \ref{lemxg} with $\Gamma_{[p]}$ as before. Let $f_p:V_p \to M$ be a submersion. A vector space $W_p$ is chosen to have sufficiently large dimension so that there exists a surjective bundle map 
$Sur:W_p \times V_p \to E_p.$
The extended space $W_p$ is given the trivial $\Gamma_{[p]}$-action.
We put
$$\mathfrak{s}_p^{(1)} (w,x) = Sur (w,x) + s_p (x).$$
It is no longer $\Gamma_{[p]}$-equivariant, so we take (family of) multi-sections 
\begin{equation}\label{mult}
\mathfrak{s}_p(w,x) :=\mathfrak{s}_p^{(2)} (w,x) = [\gamma_1\mathfrak{s}_p^{(1)} (w,x), \cdots, \gamma_g \mathfrak{s}_p^{(1)} (w,x) ]
\end{equation} 
where $\Gamma_{[p]} = \{\gamma_1, \cdots, \gamma_g\}$.
Then, $\mathfrak{s}_p$ is transversal to $0$ and the restriction of
$f_p \circ \pi_p : W_p \times V_p \to M$
to $\cup_i \{ (w,x) \, | \, \mathfrak{s}_{p,i} (w,x) = 0 \}$ is a submersion.
For $g \notin G_p$, we can choose $W_p=W_{g(p)}$ and make the related construction done in $g$-equivariant way.

Recall from \cite{FOOOT2} that smooth correspondence is defined locally as follows.
Let $\theta_p$ be a compactly supported smooth differential form on $V_p$, which is $\Gamma_{p}$-invariant. We put smooth measure $\omega_p$ on $W_{p}$ so that
$\int_{W_p} \omega_p=1$. 
(For $g \notin G_p$, we take $\omega_{g p} = (g^{-1} )^\ast \omega_p$.)
\begin{definition}\label{localcorres}
We define
\begin{equation}
\begin{array}{rl}
&\left(  ( V_{p},E_{p},\Gamma_{p}, \psi_{p},s_{p}), (W_p, \omega_p), \mathfrak{s}_{p}, f_p  \right)_\ast (\theta_p) \\
&:=\dfrac{1}{|\Gamma_p|} \sum_{j=1}^{l} \dfrac{1}{l} \left(f_p \circ \pi_p|_{\mathfrak{s}_{p,j }^{-1} (0)} \right)_! (\pi_p^\ast \theta_p \wedge \omega_p)|_{\mathfrak{s}_{p,j }^{-1} (0)}
\end{array}
\end{equation}
where $l$ is the number of branches of $\mathfrak{s}_p$. This is a $G_p$-equivariant map.
\end{definition}
We can put various $W_p$'s compatibly in the good coordinate system by using almost the same technique as in \cite{FOOOT2}.

\section{Group actions on $\AI$-algebra of a Lagrangian Submanifold}
Let $L$ be a (relatively) spin Lagrangian submanifold in a symplectic manifold $(M,\omega)$.
The Floer cochain complex between $L$ and itself has a structure of an $\AI$-algebra as observed in \cite{FOOO}. If the finite group $G$ acts on $M$ and preserves $L$,
then we can construct a group action on the $\AI$-algebra of $L$ which is compatible with $\AI$-operations. The construction does not involve spin profiles and $G$-Novikov theory. We only need  equivariant transversality in the previous section 
with a little bit of algebraic manipulation. We will mainly discuss the latter in this section.

The actual construction of $\AI$-algebra involves a long sequence of technical constructions. We do not reproduce the whole steps here, but only indicate how to adapt constructions in \cite{FOOO} in our setting, omitting most of the steps which can be done by straightforward modification of theirs.
We remark that the ($m_1$-)cohomology of the $\AI$-algebra is the Floer cohomology $HF(L,L)$, and in particular, spin profiles automatically match. This provides a explanation on why the spin profile condition is not necessary in this case. Moreover, we do not have to reduce the $G$-action to the action of the energy zero subgroup $G_\alpha$. The reason behind it will be discussed in Section \ref{sec:12}.


Suppose that a finite group $G$ acts on $M$ effectively, preserving the symplectic form $\omega$, almost complex structure $J$, a spin Lagrangian submanifold $L$ and its isomorphism class of spin structure. Hence, $G$ acts on the moduli spaces of stable $J$-holomorphic discs bounding $L$,  and the local Kuranishi model for a stable map is identical under group action.

As long as two spin structures
$(A_{g-1})^*P_{spin}(L)$ and  $P_{spin}(L)$ are isomorphic, the induced orientations
on the moduli spaces of discs are the same, and the choice of the particular
isomorphism $A_g^{spin}$ do {\em not}  play a role, here.
Thus, the following
proposition follows from the construction of \cite{FOOO} and $G$-Kuranishi structure built in the previous section. Let $\beta \in H_2(M,L;\Z)$, and $k \geq 1$.
\begin{prop}\cite[Theorem 2.1.29]{FOOO}
The moduli space $\CM_{k,l}^{main}(L, J, \beta)$ of stable $J$-holomorphic discs with boundary on $L$ with $k$ boundary marked points (cyclically ordered) and $l$ interior marked points has a Kuranishi structure with an action of $G$ in the sense of Definition \ref{def:gpkr}.
We can choose family of $G$-equivariant multisections $\mathfrak{s}$ for which the moduli space has a virtual fundamental chain of
dimension $m + \mu_L(\beta) + k +2l -3$. 
\end{prop}

Let $\Phi_{\beta,k}:=(ev_0,\cdots, ev_k):\CM_{k+1}(L,\beta) \mapsto (L)^{\times (k+1)}$ be the boundary evaluation map. Then, the $g$-action induces an isomorphism
$(A_g)_*:\CM_{k+1}(L,\beta) \to \CM_{k+1}(L,\beta)$
with commuting diagram
\begin{equation}\label{gactionoricomm}
\xymatrix{ \CM_{k+1}(L,\beta) \ar[d]_{(A_g)_*} \ar[rr]^{\Phi_{\beta,k}}&& L \times \cdots \times L \ar[d]_{A_g} \\
 \CM_{k+1}(L,\beta)  \ar[rr]^{\Phi_{\beta,k}}&& L \times \cdots \times L }
\end{equation}
where $A_g$ in the right column is the diagonal action of $g$.

For convenience, we work on the de Rham model of an $\AI$-algebra, whose construction is described in \cite{FOOOT2}.
A de Rham model $\AI$-algebra has been extensively used in the toric setting recently, where Lagrangian submanifolds are $T^n$-orbits.  $T^n$-equivariant evaluation
maps are automatically submersive, and hence was directly used to define the push-forward of differential forms.

Consider the de Rham complex $\Omega(L)$ of $L$. A multi-linear map
$$m_{k,\beta} : \overbrace{\Omega(L) \otimes \cdots \otimes \Omega(L)}^{k} \to \Omega(L)$$
is defined as follows:
$$m_{k,\beta} (\rho_1, \cdots, \rho_k) = (ev_0)_! (ev_1, \cdots, ev_k)^\ast (\rho_1, \cdots, \rho_k).$$
where $(ev_0)_!$ is an integration along fibers if $ev_0$ is a submersion. In general, it is the $G$-equivariant smooth correspondence described in the previous section.
(See the diagram below.)
\begin{equation}
\xymatrix{
 L && \ar[ll]_{ev_0} \mathcal{M}_{k+1}^{main} (L;\beta)^{\mathfrak{s}} \ar[rr]^{\Phi_{\beta,k}} & &  {\overbrace{L\times \cdots \times L}^{k}} }
\end{equation}

We put
$$m_k  = \sum_{\beta} T^{\omega (\beta) } m_{k,\beta} \quad (\mbox{for} \,\, k \neq 1), \qquad m_1(\rho_1) = (-1)^{\deg \rho_1+ n+1}d\rho_1 +  \sum_{\beta} T^{\omega (\beta) } m_{1,\beta}(\rho_1).$$
As shown in \cite{FOOOT2}, $\{m_k\}_{k \geq 0}$ defines an $A_\infty$-structure on $\Omega(L)$. i.e. the family of operations satisfying
\begin{equation}
\sum_{k_1 + k_2 = k+1} \sum_i (-1)^\ast \, m_{k_1} \left(\rho_1, \cdots, m_{k_2} (\rho_i, \cdots, \rho_{i+k_2 -1} ) , \cdots, \rho_k \right) = 0
\end{equation}
where $\ast = |\rho_1| + \cdots + |\rho_{i-1}| +i-1$.

Note that if $E$ and $M$ are manifolds with $G$-actions, and if $f:E \to M$ is a $G$-equivariant  submersion,
then the integration along fiber $f_!$ is well-defined on the cochain level, and is $G$-equivariant. Therefore, we obtain the following lemma.
\begin{lemma}
The operations $\{m_k\}$ are $G$-equivariant in the sense that
\begin{equation}\label{eq:grho}
m_{k,\beta} ( g\cdot \rho_1 , \cdots, g\cdot \rho_k) = g\cdot m_{k,\beta} (\rho_1, \cdots \rho_k).
\end{equation}
where $g \cdot \rho = (g^{-1})^* \rho$
\end{lemma}

Due to the problem ``running out of Kuranishi neighborhood" \cite[Section 7.2]{FOOO}, 
one first construct a $G$-equivariant $A_N$-algebra of the
Lagrangian submanifold $L$ for a large $N$. Then, the desired $A_\infty$-algebra is obtained by further applying algebraic formalism as follows. (See \cite[Theorem 7.2.27]{FOOO} for more details.) 

In what follows, we will use the (tensor) coalgebra language which much reduces the notational complexity.
$m_k$ gives rise to the coderivation $\hat{m_k}$, and $\hat{d} := \sum_{k=0}^\infty \hat{m_k}$ on the tensor coalgebra $T \Omega(L)$.
We write $\bx = x_1\otimes  \cdots \otimes x_k$, and $g(\bx) = g(x_1) \otimes \cdots \otimes g(x_k)$.
An $\AI$-homomorphism $f$ induces a cohomomorphism $\hat{f}$ between coalgebras.
One can average an $\AI$-homomorphism to make it $G$-equivariant in the following way.
\begin{lemma}\label{avgAhom}
Let $\mathfrak{f} = \{f_i\}$ be an $A_\infty$ homomorphism between two $A_\infty$ algebras $C_1$ and $C_2$ with $G$-actions. The average $\hat{\mathfrak{f}}_{avg}$ is defined by
$$\hat{\mathfrak{f}}_{avg} ({\bx}) := \frac{1}{|G|} \sum_{g} g \hat{\mathfrak{f}} (g^{-1} {\bx}).$$
Then, $\hat{\mathfrak{f}}_{avg}$ gives a $G$-equivariant $A_\infty$ homomorphism.
\end{lemma}
\begin{proof}
$\hat{\mathfrak{f}}_{avg}$ is $G$-equivariant by construction.
We need to prove that  $\hat{\mathfrak{f}}_{avg} \circ \hat{d} = \hat{d} \circ \hat{\mathfrak{f}}_{avg}$.
From \eqref{eq:grho}, $\hat{d}$ commutes with the $g$-action. Also, $\hat{d}$ commutes with $\hat{\mathfrak{f}}$ since
$f$ is an $\AI$-homomorphism. Hence, the lemma follows.
\end{proof}

\begin{remark}
Averaging of $\AI$-operations of an non-$G$-equivariant $\AI$-algebra does not produce
an $\AI$-algebra. 
\end{remark}

The lemma below follows immediately.
\begin{lemma}\label{compavg}
$(\hat{\mathfrak{h}} \circ \hat{\mathfrak{f}})_{avg} = \hat{\mathfrak{h}}_{avg} \circ \hat{\mathfrak{f}}$ and $(\hat{\mathfrak{f}} \circ \hat{\mathfrak{h}})_{avg} = \hat{\mathfrak{f}} \circ \hat{\mathfrak{h}}_{avg}$ provided that $\mathfrak{f}$ is $G$-equivariant (i.e. $\hat{\mathfrak{f}}_{avg} = \hat{\mathfrak{f}}$).
\end{lemma}

We introduce the notion of $A_\infty$-homotopies in the existence of the group action.
Let $(C[1], \{m_k\})$ be an $A_\infty$-algebra over $R$ on which a finite group $G$ acts and suppose that each $m_k$ is $G$-equivariant. ($[1]$ means the degree shifting.)
The $\AI$-homotopy is defined by employing a model of $C \times [0,1]$. When $R \supset \R$, the $A_\infty$-algebra $C[1] \otimes R[t,dt]$ can be used as a model. (See \cite{Fuk2} for extended $\AI$-operations on $C[1] \otimes R[t,dt]$.) A $G$-action on $C[1] \otimes R[t,dt]$ is defined by giving a trivial
$G$-action on $R[t,dt]$ factor. i.e. the $G$-action on $C[1] \otimes R [t,dt]$ is $R[t,dt]$-linear. It is easy to see that $A_\infty$ operations on $C[1] \otimes R[t,dt]$ are $G$-equivariant.

Recall that we have an $A_\infty$ homomorphism ${\rm Eval}_{t=t_0} : C[1] \otimes R[t,dt] \to C[1]$ given by
$${\rm Eval}_{t=t_0} (P(t) + Q(t) dt) = P(t_0).$$
All other higher components of ${\rm Eval}_{t=t_0}$ are defined to be zero. ${\rm Eval}_{t=t_0}$ is obviously $G$-equivariant and so is $\hat{{\rm Eval}}_{t=t_0}$.

\begin{definition}
Two ($G$-equivariant) $A_\infty$ homomorphisms $\phi, \phi' : C \to C'$ are said to be ($G$-)homotopic to each other, if there exists a ($G$-equivariant) $A_\infty$ homomorphism $H : C \to C' \otimes R[t,dt]$ such that ${\rm Eval}_{t=0} \circ H = \phi$ and ${\rm Eval}_{t=1} \circ H = \phi'$.
\end{definition}

\begin{lemma}\label{equivhpty}
Suppose that $C$ and $C'$ are both equipped with $G$-equivariant $A_\infty$ structures and $\phi : C \to C'$ is $G$-equivariant. If $\phi$ and $\phi'$ are homotopic, then so are $\phi$ and $\phi'_{avg}$. Moreover, one can take a $G$-equivariant homotopy between $\phi$ and $\phi'_{avg}$.
\end{lemma}
\begin{proof}
Let $H$ be a homotopy between $\phi$ and $\phi'$. Then, by taking ``hat" and averaging the equations ${\rm Eval}_{t=0} \circ H = \phi$ and ${\rm Eval}_{t=1} \circ H = \phi'$, we get
$$\hat{{\rm Eval}}_{t=0} \circ \hat{H}_{avg} = \hat{\phi} \qquad \hat{{\rm Eval}}_{t=1} \circ \hat{H}_{avg} = \hat{\phi}'_{avg}.$$
\end{proof}
We are now ready to state the $G$-equivariant version of the $A_\infty$-Whitehead theorem.
\begin{theorem}
Let $\mathfrak{f}: C_1 \to C_2$ be a $G$-equivariant weak homotopy equivalence between two filtered $A_\infty$ algebras $C_1$ and $C_2$ with $G$-actions. Then, there exists a quasi-inverse $\mathfrak{h} : C_2 \to C_1$ of $\mathfrak{f}$ which is also $G$-equivariant.
\end{theorem}

\begin{proof}
Usual Whitehead theorem for $A_\infty$-algebras in \cite{FOOO} produces a quasi-inverse $\tilde{\mathfrak{h}}$ of $\mathfrak{f}$. By averaging $\mathfrak{h}$ as in Definition \ref{avgAhom}, we obtain the $A_\infty$ homomorphism $\mathfrak{h}_{avg}$ which is $G$-equivariant. 
Since $\mathfrak{f}$ is $G$-equivariant, we may apply Lemma \ref{compavg} so that $(\hat{\mathfrak{h}} \circ \hat{\mathfrak{f}})_{avg} = \hat{\mathfrak{h}}_{avg} \circ \hat{\mathfrak{f}}$ and $(\hat{\mathfrak{f}} \circ \hat{\mathfrak{h}})_{avg} = \hat{\mathfrak{f}} \circ \hat{\mathfrak{h}}_{avg}$.
Note that $\hat{id}_{avg} = \hat{id}$. Then, by Lemma \ref{equivhpty}, we get $G$-equivariant homotopies
$$ \hat{\mathfrak{h}}_{avg} \circ \hat{\mathfrak{f}} \simeq \hat{id}_{C_1}, \qquad 
 \hat{\mathfrak{f}} \circ \hat{\mathfrak{h}}_{avg} \simeq \hat{id}_{C_2}.$$
\end{proof}
In order to construct $\AI$-algebra from $A_N$ (or $A_{n,K}$-algebra in the filtered case) the following
theorem is essential.
\begin{theorem} (c.f. \cite[Theorem 7.2.72]{FOOO})
Let $C_1$ be a filtered $A_{n,K}$ algebra and $C_2$ a filtered $A_{n',K'}$ algebra with $(n,K) < (n',K')$, both of which are assumed to be  $G$-equivariant.  Let $\mathfrak{h} : C_1 \to C_2$ be a $G$-equivariant filtered $A_{n,K}$ homomorphism which gives rise to a filtered $A_{n,K}$ homotopy equivalence.

Then, there exist a $G$-equivariant filtered $A_{n',K'}$ algebra structure on $C_1$ extending the given $A_{n,K}$ algebra structure and a $G$-equivariant filtered $A_{n',K'}$ homotopy equivalence $C_1 \to C_2$ extending the given filtered $A_{n,K}$ homotopy equivalence $\mathfrak{h}$.
\begin{equation*}
\xymatrix{
C_1^{(n',K')} \ar[drr]_{A_{n',K'}} &&\\
&&C_2\\
C_1^{(n,K)} \ar[uu] \ar[urr]^{A_{n,K}}_{\mathfrak{h} : \simeq}  && }
\end{equation*}
\end{theorem}
\begin{proof}
We just mention how to modify the original proof of \cite{FOOO}, which deals with the obstruction class of extending $A_{n,K}$ structure to the next level. Its $G$-equivariant version will be explained only for the unfiltered case for simplicity.
Given an $A_N$ structure with $\{m_k\}_{k=1}^N$, the subsequent operation $m_{N+1}$ should be
defined to satisfy the
next $\AI$-equation 
\begin{equation}\label{eqobstnk}
m_1 \circ m_{N+1}(\bx) + m_{N+1} \circ m_1(\bx) + \sum_{k_1 + k_2=N+1} (-1)^{\deg' \bx_1} m_{k_1}(\bx_1, m_{k_2}(\bx_2),\bx_3)=0.
\end{equation}
Thus, $m_{N+1}$ may be regarded as a $\delta$-cocycle in the cochain complex $$(Hom(B_NC[1], C[1]), \delta = [m_1,\cdot]).$$
In addition, if $A_N$ algebra is $G$-equivariant, the left hand side of 
\eqref{eqobstnk} defines a $G$-equivariant map from $B_NC[1]$ to $C[1]$, and hence
 lies in the subcomplex $(Hom_G(B_NC[1], C[1]), \delta)$. 
 
 Note that $G$-equivariant $A_N$ homotopy equivalence $\mathfrak{h}$ preserve the obstruction class $[m_{N+1}]$, and that the obstruction vanishes if the target of $\mathfrak{h}$ has a $G$-equivariant $A_{N'}$ structure with $N' > N$.
Therefore, we can choose a $G$-equivariant $m_{N+1}$ which satisfies \eqref{eqobstnk} for $C_1$. In a similar manner, the original construction
 of \cite{FOOO} can be modified to include the $G$-equivariance as above.
\end{proof}
Following the rest of \cite{FOOOT2}, we obtain
\begin{theorem}
There exist a $G$-equivariant $\AI$-algebra $\Omega(L) \widehat{\otimes} \Lambda_{nov}$ of a spin Lagrangian submanifold $L$, and
its $G$-invariant part again becomes an $\AI$-algebra, which will be called
the $\AI$-algebra of $[L/G]$ in $[M/G]$.
\end{theorem}
The $G$-equivariance implies that If $\rho_1, \cdots, \rho_k$ are all $G$-invariant, then so is $m_{k,\beta} (\rho_1, \cdots, \rho_k)$. Therefore, we obtain an operation
$$m_{k,\beta}| : \overbrace{\Omega(L)^G \otimes \cdots \otimes \Omega(L)^G}^{k} \to\Omega(L)^G $$
simply by restricting $m_{k,\beta}$ to the set of $G$-invariant differential forms on $L$. This gives rise to  a well-defined $A_\infty$-structure on 
$\Omega(L)^G$. 

Within this setting, we may define $G$-bounding cochain as a bounding cochain which is $G$-invariant. $G$-bounding cochains define boundary deformations
of $\AI$-algebra of $[L/G]$.
\begin{remark}
The average of a bounding cochain does not become a $G$-bounding cochain.
\end{remark}

The $\AI$-algebra of $[L/G]$ constructed here is only concerned with smooth $J$-holomorphic
discs with boundary on $L$. In order to obtain richer orbifold theory, one should additionally consider orbi-discs or bulk deformations by twisted sectors, which are introduced by the first author and Poddar \cite{CP} in the case of toric orbifolds.

\section{$G$-equivariant $A_\infty$-bimodule $CF^\ast_{R, l_0}(L_0, L_1)$}\label{sec:12}
We consider equivariant $A_\infty$-bimodules and consider Lagrangian Floer homology
of a pair. We first review the $\AI$-bimodule structure briefly. Recall that a boundary operator $\delta : CF_{R,l_0}^\ast(L_0,L_1)\to CF_{R,l_0}^\ast(L_0,L_1)$ was defined in \cite{FOOO} by
\begin{equation}\label{eqfldiff}
\delta ([w, l_p]) = \sum_{ \mu ([w',l_q]) - \mu ([w,l_p]) =1 } \# \mathcal{M} (L_1, L_0 ; [w, l_p ], [w',l_q]) \,\, [w', l_q].
\end{equation}
Here, $\mathcal{M} (L_1, L_0 ; [w,l_p], [w',l_q])$ is the moduli space of $J$-holomorphic strips
$u$ such that $[w \star u, l_q] = [w',l_q]$. The orientations of these moduli spaces are already well-studied by \cite{FOOO}.
\begin{theorem}\label{ori}  \cite[8.1.14]{FOOO} Suppose that a pair of Lagrangian submanifolds $(L_0, L_1)$ is (relatively) spin. Then for any $p,q \in L_0 \cap L_1$, the moduli space $\mathcal{M} (L_1, L_0 ; [w,l_p], [w',l_q])$ of connecting orbits in Lagrangian intersection Floer cohomology is orientable. Furthermore, orientations on $\Theta_{[w,l_p]}^-$ and $\Theta_{[w',l_q]}^-$  and related spin structures canonically determine the orientation on $\mathcal{M} (p,q)$.
\end{theorem}

In general, disc bubbling phenomenon produce non-trivial $m_0$ term, and hence $\delta^2 \neq 0$.
The general algebraic structure one obtains is in fact $\AI$-bimodule over $\AI$-algebras of  $L_0$ and $L_1$.
Recall that $\AI$-bimodule $M$ over $\AI$-algebras $C_1,C_2$ are given by
sequence of maps $$\{ n_{k_1,k_2}: B_{k_1}C_1 \otimes M \otimes B_{k_2}C_2 \to M\}_{k_1,k_2 \geq 0},$$
satisfying an $\AI$-bimodule equation.
Such an $\AI$-bimodule is called $G$-equivariant, if $G$ acts on $M$ and $C_i$ linearly and each $n_{k_1,k_2}$ satisfies 
$$g \cdot n_{k_1,k_2}(a_1,\cdots, a_{k_1}, \xi, b_1,\cdots, b_{k_2}) = 
n_{k_1,k_2}(g \cdot a_1,\cdots, g \cdot a_{k_1}, g \cdot \xi, g \cdot b_1,\cdots, g \cdot b_{k_2}).$$

Now, suppose $L_0$ and $L_1$ admit $G$-actions inherited from the one on the ambient symplectic manifold. Assume further that the spin structures on $L_0$ and $L_i$ are $G$-invariant. i.e. they are isomorphic to their pull-back under any element of $G$. Then, the spin bundle  on $L_i$ gives a spin profile $s_i \in H^2 (G; \ZZ /2)$ for $i=0,1$. Assume $s_0=s_1$.

\begin{remark}
Since we will only consider the action of the energy zero subgroup $G_\alpha$ of $G$ for $L_0$ and $L_1$, the condition $s_0=s_1$ can be weakened to the agreement of $s_0$ and $s_1$ when restricted to $H^2 (G_\alpha ; \ZZ /2)$.
\end{remark}

As studied in Section \ref{sec:Ku}, the union of all moduli spaces $\mathcal{M} (L_1, L_0 ; [w,l_p], [w',l_q])$ with a fixed symplectic energy, has a Kuranishi structure with a $G_\alpha$-action. Therefore, $\delta=n_{0,0}$ is $G_\alpha$-equivariant and the homology of
$\big(CF_{R,l_0}^\ast(L_0,L_1), \delta \big)$ admits a natural $G_\alpha$-action if $\delta^2=0$.

\begin{definition}\label{def:flwithout}
With the setting as above, the $G_\alpha$-invariant part of the homology is called the Lagrangian Floer cohomology of the pair $([L_0/G], [L_1/G];l_0)$ and will be written by 
$$HF^\ast_{R,l_0}([L_0/G], [L_1/G]).$$ We take the sum over all possible $l_0$ to define Lagrangian Floer cohomology of the
pair $([L_0/G], [L_1/G])$.
\end{definition}

\begin{remark}
If $G \neq G_\alpha$, we introduced orbifold Novikov ring $\Lambda_{nov}^{orb}$ in Subsection \ref{subsec:newNov},
and explained that $\Lambda_{nov}^{orb}$ is not  visible if we use the universal Novikov field.
\end{remark}

As in the previous section, it is not hard to adapt the construction of \cite{FOOO} (or \cite{FOOOT2} for the de Rham version) to define $G_\alpha$-equivariant $\AI$-bimodules for the pair $(L_0, L_1)$. Again using equivariant Kuranishi perturbations, and adapting the construction of \cite{FOOO} as in the previous section,  one can prove
\begin{prop}\label{prop:12.4}
We have a $G_\alpha$-equivariant filtered $\AI$-bimodule  $(CF_{R,l_0}^\ast(L_0,L_1), \{n_{k_1,k_2}\})$
with $n_{0,0} = \delta$ over $G_\alpha$-equivariant $\AI$-algebras of $L_0$ and $L_1$.
\end{prop}


Now, let us consider the case that $L_1$ is obtained by Hamiltonian diffeomorphism of $L_0$, i.e. $L_1=\phi_{H}^1(L_0)$ for a time-dependent Hamiltonian $H$. If $H:M \times [0,1] \to \R$ is not $G$-invariant, the corresponding Hamiltonian perturbation produces a $G$-equivariant Lagrangian immersion which will be studied in \cite{CH2}. Here, we only consider $H:M \times [0,1] \to \R$ which are $G$-invariant. Here, the $G$-invariance of $H$ means that each $H_t$ for $t \in [0,1]$ is a $G$-invariant function.
This implies that $\phi_{H}^1$ is $G$-equivariant, and hence, for example, $g$-fixed points of $M$ can
only move to a $g$-fixed points by  $\phi_{H}^1$. $\phi_H^1 (L_0)$ is, then, preserved by $G$. Since $L_0$ and $\phi_H^1 (L_0)$ are equivariantly isotopic, one can identify their spin bundles and liftings of $G$-actions on spin bundles. Thus, spin profiles of $L_0$ and $\phi_H^1 (L_0)$ coincide.

The energy zero subgroup $G_\alpha \subset G$ for the pair $(L_0, \phi_H^1(L_0))$ is particularly simple. Indeed, we will show that $G_\alpha$ becomes the entire group $G$. We first choose a candidate for energy zero path for each $g \in G$ as follows.

If we define $\WT{H} : M \to \RR$ by
$$\WT{H} (x) = \int_{0 \leq t \leq 1} (\phi^t_H)^\ast H_t (x ) dt=\int_{0 \leq t \leq 1} H_t (\phi^t_H (x) ) dt,$$
then $\WT{H}$ is also $G$-invariant (since $\phi^t_H$ is $G$-equivariant.) Fix a generic point $x_0$ in $L_0$ and choose for each $g \in G$ a path $\gamma_g$ in $L_0$ from $x_0$ to $g \cdot x_0$. We take a base path $l_0$ from $L_0$ to $\phi_H^1 (L_0)$ to be $l_0 (t) := \phi^t_H (x_0)$. Then, the strip $w_g : [0,1]^2 \to M$ defined by
$$w_g (t,s) = \phi^t_H (\gamma_g (s) )$$
gives a path from $l_0$ to $g(l_0)$ in $\Omega(L_0, \phi_H^1(L_0);l_0)$. We claim that this path has zero energy.

\begin{lemma}
$w_g$ is an energy zero path.
\end{lemma}

\begin{proof}
We have to show that $\int_{w_g} \omega=\int_{[0,1]^2} w_g^\ast \omega =0$. Observe that
\begin{eqnarray*}
w_g^\ast \omega \left(\frac{\partial}{\partial t}, \frac{\partial}{\partial s} \right)_{(t,s)} &=& \omega \left( (X_{H_t})_{\phi^t_H (\gamma_g(s))}, (\phi_H^t)_\ast \gamma_g'(s) \right) \\
&=& dH_t ((\phi_H^t)_\ast \gamma_g'(s))_{\phi^t_H (\gamma_g (s))} \\
&=& d ( (\phi_H^t )^\ast H_t ) (\gamma_g'(s))_{\gamma_g(s)} \\
&=& \frac{d}{ds} (\phi_H^t )^\ast H_t \, (\gamma_g(s)) 
\end{eqnarray*}
Therefore, the energy of $w_g$ is given by
\begin{eqnarray*}
\mathcal{A} (w_g) &=& \int_{0 \leq s \leq 1} \int_{0 \leq t \leq 1} \left( \frac{d}{ds} (\phi_H^t )^\ast H_t \, (\gamma_g(s)) \right ) dt ds \\
&=&  \int_{0 \leq s \leq 1}  \frac{d}{ds} \left(  \int_{0 \leq t \leq 1} (\phi_H^t )^\ast H_t \, (\gamma_g(s))  dt \right ) ds \\
&=& \int_{0 \leq s \leq 1}  \frac{d}{ds} \left( \WT{H} (\gamma_g (s) ) \right) ds \\
&=& \WT{H} (\gamma_g (1) ) - \WT{H} ( \gamma_g (0)) =0
\end{eqnarray*}
since $\WT{H}$ is $G$-invariant and $\gamma_g (1) = g \cdot \gamma_g (0)$.
\end{proof}
It directly follows from the lemma that
\begin{corollary}
With the setting as above, the energy zero subgroup for a pair $(L_0, \phi_H^1 (L_0))$ is the entire group $G$.
\end{corollary}


Recall that the $\AI$-algebra is a bimodule over itself. So we can compare two $\AI$-bimodules $\Omega(L_0)$ and $CF^\ast_{R,l_0} (L_0, L_1)$ over $\AI$ algebra of $L_0$. The following proposition of \cite{FOOO} can be also proved in a $G$-equivariant setting along the same line of  their proof, and we omit the details. 
\begin{prop}
Suppose that $L_0$ is connected 
and let $L_1=\phi_H^1 (L_0)$ as above .
Then there is a $G$-equivariant filtered $\AI$-bimodule quasi-isomorphism
\begin{equation}\label{continuationG}
\Omega(L_0) \widehat{\otimes} \Lambda_{nov} \to CF_{R,l_0}^\ast(L_0,L_1) \otimes_{\Lambda (L, L_1;l_0)}\Lambda_{nov}.
 \end{equation}
\end{prop}
In particular, this gives the following Lagrangian intersection theoretic results.
If the $G$-invariant $\AI$-algebra of $L_0$ has a non-vanishing homology,
$L_0$ and $\phi_{H}^1 (L_0)$ have  a non-trivial 
intersection for any time-dependent Hamiltonian isotopy $\phi_H$.

We remark that Seidel already observed in $\Z/2$-action of exact case that
such a Lagrangian Piunikhin-Salamon-Schwarz morphism exists and becomes an isomorphism,
as well as Hamiltonian invariance of equivariant Lagrangian Floer homology. (see the last paragraph of section (14b) of \cite{Se}).\\

\section{Equivariant flat vector bundles on $G$-invariant Lagrangians}\label{sec: orbibundle}
In this section, we explain how to add $G$-equivariant flat vector bundles to the
theory. In homological mirror symmetry, one needs to consider rank 1 flat unitary bundles on a Lagrangian submanifold. With a finite group action on the Lagrangian submanifold, it is natural to introduce $G$-equivariant structures on such line bundles. 
We also observe that we can define a natural action of a character group $\hat{G}$ (which is the dual group $G^*$ for abelian $G$)  on both equivariant and orbifolded Fukaya categories by twisting $G$-equivariant structures on flat
vector bundles. Such character group action comes into play in $G$-equivariant homological mirror symmetry, which will be studied in the next section.  We remark that such character group action will be also discussed in our joint work
with Siu-Cheong Lau in preparation.

We will assume that  $G=G_\alpha$ for simplicity from now on.
We first review the standard way of including these flat bundles into Lagrangian Floer theory (without group actions). For a unitary bundle $U_i$ on a Lagrangian $L_i$ ($i=0,1$), the new Floer complex $CF_{l_0}( (L_0,U_0), (L_1,U_1))$ is defined by replacing each generator $[w,l_p]$ of  $CF_{\C,l_0}^\ast(L_0,L_1)$
with $[w,l_p] \otimes \Hom (U_0|_p, U_1|_p)$. The Floer differential $\delta$ in \eqref{eqfldiff}
is modified by additional contributions from holonomies. Namely, given $\lambda_p \in \Hom (U_0|_p, U_1|_p)$ and a $J$-holomorphic strip $u$ which maps $[w,l_p]$ to $[w \sharp u, l_q]$,
we get $\lambda_q \in \Hom (U_0|_q, U_1|_q)$ by composing $\lambda_p$ with holonomies around two boundary components of $u$ appropriately.

\subsection{Equivariant structures}
Consider a trivial complex vector bundle $L \times \C^n \to L$ of rank $n$ where $G$ acts on $L$. A $G$-equivariant structure on this bundle 
means the choice of a $G$-action on $\C^n$, or equivalently the choice of a representation $\theta:G \to End(\C^n)$. Such a homomorphism induces the diagonal $G$-action on $L \times \C^n$.
If $\chi:G \to U(1)$ is a character of $G$, then one can twist $\theta$ by $\chi$ and obtain $\theta^\chi:G \to End(\C^n)$ given by $\theta^\chi(g) = \chi(g)\theta(g)$.

In general,  let $U \to L$ be a possibly non-trivial vector bundle on $L$. A $G$-equivariant structure on $U$ is given by the choice of an isomorphism
\begin{equation}\label{gequivU}
\theta_g : U \stackrel{\cong}{\longrightarrow} g^\ast U
\end{equation}
for each $g \in G$, which satisfies the cocycle conditions. 

\begin{definition}
The {\em twisting of a $G$-equivariant structure} $\{\theta_g\}$ by a character $\chi:G \to U(1)$ is a $G$-equivariant structure $\{\theta_g^\chi\}$ on $U$ defined by
\begin{equation}\label{def:twist}
\theta_g^\chi = \chi(g) \cdot \theta_g.
\end{equation}
Here, $\chi(g) \cdot (-)$ is defined as a  fiberwise complex multiplication on a complex vector bundle. 
\end{definition}
This defines a character group action on the set of $G$-equivariant structures on a vector bundle on $L$.
Note that for an abelian group $G$,  the set of  $G$-equivariant structures on a complex trivial line bundle are in one-to-one correspondence with the dual group $G^{\ast}=\hat{G}=\Hom (G, U(1))$.


 

 Let $L_0$ and $L_1$ be two $G$-invariant Lagrangian submanifolds and choose $G$-equivariant flat complex vector bundles $(U_0,\theta^0)$ and $(U_1,\theta^1)$ on $L_0$ and $L_1$, respectively. 
Here we allow $U_i$ to have rank greater than $1$ since an irreducible representation of $G$ may have a dimension greater than $1$. In particular, we expect that bundles of higher ranks might play a non-trivial role for non-abelian $G$.

 We define the Floer cochain complex for the pair $(L_0, U_0)$ and $(L_1,U_1)$ similarly as above, but using orientation spaces.
\begin{equation}\label{Gbundleori}
CF( (L_0,U_0), (L_1,U_1)) := \bigoplus_{[w,l_p]} |\Theta_{[w,l_p]}^-|_\C \otimes_\C \Hom (U_0|_p, U_1|_p)
\end{equation}
The Floer differential for \eqref{Gbundleori} is defined in a standard way as explained above.

\begin{definition}\label{GactionGbundleori}
A $G$-action on \eqref{Gbundleori} is defined as follows: for $g \in G$, a linear map
$$g : \Hom (U_0|_p, U_1|_p) \to \Hom (U_0|_{g \cdot p}, U_1|_{g \cdot p})$$
is given by $\phi \mapsto \theta^1_{g} (p) \circ \phi \circ \theta^0_{g^{-1}} (g \cdot p)$. The $G$-action on the first factor of \eqref{Gbundleori} is exactly the same one as in Section \ref{GactFloerNov}.
\end{definition}

Now, we consider $G$-invariant connections on a $G$-equivariant complex vector bundle $(U,\theta)$ on $L$:
\begin{definition}\label{Gcompconn}
A connection $\nabla$ of a $G$-equivariant complex vector bundle $(U,\theta)$ on $L$ is said to be $G$-invariant, if
the pull-back connection $g^\ast \nabla$ is isomorphic to $\nabla$ via $\theta_g$. More precisely,
\begin{equation}\label{nablaGinv}
(\theta_{g})_\ast \left(E_{\nabla} \right)_{(x,v)} =\left( E_\nabla\right)_{(g \cdot x, g \cdot v )}.
\end{equation}
where $E_{\nabla}$ is the horizontal distribution associated to $\nabla$.
\end{definition}

We will only consider $G$-invariant connections from now on.
Now, consider the pair $(L_i, U_i)$ together with the choice of a $G$-invariant connection $\nabla_i$ on $U_i$ for $i=0,1$.

\begin{prop} 
The new Floer differential involving equivariant flat bundle data is $G$-equivariant. 
\end{prop}
\begin{proof}
Let $u$ be a holomorphic strip bounding $L_0$ and $L_1$ and let $p$ and $q$ be its end points. We denote by $\gamma_i$ the boundary component of $u$ lying in $L_i$. (See Figure \ref{equivBdiff}.)
\begin{figure}[h]
\begin{center}
\includegraphics[height=2in]{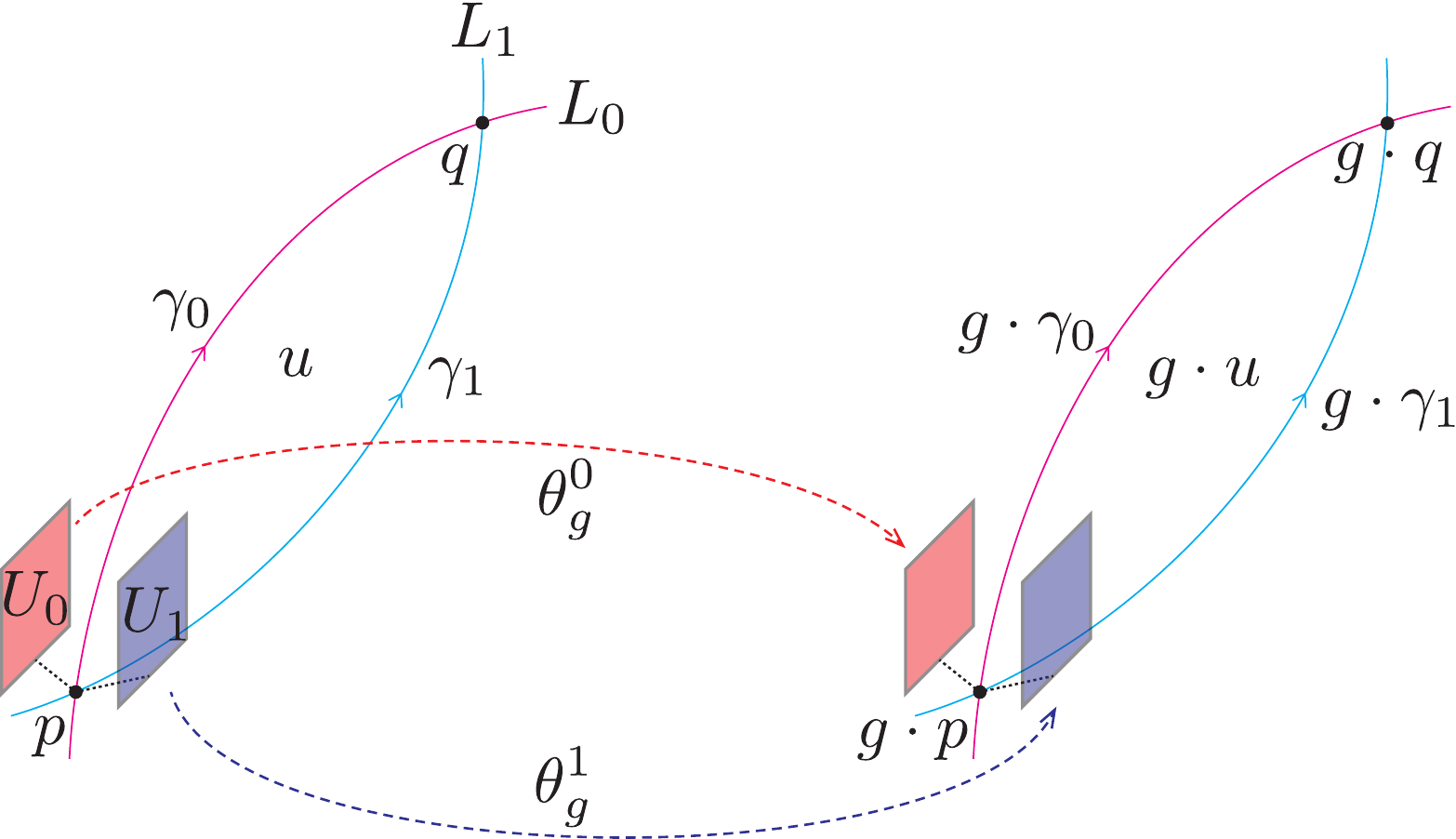}
\caption{Group actions on line bundles and the Floer differential}\label{equivBdiff}
\end{center}
\end{figure}

Take an element $\phi_p$ of $\Hom (U_0|_p, U_1|_p)$ and $g$ of $G$. 
\begin{enumerate}
\item[(i)]
We first compute $\delta \circ g$ : $g $ sends $\phi_p$ to $\theta^1_g (p) \circ \phi_p \circ \theta^0_{g^{-1}} (g \cdot p) \in \Hom (U_0|_{g \cdot p}, U_1|_{g \cdot p})$, and then, the strip $g \cdot u$ sends $\theta^1_g (p) \circ \phi_p \circ \theta^0_{g^{-1}} (g \cdot p)$ to 
\begin{equation}\label{deltacircg}
P^1_{g \cdot \gamma_1} \circ \left( \theta^1_g (p) \circ \phi_p \circ \theta^0_{g^{-1}} (g \cdot p) \right) \circ \left( P^0_{g \cdot \gamma_0} \right)^{-1}
\end{equation}
where $P_{\gamma}^i$ means the parallel transport along $\gamma$ by $\nabla_i$ for $i=0,1$.
\item[(ii)]
Secondly, let us compute $g \circ \delta$ : $u$ contributes to $\delta$ by sending $\phi_p$ to $P^1_{\gamma_1} \circ \phi_p \circ \left( P^0_{\gamma_0} \right)^{-1}$. Now, applying the $g$-action, we get
\begin{equation}\label{gcircdelta}
\theta_g^1 (q) \circ \left(P^1_{\gamma_1} \circ \phi_p \circ \left( P^0_{\gamma_0} \right)^{-1} \right) \circ \theta_{g^{-1}}^0 (g \cdot q).
\end{equation}
\end{enumerate}

We have to show that \eqref{deltacircg} and \eqref{gcircdelta} represent the same element in $\Hom (U_0|_{g \cdot q}, U_1|_{g \cdot q})$. This follows from the $G$-equivariance of $\nabla_i$,  which implies that the $G$-action on $U_i$ commutes with the parallel transport for $\nabla_i$. i.e. 
$$  P^1_{g \cdot \gamma_1} \circ  \theta^1_g (p) = \theta_g^1 (q) \circ P^1_{\gamma_1} \quad \mbox{and} \quad \theta^0_{g^{-1}} (g \cdot p) \circ \left( P^0_{g \cdot \gamma_0} \right)^{-1}= \left( P^0_{\gamma_0} \right)^{-1}  \circ \theta_{g^{-1}}^0 (g \cdot q).$$
\end{proof}

Let $\lambda$ be the action of $\CC^\ast$ on equivariant vector bundles defined by the fiberwise complex multiplication. Then, $\lambda$-action preserves connections in the sense that
$ \lambda^\ast \nabla = \nabla.$
Note that the twisting by the character group is essentially given by the complex multiplication  \eqref{def:twist}. Therefore, $G$-invariant connections still remain $G$-invariant after twisting equivariant structures on a bundle. 

\begin{remark}
For an abelian $G$, one may use a homomorphism $h^{orb} : \pi_1^{orb} ([L/G]) \to U(1)$ to classify equivariant flat line bundles on $L$ and check the above discussion more rigorously. See the related explanation after Conjecture \ref{conj:fiberpres}.
\end{remark}

\subsection{Equivariant Fukaya categories}
We summarize the construction of equivariant Fukaya categories in this section, and also  clarify once again the role of the group action and its dual group action.
We are considering the case that  a finite group $G$ effectively acts on a closed oriented  symplectic manifold $M$.
For each spin profile $s \in H^2(G, \Z/2)$, the equivariant Fukaya category $G\textrm{-}\mathcal{F}uk^s (M)$ is defined as follows.
An object of $G\textrm{-}\mathcal{F}uk^s (M)$ is given by a $G$-invariant Lagrangian submanifold $L$ whose spin profile equals $s$, together
with a $G$-equivariant flat unitary bundle $U$. For each pair  $(L_0, U_0), (L_1, U_1) \in Ob(G\textrm{-}\mathcal{F}uk^s (M))$,
we choose a reference path $l_0$ (one in each homotopy class) and we need to assume that energy zero subgroup $G_\alpha$ equals $G$ always. Then, we define the morphism to be the (completed) direct sum over $l_0$ of the Floer complex of the pair $(L_0, U_0), (L_1, U_1) $ with respect to $l_0$.
There should be $G$-equivariant $\AI$-category operations by extending the construction in this paper to the one of Fukaya \cite{Fuk}.

\begin{remark}
Note that we possibly have several equivariant Fukaya categories corresponding to spin profiles in $H^2(G,\Z/2))$, and
each $G$-invariant (spin) Lagrangian submanifold $L$ can belong to only one of them corresponding to its spin profile. 
\end{remark}
\begin{remark}
We may also consider $\alpha$-twisted $G$-equivariant vector bundles for $\alpha \in H^2(G,U(1))$. These bundles
are not $G$-equivariant, but their failures of $G$-cocycle conditions are given by $\alpha$. We remark that
there is a corresponding notion for sheaves, too (see for example \cite{E}). In fact,  from the expression \eqref{Gbundleori},
 effects of $\alpha \in H^2(G,U(1))$ and a spin profile condition in $H^2(G,\Z/2)$ on Lagrangian Floer theory
can be combined. Hence, it seems that we can enlarge the equivariant or the orbifolded fukaya category by
including such objects. We leave it for future investigation.
\end{remark}

As discussed, $G$ acts on morphisms of the equivariant Fukaya category $G$-$\mathcal{F}uk^s (M)$.
By taking $G$-invariant part of morphisms of  $G$-$\mathcal{F}uk^s (M)$, we obtain the orbifolded Fukaya category $\mathcal{F}uk_{G}^s(M)$, which still has an induced $\AI$-category structure. In fact, for the definition of orbifolded  Fukaya
category, we do not need assumptions that $G = G_\alpha$. 

Now the character group $\hat{G}$ acts on objects of $G$-$\mathcal{F}uk^s (M)$ (hence also on morphisms), by twisting equivariant structures of $G$-bundles on Lagrangians. Therefore, $\hat{G}$ acts both on the $G$-equivariant category and on the orbifolded Fukaya category.
The original Fukaya category of $M$ does not contain the data of $G$-equivariant structure of unitary bundles.
Hence, one needs to take the $\hat{G}$-invariant part of equivariant fukaya category $G$-$\mathcal{F}uk^s (M)$ to obtain
information about the Fukaya category of $M$ itself.


\section{Fukaya-Seidel category of $G$-Lefschetz fibrations}\label{sec:FS}

There is a well-known Fukaya category associated to a Lefschetz fibration $\pi:E \to S$, called the Fukaya-Seidel category of $\pi$ \cite{Se3}, \cite{Se4}.
If the Lefschetz fibration $\pi$ is $G$-invariant, i.e. if there is a $G$-action on $E$ such that
the fibration map $\pi$ is $G$-invariant, then we can define the $G$-equivariant Fukaya-Seidel category of $\pi$ and the orbifolded Fukaya-Seidel category of $\bar{\pi} : [E/G] \to S$. These will depend on the choice of 
a group cohomology class $s$ in $H^2(G,\Z/2)$.
The construction is more or less an easy modification of Seidel's construction. 

\subsection{Equivariant Fukaya-Seidel categories}
First, we review basic ingredients of Fukaya-Seidel categories. 
(See \cite{Se} for details.)
Let $\pi : E \to S$ be an exact Morse fibration where $E$ admits an exact symplectic form $\omega$ and $S$ is isomorphic to $D^2$.
We assume that $c_1(E)=0$ in order to discuss $\Z$-grading, and assume $\dim_\R(E) \geq 4$ to simplify the exposition.
Suppose there exists a $\omega$-preserving $G$-action on $E$ and $\pi$ is invariant under $G$. Then,
the symplectic connection defined by $\omega$ is $G$-invariant, and hence a parallel transport $\rho_c : E_z \to E_w$ along a curve
 $c : [a,b] \to S$  is a $G$-equivariant diffeomorphism.

Let $z_0$ be a regular value of $\pi$ and $M$ the fiber of $\pi$ over $z_0$. If $\{z_1, \cdots, z_n \} \in S$ is the set of critical values of $\pi$, we choose disjoint paths $c_i$ (except at $z_0$) from $z_0$ to $z_i$ for each $1 \leq i \leq n$. $(c_1, \cdots, c_n)$ are arranged by their cyclic order at $z_0$. See Figure \ref{distbasis} where we set $z_0=-i$.
\begin{figure}[h]
\begin{center}
\includegraphics[height=2.3in]{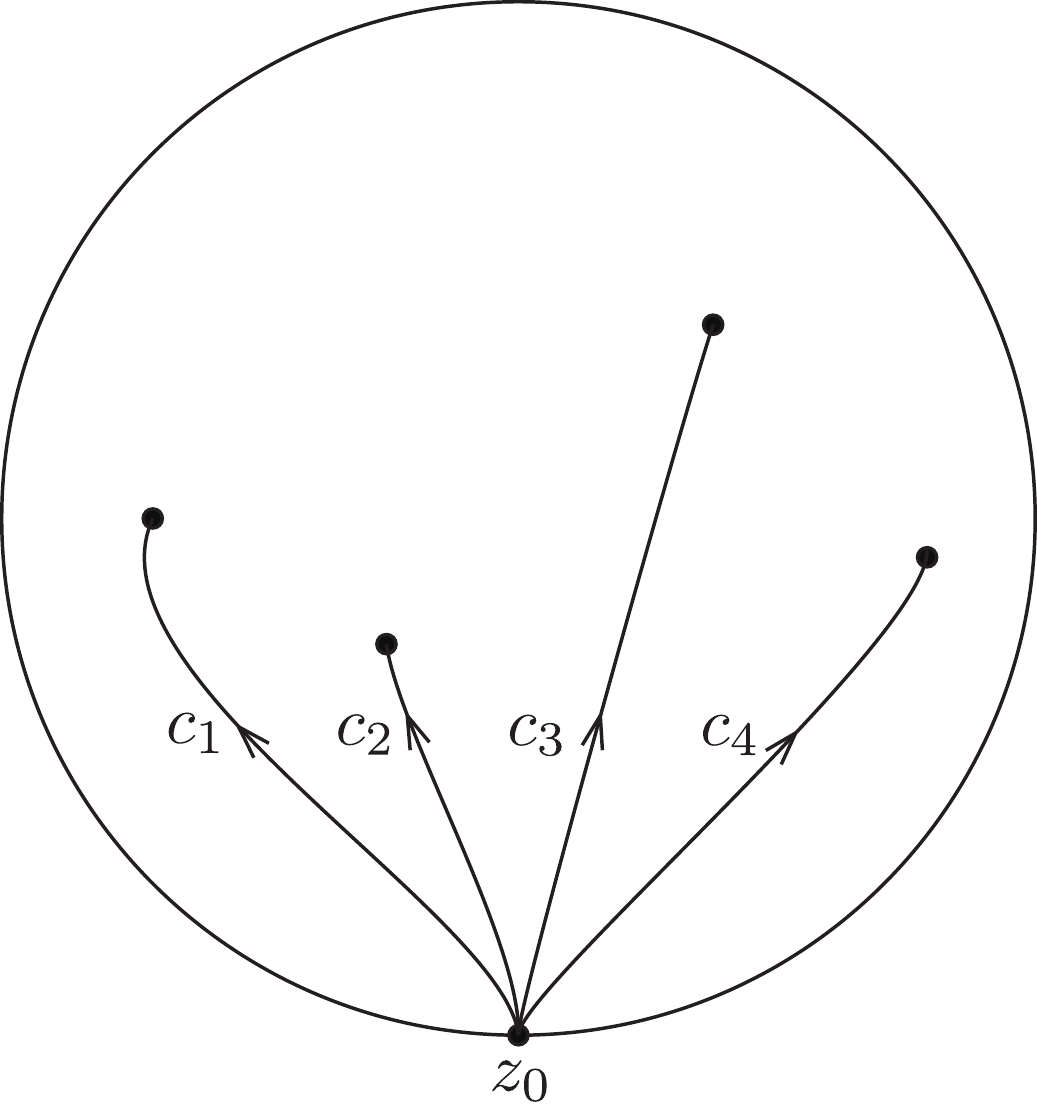}
\caption{Configuration of vanishing paths}\label{distbasis}
\end{center}
\end{figure}
There may be several different critical points over single critical value $z_i$. Denote by $a(i)$
the number of critical points over $z_i$. Then, the parallel transport along $c_i$ defines a set of vanishing cycles $\sqcup_{j=1}^{a(i)} V_{ij}$. Since $G$ acts on these $a(i)$ critical points, $G$ also acts on the set of vanishing cycles $\sqcup_{j=1}^{a(i)} V_{ij}$.
We do not need to give an ordering to the family of vanishing
cycles for $c_i$, because they are disjoint from each other, and hence Floer cohomology groups and Dehn twists between them are all trivial.

As usual, we also pay attention to the Lefschetz thimbles which will be denoted 
by $\Delta_{ij}$'s. More precisely, $\Delta_{ij}$ is the union of images of the vanishing cycle $V_{ij}$ under the parallel transport over $c_i$ so that $\partial \Delta_{ij} = V_{ij}$.
Since the symplectic connection as well as $\pi$ are $G$-invariant, 
the group $G$ acts on the union of Lefschetz thimbles over each $c_i$, and we immediately obtain the following lemma.
\begin{lemma}
For each $i$ and $g \in G$, we have $g(\Delta_{ij}) = \Delta_{ik}$ for some $k \in \{1,\cdots, a(i)\}$. Namely, an action of each $g \in G$ permutes thimbles over the given path $c_i$ (while some of thimbles can be preserved by $g$).
\end{lemma}
Gradings and spin structures on vanishing cycles are defined as follows.
As explained right after Definition \ref{def:seqvbrane}, we can equip $V_{ij}$'s with
$G$-invariant gradings. If $\dim_\R(M) \geq 4$, then the vanishing cycles are spheres of $\dim \geq 2$, which
have canonical spin structures. In case $\dim_\R(M) =2$, vanishing cycles are circles, and we choose a spin structures of each vanishing cycle to be the one obtained from the restriction of the trivialization of the tangent bundle of the corresponding Lefschetz thimble.


We remark that an object consists of a single connected component if $G$-action preserves $V_{ij}$, and has several disjoint components otherwise. For example, if there is only one critical point over the critical value $z_i$, then the vanishing cycle $V_{i1}$ is preserved by the $G$-action and is an object of the category as this critical point should be fixed by the whole group $G$.

Now, we define an equivariant (or an orbifold) version of a directed $\AI$-category $\mathcal{FS}$, so called the Fukaya-Seidel category. For each $s$ in $H^2 (G;\ZZ/2)$, we will have a category consisting of $s$-equivariant branes, which are $G$-invariant (union of) vanishing cycles with the spin profile $s$.

\begin{definition}\label{def:FS}
Consider a $G$-invariant Lefschetz fibration $\pi:M \to \C$,
and choice of paths $\{c_i\}$ and vanishing cycles $\{V_{ij}\}$ as before.
An object of both the $s$-equivariant Fukaya-Seidel category $G$-$\mathcal{FS}^s (M,\pi)$
and orbifold Fukaya-Seidel category $\mathcal{FS}^s_G (M, \pi)$ is given by
a $G$-orbit $\cup_g g \cdot V_{ij}$ for each vanishing cycle $V_{ij}$ for some $1 \leq i \leq n, 1\leq j \leq a(i)$ whose spin profile is $s$, together with a $G$-invariant grading
and $G$-equivariant flat complex vector bundle on $\cup_g g \cdot V_{ij}$. We require that this bundle extends to a $G$-equivariant flat complex vector bundle on the corresponding thimble. We denote by $\mathcal{V}_{ij}$ the $G$-orbit of vanishing cycle $\cup_g g \cdot V_{ij}$ together with these additional data.
Also, if $V_{ij}$ and $V_{ij'}$ lie in the same $G$-orbit, then we identify $\cup_g g \cdot V_{ij}$ and $\cup_g g \cdot V_{ij'}$.
\end{definition}
The objects of equivariant and orbifold categories are the same as above, but
their morphisms will be defined differently. We postpone the discussion of  adding $G$-equivariant flat complex vector bundle to the
construction in Subsection \ref{equivbundleFS}, and we first explain the construction without them.

We may suppose that all vanishing cycles (and hence, all objects in $\mathcal{C}$) are transversal to each other by choosing vanishing paths in general positions (Figure \ref{distbasis}). The cyclic order at $z_0$ gives a partial order, on the index $i$
of $V_{ij}$.
\begin{definition}\label{def:FS2}
We define morphisms between two objects $\mathcal{V}_{i_1j_1}$ and $\mathcal{V}_{i_2j_2}$ of the $s$-equivariant category $G$-$\mathcal{FS}^s (M,\pi)$ as follows.
 
For $i_1 \neq i_2$,
\begin{align*}
&hom_{G\textrm{-}\mathcal{FS}^s} (\mathcal{V}_{i_1j_1}, \mathcal{V}_{i_2j_2}) \\
&= \left\{
\begin{array}{ll}
CF^\ast (\mathcal{V}_{i_1j_1}, \mathcal{V}_{i_2j_2}) = \bigoplus_{y \in \mathcal{V}_{i_1j_1} \cap \mathcal{V}_{i_2j_2}}| o(y)|_R &  k<l\\
0 & k>l
\end{array}\right.
\end{align*}
For $i_1 = i_2$, and
two objects  $\mathcal{V}_{i_1j_1}$ and $\mathcal{V}_{i_2j_2}$ are 
in fact the same, then we define
$$hom_{G\textrm{-}\mathcal{FS}^s} (\mathcal{V}_{i_1j_1}, \mathcal{V}_{i_1j_1}) = R \cdot (id_{V_1} \oplus \cdots id_{V_k}).$$
For the remaining case, their morphisms are defined to be zero (since they do not intersect). 
\end{definition}
\begin{remark}
This $\AI$-category is not strictly but partially directed in a sense.
\end{remark}
Note that the morphism space admits a $G$-action. For $\mathcal{V}_{i_1j_1} \neq \mathcal{V}_{i_2j_2}$, the
$G$-action on their morphism is exactly the same as in Section \ref{sec:eqfuex}. If $i_1 = i_2$ and
two objects $\mathcal{V}_{i_1j_1}$ and $\mathcal{V}_{i_2j_2}$ are  the same, then the $G$-action on $R \cdot (id_{V_1} \oplus \cdots id_{V_k})$ is defined as follows.
We first fix orientations of $V_1,\cdots, V_k$, and the construction below depends on this choice.
Since $\mathcal{V}$ is a $G$-orbit, each connected component $V_i$ is mapped to $V_{j}$ for some $j$.
In such a case, we define $g$-action $R \cdot id_{V_i} \to R\cdot id_{V_j}$ to be
$(\pm 1)$, where the sign is positive if the $g$-action preserves the pre-fixed orientations, and negative otherwise.
This defines $G$-action on morphism spaces. 

In order to construct $G$-equivariant  $\AI$-operations on morphisms of  $G$-$\mathcal{FS}^s (M,\pi)$, we may adapt the
construction in the previous section.
Namely,  we consider a  $A_\infty$-category $\mathcal{D}$,
whose object is $(\mathcal{V}_{ij}, g)$ for each $i,j$ and $ g\in G$. Then the set of object of $\mathcal{D}$
admits a free $G$-action, and we can make the perturbation data $G$-equivariant.
As Seidel mentioned in \cite[Remark 6.1]{Se3}, we do not have to care about the chain complexes underlying $HF (\mathcal{V}_{ij}, \mathcal{V}_{ij})$ by the (partial) directedness, but only the above case ($k_0 < \cdots < k_d$) is needed to be considered.
Therefore, the homological perturbation lemma of Seidel (Lemma \ref{hompert}) provides $G$-equivariant $A_\infty$-operations on morphisms of $G$-$\mathcal{FS}^s (M,\pi)$. 

\begin{definition}
For the orbifold Fukaya-Seidel category,
morphisms of $\mathcal{FS}^s_G (M, \pi)$ are defined to be the $G$-invariant
part of the corresponding morphisms of  $G$-$\mathcal{FS}^s (M,\pi)$, which has an induced $A_\infty$-operations.
\end{definition}

\subsection{Equivariant flat unitary bundles for $G$-$\mathcal{FS}^s (M,\pi) $}\label{equivbundleFS}
Even though vanishing cycles $\mathcal{V}$ give rise to objects of directed $\AI$-category, one has to think in terms of respective vanishing
thimbles denoted as $\mathcal{T}$,
 which are indeed Lagrangian submanifolds in $M$. We assume here that $G$ is abelian in order to simplify the exposition. General cases can be handled as in Section \ref{sec: orbibundle}. Note that any irreducible representation of $G$ is 1-dimensional and hence, it suffices to consider line bundles only.

Consider $\mathcal{V}:=\cup_g g \cdot V_{ij}$ for each $i,j$ and the $G$-equivariant line bundle $U$ on $\mathcal{V}$ which extends to the one on the corresponding thimble $\mathcal{T}$.
Let us denote it by $U^{\mathcal{T}} \to \mathcal{T}$, which is necessarily trivial. 
Thus, $U^{\mathcal{T}} \cong \mathcal{T} \times \C$ and the $G$-equivariant structure on $U^{\mathcal{T}}$ comes up with $\chi \in Hom (G, \C^\ast)$.

Suppose that $\mathcal{V}_i$ is a vanishing cycles with a $G$-equivariant line bundle $ U^{\mathcal{T}}_i$ as above for $i=1,2$. Then, pairs $(\mathcal{V}_i, U^{\mathcal{T}}_i)$ for $i=1,2$ are objects of $G$-$\mathcal{FS}^s (M,\pi)$ and we want to define their morphism space.

If $\mathcal{V}_1 \neq \mathcal{V}_2$
We define their morphism space to be as in section \ref{sec:FS}  using $o(y)$ for each intersection points $y$ but with $\C$ coefficients. The group action on the morphism space is twisted by $\chi_1$ and $\chi_2$. i.e. $g : o (y) \to o(g \cdot y)$ is defined as the original action multiplied by $\chi_1 (g)^{-1} \chi_2 (g)$.
This is a canonical action on  $o(y) \otimes \Hom (U^{\mathcal{T}}_0|_y,U^{\mathcal{T}}_1|_y)$ consistent with \eqref{Gbundleori}. However, since $U^{\mathcal{T}}_i|_y$ is canonically isomorphic to $\CC$, we may just use $o(y)$ as above.

Let us consider the case $\mathcal{V}_1 = \mathcal{V}_2=:\mathcal{V}$. Denote the connected components of $\mathcal{V}$ as $V_1,\cdots V_k$ associated to critical points $x_1,\cdots, x_k$. Recall that $\Hom (\mathcal{V}, \mathcal{V})$ ( with $R = \C$) is given by
\begin{equation}\label{mathcalVend}
\CC \cdot (id_{V_1} \oplus \cdots \oplus id_{V_k}).
\end{equation}
The $G$-action on the above morphism space has been already described in Section 7. 
When they are equipped with $G$-equivariant line bundles,  $\Hom ( (\mathcal{V},U^{\mathcal{T}}_1), (\mathcal{V},U^{\mathcal{T}}_2))$ is
defined to be the same as \eqref{mathcalVend}, whose  $g$-action is given by the previous $g$-action defined in section \ref{sec:FS}  with additional multiplication of $\chi_1 (g)^{-1} \chi_2 (g)$.

Note that the $G$-invariant part $\Hom ( (\mathcal{V},U^{\mathcal{T}}_1), (\mathcal{V},U^{\mathcal{T}}_2))^G$
is non-trivial if $G$-action on the set of connected components of $\mathcal{V}$  is free. 

\section{Some examples of group actions and Mirror Symmetry}\label{MirrorSymm}
We discuss a few examples of well-known homological mirror symmetry, but with an additional finite group action. The
homological mirror symmetry (conjectured by Kontsevich) asserts that the derived Fukaya category of a
symplectic manifold (or the derived Fukaya-Seidel category of a  LG model) is equivalent to the derived category of coherent sheaves of the mirror complex manifold (or the matrix factorization category of the mirror LG model).
Strominger-Yau-Zaslow approach explains such a phenomenon as a correspondence between dual torus fibrations. We will be rather sketchy, as our motivation is to consider its relationship with
group actions.

Consider a Lagrangian torus fibration $\pi:X \to B$ where $\pi^{-1}(b)$ is regular torus fiber for $b \in \mathring{B}$, and assume that $\mathring{B}$ is
simply connected. Suppose that the mirror is given as
a Landau-Ginzburg (LG) model $W:Y \to \C$. Here,  $Y \to \mathring{B}$ is the dual torus fibration 
$$Y = \{(X_b, \nabla_b) : \nabla_b \in \Hom (X_b, U(1)) , b \in \mathring{B}\}$$
where $\Hom (H_1( X_b), U(1))$ is considered to be the space of flat connections on the trivial line bundle on $X_b$. 
Suppose that the mirror potential  $W : Y \to \CC$ is defined as a Lagrangian Floer potential
(see for example \cite{CO}, \cite{FOOOT1}, \cite{Au})
$$W(L_u, \nabla_u) = \sum_{\beta, \mu (\beta)=2} n_\beta (L_u) \exp \left( - \int_\beta \omega \right) {\rm hol}_{\nabla_u} (\partial \beta).$$

\begin{definition}
We say that an action of a finite group $G$ on $X$ is compatible with the fibration $\pi$ if
the group action sends torus fibers to torus fibers. 
\end{definition}
 
 We will assume that $G$ acts on $X \to B$ in a compatible way from now on. Thus 
  each $g \in G$ sends
a fiber $X_b$ into $X_{b'}$ for some $b'$ for each $b \in B$.
\begin{definition}
We define the $G$-action on $Y$ as
$$g \cdot (X_b, \nabla_b) = (X_{g \cdot b}, \, g_\ast \nabla_b ).$$
(Here, $g_\ast \nabla_b := \left( g^{-1} \right)^{\ast} \nabla_b$.) 
\end{definition}
Note that this $G$-action could be trivial even if the $G$-action on $X$ is effective
as we will see in the example of  $\CP^1$.
For a $G$-action on $X$ compatible with $\pi$, we have
an induced homomorphism $\rho: G \to Aut(B)$. This gives rise to the following exact sequence
$$ 1 \to Ker(\rho) \to G \to  G/Ker(\rho) \to 1.$$
We  may denote  $ G_F = Ker(\rho)$ and $G_B = G/Ker(\rho)$ . 
The example of $\CP^1$ below is the case when $G = G_F$. 
But if $G = G_B$, it is clear that $G$-action on $Y$ is also effective if
that on $X$ is. We will see such an example in the Section \ref{sec:CP2}, where the $G$-action  on $\CP^2$
rotates the moment triangle $B$.

It is not hard to show that $W$ becomes $G$-invariant for the induced $G$-action on $Y$.
\begin{lemma}\label{chitriv}
$W : Y \to \C$ is $G$-invariant.
\end{lemma}
\begin{proof}
It follows from the invariance of the symplectic form $\omega$ under the $G$-action:
\begin{eqnarray*}
W(X_{g \cdot b}, \, g_\ast \nabla_b) &=& \sum_{\substack{\beta' \in \pi_2 (X,X_{g \cdot b}) \\ \mu (\beta')=2}} n_{\beta'} (X_{g \cdot b}) \exp \left( - \int_{\beta'} \omega \right) {\rm hol}_{g_\ast \nabla_b} (\partial \beta') \\
&=& \sum_{\substack{ \beta \in \pi_2 (X,X_b) \\ \mu (\beta)=2}} n_{g \cdot \beta} (X_{g \cdot b}) \exp \left( - \int_{g \cdot \beta} \omega \right) {\rm hol}_{g_\ast \nabla_b} (\partial (g \cdot  \beta)) \\
&=&  \sum_{\substack{ \beta \in \pi_2 (X,X_b) \\ \mu (\beta)=2}} n_\beta (X_b) \exp \left( - \int_{\beta} \omega \right) {\rm hol}_{\nabla_b} (\partial \beta) = W(X_b, \nabla_b)
\end{eqnarray*}
\end{proof}

We denote by $s$ the spin profile of torus fibers $X_b$ for $b \in \mathring{B}$.
We can formulate the homological mirror symmetry conjecture in this setting as follows:
\begin{conjecture}
There are equivalences of derived categories,
\begin{eqnarray}\label{GHMSconj}
D^b \mathcal{F}uk_G^s (X) &\cong& D^b MF_G (W) \\
  D^b\mathcal{FS}_G^s(Y,W) &\cong&D^b Coh_G{X}
\end{eqnarray}
where $D^b \mathcal{F}uk_G^s (X)$ is the derived $s$-orbifolded Fukaya category of $X$ defined in this paper,
$D^bCoh_G(X)$  is the derived category of $G$-equivariant coherent sheaves on $X$
and $MF_G(W)$ is the $G$-equivariant matrix factorization category of $W$.
The character group $\hat{G}$ acts on both sides of \eqref{GHMSconj} by twisting
equivariant structures in a compatible way.
\end{conjecture}
In general, for a mirror pair $X, Y$, one may conjecture the equivalence of two categories $D^b \mathcal{F}uk_G^s (X)$
and $D^bCoh_G(Y)$.
Here, note that the right hand sides do not involve spin profiles. This is because sheaves on the right hand side
are expected to be obtained as certain family version of Floer homologies, and hence $G$-action is well-define for them
(without spin profiles).
Note also that the right hand side is well-defined even when $G$ acts trivially.

\subsection{Fiberwise  $G_F$-actions and toric examples}
Suppose that an abelian group $G=G_F$ acts freely on regular torus fibers. In particular, the character group $G^\ast$ is the
dual group of $G$.
Its dual torus fibration can be understood as follows. Consider
\begin{equation}\label{WTYfiber0}
 1 \to \pi_1 (X_b) \to \pi_1^{orb} ([X_b/G]) \to G \to 1
\end{equation}
and take $\Hom(-, U(1))$  to obtain the following exact sequence:
\begin{equation}\label{WTYfiber}
0 \to G^\ast = \Hom (G, U(1)) \to \WT{Y}_b:=\Hom (\pi_1^{orb} ([X_b /G]), U(1)) \to Y_b=\Hom (\pi_1 (X_b) , U(1)).
\end{equation}
If $X_b/G_F$ is still a (Lagrangian) torus, then the last map $\WT{Y_b} \to Y_b$ is surjective. This follows from the fact that the first and the second term in the exact sequence \eqref{WTYfiber0} are isomorphic to $\ZZ^n$ and $U(1)$ is divisible. 

Let us assume that $X/G_F \to B$ is again a Lagrangian torus fibration from now on. This happens for example when $X$ is a toric manifold with the action of a finite subgroup $G=G_F$ of the torus
$U(1)^n \subset (\C^*)^n$.

Since $\pi_1^{orb}([X_b/G]) = \pi_1(X_b/G)$,  the middle term $\WT{Y}_b$ gives the dual torus fiber of $X_b /G$, 
giving rise to the dual torus fibration $\WT{Y} \to \mathring{B}$.
From \eqref{WTYfiber}, it is easy to see that the dual group $G^\ast$ acts on $\WT{Y}$ and the quotient is isomorphic to the mirror manifold $Y$ of $X$. Let $p : \WT{Y} \to Y$ be the quotient map via $G^\ast$. Indeed $G^\ast$ acts on $\WT{Y}$ freely and $p:\WT{Y} \to Y$ is a covering whose deck transformation group is precisely $G^\ast$ since the first map in \eqref{WTYfiber} is injective. 

Let $\WT{W} : \WT{Y} \to \CC$ be given by the composition $W \circ p$. $\WT{W}$ is $G^\ast$-invariant by definition. We define $\WT{W}$ by the composition because smooth holomorphic discs in $[X/G]$ are
in fact obtained from smooth holomorphic discs in $X$. (See \cite{CP} for more details: we do not consider bulk-deformations by twisted sectors, here.)

The situation is summarized in the diagram below:
\begin{equation*}
\xymatrix{X \ar[d]_{G} \ar@{<.>}[rr]^{mirror} && Y \ar[r]^W & \CC \\
[X/G] \ar@{<.>}[rr]_{\quad mirror} && \WT{Y}\ar[u]^{G^\ast} \ar[ur]_{\WT{W}}&.
}
\end{equation*}
Namely, in this case, taking $G$-quotient on one side corresponds to taking a $G^\ast$-covering on the other side. We remark that such a phenomenon has been already observed for the pair of pants case by Abouzaid, Auroux, Efimov, Katzarkov and Orlov \cite{Ab}.

Now, it is natural to expect that $(\WT{Y}, \WT{W})$ is the LG mirror of $[X/G]$ in the sense of SYZ. Moreover, the $A$-model category for $X$ and the one for its quotient are expected to have a close relationship.
\begin{conjecture}\label{conj:fiberpres}
In the above setting, we have equivalences of derived categories
$$D^b\mathcal{F}uk_G^s(X) \cong D^bMF (\WT{W})$$
and
$$D^b\mathcal{F}uk(X) \cong \big( D^b\mathcal{F}uk_G^s(X) \big)^{G^\ast}.$$
\end{conjecture}
Here, $MF_{G^\ast} (\WT{W})$ is 
the $G^\ast$-equivariant matrix factorization category. As the $G^\ast$-action on $\WT{Y}$ is free, $MF_{G^\ast} (\WT{W})$ is equivalent to $MF(W)$ where the equivalence from $MF(W)$ to $MF_{G^\ast} (\WT{W})$ is the pull-back functor by the projection $\WT{Y} \to Y$. (This is a special case of Proposition 2.2 in \cite{PV}.)
Hence the second part of the conjecture follows from the first part of the conjecture.

We can relate the $G^\ast$-action on $\WT{Y}$ with the $G^\ast$-action on $\mathcal{F}uk^s_G (X)$ in a geometric way. In view of SYZ, a point in $\WT{Y}_b$ corresponds to a torus fiber $[X_b /G]$ together with a flat line bundle over it, or equivalently a $G$-equivariant flat line bundle $(U,\theta, \nabla)$ over $X_b$. Since $G$ acts freely on $X_b$, a $G$-equivariant structure on the line bundle $U$ is unique up to isomorphism. However, if we apply the character group action of $\chi \in G^\ast$ to $(U, \theta, \nabla)$, the flat structure, or holonomy in an orbifold sense,  of $[U/G]$ changes as follows.

Consider a generalized curve $\gamma \in \pi_1^{orb} ([X_b/G])$ with $\gamma(0) = x_0$ to $\gamma(1) = g \cdot x_0$ where $x_0$ is the base point for $\pi_1 (X_b)$. After projecting down to $[X_b/G]$, $\gamma$ becomes a genuine loop $\bar{\gamma}$. Then, the holonomy along the loop $\bar{\gamma}$ in $[X_b /G]$ can be measured by the difference between the parallel transport along $\gamma$ and $\theta_g$, both of which are linear maps from $U_{x_0}$ and $U_{g \cdot x_0}$ If $(U,\theta,\nabla)$ is twisted by an element $\chi \in G^\ast$, then the identification $\theta_g$ of $U_{x_0}$ and $U_{g \cdot x_0}$ is also twisted by $\chi(g)$. i.e. $U_{x_0}$ and $U_{g \cdot x_0}$ are now identified by $\chi(g) \theta_g$. As a result, the holonomy around $\bar{\gamma}$ in $[X_b /G]$ is changed by $\chi(g)$. This explains the action of $G^\ast$ on $\WT{Y}_b$ shown in \eqref{WTYfiber}.

\subsection{Example of $\CP^1$}
We provide an example which illustrate the phenomenon discussed in the previous subsection. Equip $X:=\CP^1$ with the $\ZZ/3$-action given by $[z_0 : z_1] \to [\rho z_0 : z_1]$ where $\rho=e^{2 \pi i /3}$. It  preserves all torus fibers, and hence $G= G_F$ in this example. Associated to this action, we can construct a three-fold cover $\WT{Y}$ of the mirror $Y:= \CC^\ast$ of $X$. Identifying $\WT{Y}$ with $\CC^\ast$, the covering map is explicitly given by
$$w \in \WT{Y} \mapsto z= w^3 \in Y.$$
The $G^\ast(\cong \ZZ/3)$-action on $\WT{Y}$ \eqref{WTYfiber} sends $w$ to $\rho w$. 
\begin{remark}
The variable $w$ indeed comes from orbidiscs bounding torus fibers in $[X/G]$. In Figure \ref{CP1modZ3}, $w = T^{\omega (\beta)} hol_{\nabla} \partial \beta$. See \cite{CP} or \cite{CHL} for more details.
\end{remark}
\begin{figure}[h]
\begin{center}
\includegraphics[height=3in]{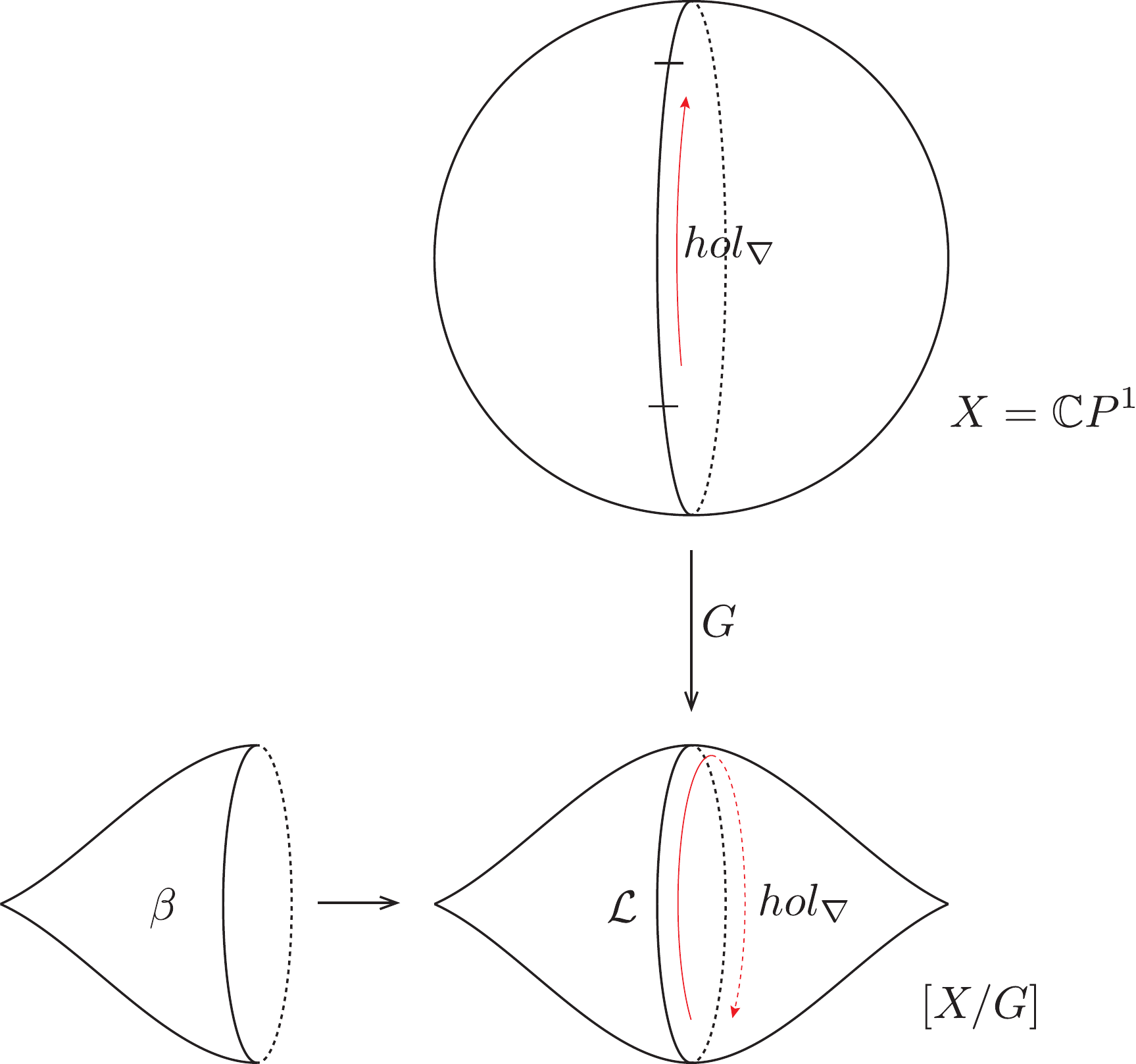}
\caption{}\label{CP1modZ3}
\end{center}
\end{figure}

Then, in terms of a coordinate $w$ on $\WT{Y}$, the pulled-back potential $\WT{W}$ can be written as
$$\WT{W} = w^3 + \frac{1}{w^3}.$$
(For simplicity, we set the K\"ahler parameter $q$ to be $1$.)
Note that $\WT{W}$ is invariant under the $G^\ast$-action. Consider the Lagrangian brane $\mathcal{L}$ supported on the balanced torus fiber with the trivial holonomy in $[X/G]$. In \cite{CHL}, it is shown that the SYZ transformation of $\mathcal{L}$ is the following matrix factorization of $\WT{W}$.
\begin{equation}\label{mirMFWTW}
(1-w) \left( \frac{1}{w^3} + \frac{1}{w^2} + \frac{1}{w}-1-w -w^2 \right) = w^3 + \frac{1}{w^3} -2
\end{equation}
Notice that $2$ is a critical value of $\WT{W}$. Now, $\rho \in G^\ast$ sends $\mathcal{L}$ to the Lagrangian brane supported on the same torus fiber but with the holonomy $\rho$. In the mirror, it sends \eqref{mirMFWTW} to 
$$(1-\rho w ) \left( \frac{1}{w^3} + \frac{1}{\rho^2 w^2} + \frac{1}{\rho w}-1-\rho w -\rho^2 w^2 \right) = w^3 + \frac{1}{w^3} -2,$$
where we used $\rho^3 =1$.

There is aa easy example of a $G^\ast$-invariant object in $MF(\WT{W})$, which can be written down only in terms of $w^3$:
\begin{equation}\label{GdualinvMF}
(1- w^3) \left( \frac{1}{w^3} -1 \right) = w^3 + \frac{1}{w^3} -2.
\end{equation}
Using $z = w^3$, \eqref{GdualinvMF} is expressed as
$$(1- z) \left(\frac{1}{z} -1 \right) = z + \frac{1}{z} -2 (=W- 2),$$
which was shown to be the mirror matrix factorization of the equator in $\CP^1$ in \cite{CL}.

\section{$\CP^2$ with a $\ZZ /3$-action}\label{sec:CP2}
In this section, we consider the projective plane $\CP^2$ equipped with the $G= \ZZ /3$-action which cyclically permutes three homogeneous coordinate of $[z_0: z_1 : z_2]$ in $\CP^2$. Note that the torus fibration $\CP^2 \to \Delta$ is compatible with  this $\ZZ / 3$-action. 
The Landau-Ginzburg mirror of $\CP^2$ is given by $\{ xyz=1 \} \subset \CC^3$ with the superpotential $W(x,y,z) = x+ y+z$. It is easily check that the mirror $\ZZ /3$-action also permutes three variable $x$, $y$, and $z$, and hence $W$ is invariant under this action. 
In what follows, we identify the mirror of $\CP^2$ with $(\CC^\ast)^2$ and use $W(x,y) = x + y + \frac{1}{xy}$ as the superpotential.

We will compare $D^b_{\ZZ /3} Coh (\CP^2)$ and $\mathcal{FS}_{\ZZ /3} (W)$ in Subsection \ref{BA}, and compare $\ZZ /3$-invariant objects in $\mathcal{F}uk (\CP^2)$ and $D^b_{sing} (W)$ in Subsection \ref{AB}.

Let us first remark on spin profiles and energy zero subgroups. For $Fuk (\CP^2)$, we will consider the Floer homology of a torus fiber with itself, and in such a case a choice of spin profile does not affect the group action.  
Indeed, spin profiles for invariant Lagrangians involved in this example are all zero since $H^2(\Z/3, \Z/2) =0$.
Also the central fiber has
a fixed point of group action.  Thus from Lemma \ref{locgpinGfl}, we have $G_\alpha =G$, and  we can define $G$-actions on Lagrangian Floer homology groups.

\subsection{Comparison of $D^b_{\ZZ /3} Coh (\CP^2)$ and $\mathcal{FS}_{\ZZ /3} (W)$}\label{BA}
The following discussion is based on the computations in \cite{A} except that we additionally consider the group action.

Before going into details, we briefly recall what equivariant sheaves are. Let $X$ be a topological space with an action of a finite group $G$.

\begin{definition}\label{Gequivsh1}
A $G$-equivariant structure on a sheaf $\mathcal{F}$ on $X$ is a choice of an isomorphism $\theta_g :  \mathcal{F} \cong g^\ast \mathcal{F}$ for each $g$ such that these isomorphisms satisfy the cocycle condition
\begin{equation}\label{cocycEquiSh}
\theta_{gh} = h^\ast \theta_g \circ \theta_h.
\end{equation}
\end{definition}
Since $\left( g^\ast \mathcal{F} \right)_x = \mathcal{F}_{g \cdot x}$ for $x \in X$, 
the $G$-equivariant structure on $\mathcal{F}$ gives an isomorphism $\theta_g (x) : \mathcal{F}_x \stackrel{\cong}{\longrightarrow} \mathcal{F}_{g \cdot x}$ for each $g \in G$ which covers the base action.
\begin{equation}\label{sheafact}
\xymatrix{\mathcal{F}_x \ar[rr]^{\theta_g (x)} \ar@{.>}[d] && \mathcal{F}_{g \cdot x} \ar@{.>}[d] \\
x \ar@{|->}[rr]^g && g \cdot x
}\qquad
\xymatrix{ \mathcal{F}_x\ar[rr]^{\theta_{gh} (x)} \ar[dr]_{\theta_h (x)} && \mathcal{F}_{(gh) \cdot x}\\
&\mathcal{F}_{h \cdot x} \ar[ur]_{\theta_g (h \cdot x)}&
}
\end{equation}
On the level of stalks, the cocycle condition \eqref{cocycEquiSh} can be understood more clearly shown in the right side diagram in \eqref{sheafact}.

Denote the category of $G$-equivariant sheaves by $Sh_G (X)$ and define

\begin{equation}\label{CohG}
Coh^G (X) := \{ (\mathcal{F},\theta) \in Sh_G (X) : \mathcal{F} \in Coh (X) \}.
\end{equation}
 
We write $D_G^b Coh (X)$ for the derived category of \eqref{CohG}. The original definition of $D^b_G Coh (X)$ is slightly different from the one given here, but it is known that they are equivalent if $G$ is finite. (See \cite{BL}.)

\noindent{\bf (i)} We first find generators $D^b_{\ZZ /3} Coh (\CC P^2)$. 
Before describing the equivariant version, we recall the following well-known fact.
\begin{prop}\label{beilinsoncot}\cite{B}
Let $\Omega=\Omega^1$ be the cotangent sheaf on $\CP^2$. Then, 
\begin{equation}\label{genCP} 
{\Omega}^2 (2) , \quad {\Omega}^1 (1), \quad {\Omega}^0 =\mathcal{O}
\end{equation}
forms a full strong exceptional collection in $D^b (Coh (\CC P^2) )$.
\end{prop}
\begin{remark}
There is another well-known generators of $D^b Coh (\CP^2)$, $\{\mathcal{O},\mathcal{O}(1),\mathcal{O}(2)\}$. See \cite{B} or \cite[Theorem 2.12]{A}.
\end{remark}
It is also known from \cite{B} that
\begin{equation}\label{RHomOmega0}
R\Hom ({\Omega}^1 (1),\mathcal{O}) \cong R\Hom (  {\Omega}^2 (2),  {\Omega}(1) ) \cong  \CC \left< x\right> \oplus  \CC \left< y\right> \oplus \CC \left< z\right>,
\end{equation}
and
\begin{equation}\label{RHomOmega1}
R\Hom (\Omega^2 (2), \mathcal{O} ) \cong   \CC \left< x\wedge y\right> \oplus  \CC \left< y \wedge z\right> \oplus \CC \left< z \wedge x \right>.
\end{equation}

%
%
%
 
Note that three sheaves in Proposition \ref{beilinsoncot} are all $\ZZ /3$-invariant. We fix an equivariant structure on each of these sheaves so that the induced action on morphism spaces \eqref{RHomOmega0} and \eqref{RHomOmega1} are given by cyclic permutation on $dx, dy, dz$. Let $V_i$ ($i=0,1,2$) be the irreducible representation of $\ZZ /3$ on which the primitive generator of $\ZZ /3$ acts by $\rho^i$. Then, any $\ZZ /3$-equivariant structure on $\Omega(i)$ is isomorphic to $\Omega (i) \otimes V_j$ for some $j$. Moreover, for $i_1, i_2, j_1, j_2 \in \{0,1,2\}$ with $i_1 \geq i_2$,
$$\Hom_G (\Omega^{i_1} (i_1) \otimes V_{j_1},\Omega^{i_2} (i_2) \otimes V_{j_2} )\cong \left(\Hom (\Omega^{i_1} (i_1), \Omega^{i_2} (i_2) ) \otimes V_{j_2 - j_1} \right)^G.$$
(See the proof of \cite[Theorem 2.1]{E}.) If $i_1$ is strictly greater than $i_2$, then $\dim \Hom (\Omega^{i_1} (i_1), \Omega^{i_2} (i_2)) =3$ and $G$ permutes its basis (which are chosen in \eqref{RHomOmega0} and \eqref{RHomOmega1}). Therefore, the rank of the right hand side is $1$ regardless of $j_1$ and $j_2$. However, if $i_1 = i_2=i$, then
\begin{equation}\label{iidelta}
\Hom_G (\Omega^{i} (i) \otimes V_{j_1},\Omega^{i} (i) \otimes V_{j_2}) \cong \ V_{j_2 - j_1}^G \cong \CC^{\delta_{j_1 j_2}}.
\end{equation}

\noindent{\bf (ii)} The mirror $\Z/3$-action fixes three critical points of $W$, 
$$p_0= (1,1),\quad  p_1= (\rho, \rho),\quad p_2= (\rho^2, \rho^2),$$
where $\rho= e^{2 \pi i /3}$. They are indeed the only fixed points of this $\ZZ /3$-action. We have vanishing cycles $V_0$, $V_1$ and $V_2$ in $W^{-1} (0)$ associated to $p_0$, $p_1$, and $p_2$ respectively. Each $V_i$ is a $\ZZ /3$-invariant Lagrangian submanifold of $W^{-1} (0)$.
Since the $\ZZ /3$-action on $W^{-1} (0)$ is free, the action on $V_i$ is topologically equivalent to the one generated by the $2 \pi i /3$-rotation on the circle($\cong V_i$) and in particular, it is orientation preserving. 
We equip $V_i$'s with nontrivial spin structures (which extend to the vanishing thimbles). As mentioned, we do not need to care about spin profiles as $H^2 (\ZZ /3 , \ZZ /2)$ is zero and therefore, the group action lifts to the spin bundles.

%
%
%

By projecting to the $x$-coordinate plane (minus the origin), we see that each pair of vanishing cycles intersects at three points. We denote them as 
$$V_0 \cap V_1 = \{x_0, y_0, z_0 \},$$
$$V_1 \cap V_2 = \{x_1, y_1, z_1 \},$$
$$V_0 \cap V_2 = \{\bar{x}_0, \bar{y}_0, \bar{z}_0 \}.$$
Since the group action is free on $W^{-1} (0)$, it also freely permutes three points in $V_i \cap V_j$ for $i \neq j \in \{0,1,2\}$. From \cite{A},
\begin{lemma}\label{CP2m1}
$m_1$
is identically zero.
\end{lemma}
(For the full computation of the $A_\infty$ structure of the Fukaya-Seidel category of $W$, see \cite{A}.)


Now, we consider the equivariant vector bundles on $V_i$'s. Recall from Section \ref{sec: orbibundle} that such bundles extend to the vanishing thimbles ($\Delta_i$'s) so that the equivariant structures are determined by characters in $(\ZZ/3)^\ast = \Hom (\ZZ /3, U(1))(\cong \ZZ /3)$. Let $\chi_i$ ($i=0,1,2$) be the character that sends the primitive generator of $\ZZ /3$ to $\rho^i \in U(1)$ ($\rho = e^{ 2 \pi i / 3}$). Then, the $\ZZ /3$-equivariant Fukaya-Seidel category has $9$ objects : $(V_i, \chi_j)$, $i, j \in \{0,1,2\}$. 

One can easily check that if $i_1 \neq i_2$, the $G$-invariant part of $\Hom ((V_{i_1}, \chi_{j_1}), (V_{i_2}, \chi_{j_2}))$ is of rank $1$ ( regardless of $j_1$ and $j_2$) since the $G$-action is free.
If $i_1=i_2$ but $\chi_{j_1}\neq \chi_{j_2}$, then  the $\Z/3$-action on $\CC (\cong \CC \cdot id_{V_i})$ is nontrivial as $\chi_1 (g)^{-1} \chi_2 (g) \neq 1$ for some $g \in \Z/3$. Therefore, there can not exist any $g$-invariant element other than zero. 
Hence we have
$$\Hom ((V_{i}, \chi_{j_1}), (V_{i}, \chi_{j_2})) \cong \CC^{\delta_{j_1, j_2}}$$

\subsection{Equivariant $A$-branes in $\CP^2$ and $B$-branes in its mirror}\label{AB}

\noindent{\bf (i)} In $\CP^2$,  there are three nontrivial $A$-branes, the Clifford torus $T$, together with
three different holonomies,  $(T,(1,1))$, $(T, (\rho, \rho))$ and $(T, (\rho^2, \rho^2))$, which are known to have
non-vanishing Floer homologies \cite{C}. In fact they are expected to generate the Fukaya category of $\CP^2$ which
has been announced by Abouzaid, Fukaya, Oh, Ohta and Ono \cite{AFOOO}.

The central fiber $T$ is preserved by the $\Z/3$-action. As the values of the (Lagrangian Floer) potential function are distinct for these three $A$-branes,
Floer homologies between them are not defined, but only Floer homologies for the same pairs of objects are defined, which turn out to be isomorphic to the
singular cohomology of $T$ (as modules).
We work with $\ZZ /2$-grading, and hence both the even and the odd degree components of their endomorphism spaces have rank $2$.
Let us consider $(T, (1,1))$ only, since the other cases are almost the same. We identify $HF(T,T)$ with the space of harmonic forms on $T$, which is generated by $d \theta_0$, $d \theta_1$, $d \theta_2$ (and their products as well) with the relation $d \theta_0 + d \theta_1 + d \theta_2 =0$. The induced action on $T$ sends $d \theta_i$ to $d \theta_{i+1}$ (with $i+3 =i$). One can easily check that $1$ and $d \theta_0 \wedge d \theta_1$ is invariant under this action. However, the $G$-invariant part of $HF(T,T)$ does not have any 1-forms since the only 1-form possibly $G$-invariant is $d \theta_0 + d \theta_1 + d \theta_2(=0)$.


\noindent{\bf (ii)} Likewise, among three critical value $3$, $3\rho$, and $3 \rho^2$ of the potential $W = x+y+ \frac{1}{xy}$, we only consider the category $D^b_{sing} (W^{-1} (3))$. $W^{-1} (3)$ has a unique singular point $(x,y) = (1,1)$, which is nodal. Let $\mathcal{F}$ be the skyscraper sheaf supported at this point. Note that $\mathcal{F}$ is $\ZZ /3$-invariant. We compute $R \Hom (\mathcal{F}, \mathcal{F})$ in the split-closure of $D^b_{sing} (W^{-1} (3))$. 

After the idempotent completion (taking the split-closure), it is enough to consider the skyscraper sheaf $\mathcal{F}=\CC_0$ at the origin in ${\rm Spec}\, A$ where $A:=\CC[x,y] / xy$. Recall from the previous subsection that the action on nearby regular fibers (locally identified with $T^\ast S^1$) are induced by $2 \pi / 3$-rotations on zero sections. Thus, the continuity of the action shows that the corresponding group action on ${\rm Spec}\, A$ (near the origin) is analytically equivalent to the one given by $(x,y) \mapsto (\rho x, \rho^{-1} y)$. (See Figure \ref{egDbsing}.)

\begin{figure}[h]
\begin{center}
\includegraphics[height=1.7in]{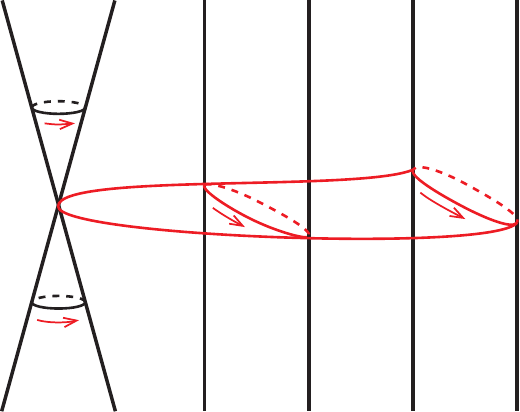}
\caption{$\ZZ / 3$-action on ${\rm Spec}\, A$}\label{egdbsing}
\end{center}
\end{figure}

Note that we have an exact triangle
$$ Ax \oplus Ay \to A \to \CC_0 \to [1]$$
consisting of $\Z /3$-invariant sheaves and that $A$ is a free module. Therefore, $\CC_0$ is isomorphic to $Ax \oplus Ay$ in the derived category of singularities. The even degree part of $R\Hom (Ax \oplus Ay, Ax \oplus Ay)$ is, then, given by
\begin{equation}\label{RHom0}
R\Hom^0 (Ax \oplus Ay, Ax \oplus Ay) = \Hom (Ax \oplus Ay, Ax \oplus Ay) \cong \CC^2.
\end{equation}
(Here, $\Hom (Ax \oplus Ay, Ax \oplus Ay)$ means the space of $A$-module homomorphisms from $Ax \oplus Ay$ to itself which do not factor through a free $A$-module.)

For $R\Hom^1 (Ax \oplus Ay, Ax \oplus Ay)=\Hom (Ax \oplus Ay, Ax \oplus Ay [1])$, we first compute $Ax [1]$ and $Ay [1]$. From the following exact triangle
$$ Ax \to A \to A/ Ax \to [1]$$
we see that $Ax[1] \cong A / Ax = \CC [y](=\mathcal{O} (y-{\rm axis}) )$. Similarly, $Ay [1] \cong \CC[x]$. Therefore,
\begin{equation}\label{RHom1}
R\Hom^1 (Ax \oplus Ay, Ax \oplus Ay)=\Hom (Ax \oplus Ay, \CC[x] \oplus \CC[y]).
\end{equation}
Any morphism in the right hand side should send $x$ to a constant multiple of $1 \in \CC[x]$, or it would factor through $A$. Likewise, the image of $y$ has to be a multiple of $1 \in \CC[y]$. Therefore, $R\Hom^1 (Ax \oplus Ay, Ax \oplus Ay) \cong \CC^2$.

Note that all sheaves involved here admit natural $\ZZ /3$-actions since they are defined by geometric terms preserved by $\ZZ/3$ ( e.g. $Ax$ is the ideal sheaf consisting of regular functions on ${\rm Spec}\,A$ which vanish on $y$-axis). These actions, of course, can be twisted further by tensoring irreducible representations of $\ZZ /3$, but we do not consider them here for simplicity. 

All elements in \eqref{RHom0} are clearly invariant under the induced action on \eqref{RHom0} so that $R\Hom^0_{\ZZ /3} (Ax \oplus Ay, Ax \oplus Ay) \cong \CC^2.$
However, $R\Hom^1_{\ZZ /3} (Ax \oplus Ay, Ax \oplus Ay) =0$ since the primitive generator of $\ZZ /3$ acts on $x \in Ax$ and $y \in Ay$ by multiplying $\rho$, while it keeps both $1 \in \CC[x]$ and $1\in \CC[y]$ invariant.

\bibliographystyle{amsalpha}

\begin{thebibliography}{}
\bibitem[Ab]{Ab} M. Abouzaid, D. Auroux, A. I. Efimov, L. Katzarkov, and D. Orlov,
{\em Homological mirror symmetry for punctured spheres,}
arXiv preprint arXiv:1103.4322 (2011).

\bibitem[AFOOO]{AFOOO} M. Abouzaid, K, Fukaya, Y.-G. Oh, H. Ohta, and K. Ono, 
{\em Quantum cohomology and split generation in Lagrangian Floer theory,}
In preparation.
\bibitem[ALR]{ALR} A. Adem, J. Leida and Y. Ruan,
{\em Orbifolds and Stringy Topology,}
 Cambridge Tracts in Mathematics, 171. Cambridge University Press, Cambridge, 2007.

\bibitem[AKO]{A} D. Auroux, L. Katzarkov and D. Orlov,
{\em Mirror symmetry for weighted projective planes and their noncommutative deformations,}
Ann. of Math. 167 (2008), no. 3, 867--943
\bibitem[Au]{Au} Denis Auroux,
{\em Mirror symmetry and T-duality in the complement of the anticanonical divisor}
J. G\"okova Geom. Topol. 1, 2007, p.51-91

\bibitem[B]{B} A. A. Beilinson,
{\em Coherent sheaves on $\mathbb{P}^n$ and problems of linear algebra,}
 Functional Analysis and its Applications 12.3 (1978): 214-216.
\bibitem[BL]{BL} J.Bernstein, V.Lunts, {\em Equivariant sheaves and functors,} Springer Lecture Notes 1578 (1994)
\bibitem[CL]{CL}K. Chan and N. C. Leung,
{\em Matrix factorizations from SYZ transformations,}
Advances in Geometric Analysis, Adv. Lect. Math. (ALM) 21, Int. Press, Somerville, MA, 2011.
\bibitem[CR1]{CR1}  W. Chen and Y. Ruan,
{\em A new cohomology theory of orbifold,}
Comm. Math. Phys. 248 (2004), 1 – 31. MR 2104605
\bibitem[CR2]{CR2}  W. Chen and Y. Ruan,
{\em Orbifold quantum cohomology,}
Preprint math.AG/0005198.
\bibitem[CR3]{CR3} W. Chen and Y. Ruan,
{\em Orbifold Gromov-Witten theory,}
Orbifolds in mathematics and physics. Contemp. Math., 310, Amer. Math. Soc., Providence, RI(2002): 25–85.
\bibitem[C1]{C} C.-H. Cho
{\em Holomorphic discs, spin structures and the Floer cohomology
of the Clifford torus,} IMRN (2004) no. 35 1803-1843
\bibitem[CH]{CH} C.-H. Cho and H. Hong,
{\em Orbifold Morse-Smale-Witten complex,}
preprint, arXiv:1103.5528
\bibitem[CH2]{CH2} C.-H. Cho and H. Hong,
{\em Floer cohomology for $G$-equivariant immersions,}
In preparation.
\bibitem[CHL]{CHL} C.-H. Cho, H. Hong and S. Lee
{\em Examples of matrix factorizations from SYZ,}
SIGMA 8 (2012), 053
\bibitem[CHS]{CHS} C.-H. Cho,  H. Hong, H.S. Shin
{\em On embeddings between orbifolds,}
To appear in Journal of KMS.
\bibitem[CO]{CO} C.-H. Cho and Y.-G. Oh,
{\em Floer Cohomology and Disc Instantons of Lagrangian Torus Fibers in Fano Toric Manifolds,}
Asian J. Math. 10 (2006), no. 4, 773-814.
\bibitem[CP]{CP} C.-H. Cho, M. Poddar
{\em Holomorphic orbidiscs and Lagrangian Floer cohomology of symplectic toric orbifolds,}
preprint, arXiv:1206.3994
\bibitem[E]{E} A. D. Elagin,
{\em Semiorthogonal decompositions of derived categories of equivariant coherent sheaves,} Izvestiya: Mathematics 73.5 (2009): 893.
\bibitem[Fl]{Fl} A. Floer, {\em Morse theory for Lagrangian
intersections,} J. Differ. Geom. 28 (1988), 513-547.
\bibitem[Fu1]{Fuk1} K. Fukaya, {\em Morse homotopy,
$A_\infty$-category and Floer homologies,} in Proceedings of GARC
Workshop on Geometry and Topology, ed by H. J. Kim, Seoul National
University, Korea 1993.
\bibitem[Fu2]{Fuk2} K. Fukaya
{\em Deformation theory, homological algebra, and mirror symmetry,}
Geometry 
and physics of branes (Como, 2001), 121--209, Ser. High Energy Phys. Cosmol. 
Gravit., IOP, Bristol, 2003
\bibitem[Fu3]{Fuk} K. Fukaya
{\em Floer homology and mirror symmetry II,}
Adv. Stud. in Pure Math.34 2002, 31 
- 127
\bibitem[FO]{FO}  K. Fukaya, K. Ono,
{\em Arnold conjecture and Gromov-Witten invariants, }
Topology,  Vol. 38, No. 5, 933-1048, 1999
\bibitem[FOOO]{FOOO} K. Fukaya, Y.-G. Oh, H. Ohta, and K. Ono,
{\em Lagrangian intersection Floer theory - anomaly and obstruction,}
 AMS/IP Studies in Advaneced Math. 46, International 
Press/Amer. Math. Soc. (2009).
\bibitem[FOOOT1]{FOOOT1} K. Fukaya, Y.-G. Oh, H. Ohta, and K. Ono,
{\em Lagrangian Floer theory on compact toric manifolds I,}
Duke Math. J. 151 (2010), 23-174.
\bibitem[FOOOT2]{FOOOT2} K. Fukaya, Y.-G. Oh, H. Ohta, and K. Ono,
{\em Lagrangian Floer theory on compact toric manifolds II: Bulk deformations,}
Selecta Math, New Series,
Volume 17, Number 3 (2011), 609-711
\bibitem[FOOO3]{FOOO3} K. Fukaya, Y.-G. Oh, H. Ohta, and K. Ono,
{\em Technical details on Kuranishi structure and virtual fundamental chain,}
Preprint arXiv:1209.4410
\bibitem[K]{K} M. Kontsevich,
{\em Homological algebra of mirror symmetry},
ICM-1994 proceedings, Z\"{u}rich, Birkh\"{a}user (1995).
\bibitem[KT]{KT}R. C. Kirby and L. R. Taylor
{\em Pin structures on Low-dimensional manifolds,}
LMS Lec. note series 151, Camb. Univ. Press, (1990) 177-242
\bibitem[LT]{LT} E. Lerman and S. Tolman,
{\em Hamiltonian Torus Actions on Symplectic Orbifolds and Toric Varieties,}
Trans. Amer. Math. Soc. 349 (1997) 4201-4230
\bibitem[M]{Mi} J. Milnor
{\em Spin structures on Manifolds,}
L'Enseignement Mathematique, 9  (1963)
\bibitem[O1]{Oh1} Y.-G. Oh,
{\em Floer cohomology of Lagrangian intersections and
pseudo-holomorphic discs I,} Comm. Pure and Appl. Math. 46 (1993),
949-994 addenda, ibid, 48 (1995), 1299-1302.

\bibitem[On]{On} K. Ono,
{\em Floer-Novikov cohomology and symplectic fixed points,}
J. Symplectic Geom., 3 (2005), 545–563.

\bibitem[PV]{PV} A. Polishchuk and A. Vaintrob
{\em Matrix factorizations and singularity categories for stacks,}
arXiv preprint arXiv:1011.4544 (2010).

\bibitem[Sa]{Sa} I. Satake,
{\em On a generalization of the Notion of Manifold,}
Proc. Nat. Acad. Sci. USA 42 (1956), 359-363.
\bibitem[Se1]{Se} P. Seidel
{\em Fukaya Categories and Picard-Lefschetz theory,}
EMS Zurich Lect. Adv. Math. 2008
\bibitem[Se2]{Se2} P. Seidel,
{\em Homological Mirror Symmetry for the Quartic surface,}
preprint. arXiv:0310414
\bibitem[Se3]{Se3} P. Seidel,
{\em Vanishing cycles and mutation,}
European Congress of Mathematics. Birkhauser Basel, 2001
\bibitem[Se4]{Se4} P. Seidel,
{\em More about vanishing cycles and mutations,}
Symplectic geometry and mirror symmetry (Seoul, 2000),
429-465, World Sci. Publ., River Edge, NJ, 2001
\bibitem[T]{T} Y. Takeuchi
{\em Waldhausen's Classification Theorem for Finitely Uniformizable 3-Orbifolds,}
Trans. AMS.  328, No. 1 (1991), 151-200
\bibitem[LM]{LM} H. B. Lawson, JR and M-L. Michelsohn,
{\em Spin Geomtery,}
Princeton Univ. Press. 38 1990 
\bibitem[W]{W} C. A. Weibel,
{\em An introduction to homological algebra,}
Cambridge studies in advanced mathematics 38.

 \end{thebibliography}

\end{document}